 \documentclass[final]{article}

\usepackage{amsmath,amsfonts,amsthm,amssymb,amscd,color,mathrsfs}
       \usepackage[notref]{showkeys}
      \usepackage[subnum]{cases}                   

\renewcommand{\Re}{\mathop{\rm Re}\nolimits}
\renewcommand{\Im}{\mathop{\rm Im}\nolimits}

\theoremstyle{plain} \newtheorem{theorem}{Theorem}[section]
\newtheorem{lemma}[theorem]{Lemma}
\newtheorem{proposition}[theorem]{Proposition}
\newtheorem{corollary}[theorem]{Corollary} \theoremstyle{definition}
\newtheorem{definition}[theorem]{Definition} \theoremstyle{remark}
\newtheorem{remark}[theorem]{Remark} 
 
\newcommand{\R}{{\mathbb R}}

\newcommand{\Z}{{\mathbb Z}}

\newcommand{\N}{{\mathbb N}}

\newcommand{\Sc}{{\mathcal S}}

\newcommand{\resto}{{\mathcal R}} \def\im{{\rm i}}

\newcommand{\C}{\mathbb{C}} 

\def \im{{\rm i}}
\def \rmi{{\rm i}}

\def\({\left(}
\def\){\right)}
\def\<{\left\langle}
\def\>{\right\rangle}

\numberwithin{equation}{section}

\begin{document}

  \title{ On   stabilization   of small solutions   in the nonlinear Dirac equation with a trapping potential }

 \author {Scipio Cuccagna and Mirko Tarulli}

 \maketitle
\begin{abstract}   We consider  a Dirac operator with short range potential
 and with eigenvalues. We add a nonlinear   term and we show that the
   small standing waves of the corresponding nonlinear Dirac equation
   (NLD) are attractors for small solutions of the NLD. This extends
   to the NLD results already known for the Nonlinear Schr\"odinger Equation (NLS).\end{abstract}

\vskip+0.5cm
{\bf Keywords:} Nonlinear Dirac equation, standing waves. \vskip+0.5cm

\section{Introduction}\label{sec:intr}
We consider
\begin{equation}\label{eq:NLS}
\left\{\begin{matrix}
\rmi u_t -H u +g(u\overline{u} )\beta u=0, \text{ with }  (t,x) \in \R \times \R^3 ,\\
u(0,x)=u_0(x),
\end{matrix}\right.
\end{equation}
where for   $\mathscr{M}>0$  and for a potential $V(x)$ we have  
\begin{equation}\label{Eq:op}
 H = D_\mathscr{M} +V, \,
\end{equation}
where   $D_\mathscr{M}=-\rmi \sum _{j=1}^3\alpha_j\partial_{x_j}+\mathscr{M}\beta$,
with for $j=1,2,3$,
\begin{equation*}\index{$\alpha_j$}\index{$\beta$}\index{$\sigma_j$}
\alpha _j=  \begin{pmatrix}  0 &
\sigma _j  \\
\sigma_j & 0
\end{pmatrix} \, , \, \beta =  \begin{pmatrix}  I _{\C^2} &
0 \\
0 & -I _{\C^2}
 \end{pmatrix}
\, , \,
\sigma _1=\begin{pmatrix}  0 &
1  \\
1 & 0
 \end{pmatrix} \, ,
\sigma _2= \begin{pmatrix}  0 &
\rmi   \\
-\rmi  & 0
 \end{pmatrix}  \, ,\,
\sigma _3=\begin{pmatrix}  1 &
0  \\
0 & -1
 \end{pmatrix} .
\end{equation*}
The unknown $u$ is   $\C^4$-valued. Given two vectors of $\C^4$,
$uv:=u\cdot v$ is the inner product in $\C^4$,
${v}^{\ast}$ is the complex conjugate, $u\cdot  {v}^{\ast}$ is the
hermitian product in $\C^4$, which we write as $uv^\ast =u\cdot
{v}^{\ast}$. We set
$ \overline{u}:=\beta {u}^{\ast}$,
so that $u\overline{u} =u\cdot \beta  {u}^{\ast}$\index{$u\overline{u} $}.

We introduce the Japanese bracket $\langle x \rangle := \sqrt{1+|x| ^2} $ and the spaces
defined by the following norms:
\begin{align} & L^{p ,s}(\R ^3, \C^4 ) \text{  defined with  }
  \| u \| _{L^{p,s}(\R ^3, \C^4 ) }  :=  \|   \langle x \rangle ^s   u \| _{L^p (\R ^3, \C^4)} ;   \nonumber \\
 &  H^k(\R ^3, \C^4  ) \text{  defined with  }
  \| u \| _{H^k(\R ^3, \C^4  ) }  :=  \|   \langle x \rangle ^k   \mathcal F (u) \| _{L^2 (\R ^3, \C^4   )},
 \text{   where  $\mathcal F$ is the classical Fourier}
 \nonumber\\
 & \text{transform (see for instance \cite{Du})};
   \nonumber\\& H ^{k,s}(\R ^3, \C^4  ) \text{  defined with  }
  \| u \| _{H^{k,s}(\R ^3, \C^4   ) }  :=  \|   \langle x \rangle ^s    {u} \| _{H^k(\R ^3, \C^4 )}  ;  \nonumber\\& \Sigma _k :=L^{2 ,k}(\R ^3, \C^4 )\cap   H^k(\R ^3, \C^4 )  \text{   with  }  \| u \|^2 _{\Sigma _k} = \| u \| _{H^{k }(\R ^3, \C^4  ) } ^2+ \| u \|^2 _{L^{2,k}(\R ^3, \C^4   ) }.\label{eq:Sigma}
 \end{align}
For  $\mathbf{f},\mathbf{g}\in L^2(\R ^3, \C ^4)$ consider   the bilinear map
\begin{equation} \label{eq:Inner Product}
\left \langle  \mathbf{f}   , \mathbf{g}   \right \rangle =
\int _{\R ^3}   \mathbf{f}(x)  \mathbf{g}(x)   dx  = \int _{\R ^3}   \mathbf{f}(x)  \cdot \mathbf{g}(x)   dx .
\end{equation}
We assume the following.
\begin{itemize}
  \item[(H1)]
$g(0)=0$, $g\in C^\infty(\R,\R)$.

\item[(H2)] $V\in \mathcal{S}(\R^3,S_4(\C))$ with $S_4(\C)$ the set of
self--adjoint $4\times 4$ matrices and $\mathcal{S}(\R^3, \mathbb{E})$ the space of Schwartz functions from $\R^3$  to $ \mathbb{E}$, with the latter a Banach space on $\C$.

\item[(H3)]
$\sigma_p(H)=\{e_1<e_2< e_3 \cdots < e_n \} \subset (-\mathscr{M},\mathscr{M})$.   Here we assume that all the eigenvalues have multiplicity 1.
 Each point  $ \tau =\pm \mathscr{M}$  is  neither an eigenvalue nor a resonance (that is, if $(D _{\mathscr{M}} +  V)u=\tau u$  with $u\in  C^\infty$ and     $|u(x)|\le C|x|^{-1} $ for a fixed $C$, then $u=0$).

\item[(H4)] There is an $N\in \N$ with $N > (\mathscr{M}+|e_1|) (\min \{ e_i -e_j : i>j \}) ^{-1}$
such that
if  $\mu \cdot \mathbf{e} :=
\mu _1 {e} _{ 1} +\dots +\mu _n {e} _{n}$ then
 \begin{align}
& \text{  $\mu\in \Z^n$  with
  $|\mu| \leq 4N +6$   } \Rightarrow |\mu \cdot \mathbf{e} |\neq \mathscr{M},   \label{h41}\\&  (\mu - \nu )\cdot \mathbf{e}=0 \text{ and }  |\mu |=|\nu |\le 2N+3 \Rightarrow \mu =\nu . \label{h42}
\end{align}

\item[(H5)]

Consider the set $M_{min}$ defined in \eqref{eq:setMhat}  and for any $(\mu , \nu )\in M_{min}$    the function $G _{\mu   \nu }(x)$ (see the proof of Lemma \ref{lem:conditional4.2} or also later in the introduction the effective hamiltonian),  $\widehat{G }_{\mu \nu} (\xi
)$ the  distorted Fourier transform  associated to $ H$ (see \eqref{eq:restrDir} and Sect. \ref{app:Pformula} for more details)    of
   $G _{\mu   \nu }(x)$ and  the sphere  $S_{\mu \nu}=\{ \xi \in \R^3: |\xi |^2 +\mathscr M^2=|(\nu - \mu )\cdot \mathbf{e} |^2 \}$.
   Then we assume that  for any
   $(\mu , \nu )\in M_{min}$ the
   restriction of  $\widehat{G }_{\mu \nu} $ on the sphere $S_{\mu \nu}$ is   $ \widehat{G }_{\mu \nu}| _{S_{\mu \nu}}\neq 0$.

\end{itemize}

To each  $e_j$   we associate an eigenfunction $\phi _j$. We choose them such that  $ \Re\langle  \phi _j, {\phi }_k^{*}\rangle  =\delta _{jk} $.
To each  $\phi _j$ we associate   nonlinear bound states.

\begin{proposition}[Bound states]\label{prop:bddst}
Fix $j\in \{ 1,\cdots,n\}$. Then
$\exists a _0>0$ such that\ $\forall z_j \in     B_\C ( 0, a _0 )$, there is a unique  $  Q_{jz_j } \in \mathcal{S}(\R^3, \C ^4) := \cap _{t\ge 0}\Sigma _t (\R^3, \C ^4)$, such that
\begin{equation}\label{eq:sp}
\begin{aligned}
&H Q_{jz_j } + g(Q_{jz_j }\overline{Q}_{jz_j } )\beta Q_{jz_j }= E_{jz_j }Q_{jz_j },\\&
Q_{jz_j }=  z_j\phi_j + q_{j z_j}, \ \langle q_{j z_j}, {\phi}_j^{*}\rangle =0,
\end{aligned}
\end{equation}
and such that we have for any $r\in \N$:
 \begin{itemize}
\item[(1)]    $(q_{jz_j },E_{jz_j }) \in C^\infty ( B_\C ( 0, a _0 ), \Sigma _r\times \R )$;   we have $q_{jz_j } =  z_j \widehat{q}_{j  } (|z_j|^2)$ , with
$ \widehat{q}_{j  } (t^2 ) =t ^2\widetilde{q}_{j }(t^2)$,   $\widetilde{q}_{j } (t ) \in C^\infty (   ( - {a _0 }^{2}, {a _0 }^{2}), \Sigma _r (\R ^3, \C ^4 ) )$  and  $E_{jz_j }  =E_{j } (|z_j|^2)$ with $E_{j } (t ) \in C^\infty (   ( - {a _0 }^{2}, {a _0 }^{2}),   \R  )$;
\item[(2)]    $\exists$ $C >0$ such that
$\|q_{jz_j }\|_{\Sigma _r} \leq C |z_j|^3$, $|E_{jz_j }-e_j|<C | z_j|^2$.

\end{itemize}

\end{proposition}
 For the   $\Sigma_r$  see  \eqref{eq:Sigma}. The only non-elementary
  point in   Prop. \ref{prop:bddst} is the independence of $a_0$ with respect of $r$  (which
  strictly speaking is not necessary in this paper),
  which can be proved with routine arguments as in the      Appendix of \cite{CM1}.

\begin{definition}\label{def:contsp}
Let    $b _0 >0$ be sufficiently small so that  for $z:=(z_1,...,z_n)$ with
$z \in B_{\C^{n }}(0,b _0) $ then  $Q_{j z_j}$  exists for all   $j\in\{ 1,\cdots,n\}$. Set
$z_j = z_{j R}+ \im z_{j I},$ for $z_{j R}, z_{j I}\in \R$ and
$D_{jA} :=\frac{\partial}{\partial z_{j A}}, $ for $A=R,I$.  Then we set
\begin{equation}\label{eq:contsp}
\begin{aligned}
\mathcal{H}_c[z]&:=\left\{\eta\in L^2(\R ^3, \C ^4) ;\  \Re \<\im \,  {\eta}^*,D _{jR} Q_{j z_j}\>=\Re \<\im \, {\eta}^* ,D_{jI}Q_{j z_j}\>=0 \text{ for all $j$}\right\} ,
\end{aligned}
\end{equation}
where in  particular we have
\begin{equation}\label{eq:contsp1}
\begin{aligned}
\mathcal{H}_c[0]  =\left\{\eta\in  L^2(\R ^3, \C ^4) ;\    \<   {\eta}^*,\phi _j\>  =0\text{ for all $j$}\right\}.
\end{aligned}
\end{equation}
We denote by $P_c$ the orthogonal projection of $ L^2(\R ^3, \C ^4)$ onto $\mathcal{H}_c[0]$.
\end{definition}

We will prove the following theorem.

\begin{theorem}\label{thm:small en}   Assume $(\mathrm{H1})$--$(\mathrm{H5})$.  Then there exist  $\epsilon _0 >0$ and $C>0$ such that  for $\epsilon =\| u (0)\| _{H^4}<\epsilon _0  $ then  the  solution  $u(t)$ of  \eqref{eq:NLS} exists for all times and  can be written uniquely  for all times as
 \begin{equation}\label{eq:small en1}
\begin{aligned}&    u(t)=\sum_{j=1}^nQ_{j z_j(t)}+\eta (t), \text{ with $\eta (t) \in
\mathcal{H}_c[z(t)]$},
\end{aligned}
\end{equation}
such that  there exist  a unique $j_0$, a
$\rho  _+\in [0,\infty )^n$ with $\rho_{+j}=0$ for $j\neq j_0$
such that
  $| \rho  _+ | \le C  \| u (0)\| _{H^4}  $, an $\eta _+\in H^4 (\R ^3, \C ^4)$
with $\|  \eta _+\| _{L^\infty}\le C  \| u (0)\| _{H^4}$
 and such that  we have

\begin{equation}\label{eq:small en3}
\begin{aligned}&      \lim_{t\to +\infty}\| \eta (t )-
e^{-\im tD_{\mathscr{M}} }\eta  _+     \|_{H^4 (\R ^3, \C ^4)}=0 , \\&
 \lim_{t\to +\infty} |z_j(t)|  =\rho_{+j}  .
\end{aligned}
\end{equation}
 Furthermore we have $\eta = \widetilde{\eta} +A(t,x) $  such
    that,
    \begin{equation}\label{eq:numbers1}
\begin{aligned}&    \text{for preassigned $p_0>2$, for all $p\ge p_0$ and for } \frac{2}{p}=\frac{3}{2}\left (1-\frac{2}{ q} \right),
\end{aligned}
\end{equation}
  we have
    \begin{equation}\label{eq:small en2}
\begin{aligned}&     \| z \| _{L^\infty _t( \mathbb{R}_+ )}+ \|
\widetilde{ {\eta}}   \| _{L^p_t( [0,\infty ),B^{ 4 -\frac{2}{p}} _{q,2} (\R ^3, \C ^4)} \le C \| u (0)\| _{H^4(\R ^3, \C ^4)}   \ , \\& \| \dot z _j +\im  e_{ j }z_j \|  _{L^\infty _t( \mathbb{R}_+ ) } \le C  \| u (0)\| _{H^4(\R ^3, \C ^4)}^2\
\end{aligned}
\end{equation}
  (for the  Besov spaces  $B^{ k} _{p,q}   $  see Sect. \ref{section:set up})
and such that $A(t,\cdot )\in \Sigma _4 $  for all $t\ge 0$, with
\begin{equation}\label{eq:small en4}
\begin{aligned}&      \lim_{t\to +\infty}\| A(t,\cdot )   \|_{\Sigma _4 }=0 .
\end{aligned}
\end{equation}

\end{theorem}

\begin{remark}\label{rem:wp}
In $H^{4} (\R ^3, \C ^4)$  the functional $u\to
g(u\overline{u}) \beta  u$ is locally Lipschitz and  \eqref{eq:NLS} is locally
well posed, see pp. 293--294 in volume III  \cite{taylor}.
\end{remark}
\begin{remark}\label{rem:sobh4} There is no attempt here to get a sharp result
with a minimum amount of regularity on the initial datum  $u_0 $.
The need of $H^{4}  $  comes up \eqref{eq:H^4}.
\end{remark}

 Theorem \ref{thm:small en} states that small solutions of the NLD  are    asymptotically equal to exactly one standing wave (possibly the vacuum)  up to radiation which scatters and  up to a phase factor.
 As we explain below, this happens  because of  the natural tendency of radiation
 to scatter because of linear scattering, and of energy to leak out of
 all discrete modes, except at most for one,  because of
    nonlinear interaction with the continuous modes which    produces  a form of    friction   on    most  discrete modes. Theorem \ref{thm:small en} does not say if
some of the standing waves are stable or unstable (this remains an open problem).

The asymptotic behavior of   solutions which start  close to equilibria represents
 a fundamental problem for nonlinear systems. Here we are interested on   hamiltonian systems  which are perturbations of linear systems  admitting
 a mixture of discrete and continuous components.  Decay phenomena imposed
 by the perturbation on the discrete modes and decay of \textit{metastable} states  are a classical topic of physical relevance in the study
 of radiation matter interaction, \cite{Cohen}. There is a substantial literature focused
 on the case of linear systems, see \cite{Jensen}  for a recent reference
 and therein for some of the literature.

 The NLS has been explored in papers such as  \cite{RW},    \cite{SW1}--\cite{SW4}, \cite{PiW},   \cite{TY1}--\cite{TY4}, \cite{GNT},
\cite{GP}. An analogue of Theorem \ref{thm:small en} for the NLS is proved in \cite{CM1}.
The case of the nonlinear Klein Gordon equation (NLKG), in the context of real valued
solutions, where all small solutions scatter to 0, has been initiated in \cite{SW4} in special case  and to a large degree solved in a general way in \cite{bambusicuccagna}. For complex valued solutions of the NLKG see \cite{CMT}.

In this paper we focus on the NLD.   There are
 many papers   on    global well posedness and dispersion   for solutions belonging to much larger spaces
 than what considered here, for example see  \cite{BejHerr, Candy,Cacciafesta, Dias,Escobedo,MacNakOz,Zhang} and references therein. However  they  do not address
 the asymptotic behavior in the presence of discrete modes. For a survey on the 1D case  we refer also to \cite{Pelinovsky}.
The essential indefiniteness of the Dirac operators \eqref{Eq:op} forbids
the use of  the stability theory of standing waves developed in  \cite{CL,W1,GSS1}, which involves conditionally positive Lyapunov functionals.
See \cite{StraussVazquez} for an attempt to apply this theory to the Dirac equation, in combination
 to numerical computations. A successful use of rather elaborate Lyapunov functionals
   in the special case of an integrable NLD is in \cite{pel1}.
 Since   energy methods cannot be used to prove stability
then another option is to use linear dispersion.
This is a problem arising also in other settings, see for example
\cite{pego}.
Thus  on a strictly technical standpoint    the NLD is a rather interesting class of hamiltonian systems, especially because it has been explored
much less than systems like the NLS.

We turn  to the specific problem considered in this paper.
In the case of the  linear Dirac equation $\im u_t=Hu$
we know that  there are invariant complex lines formed by standing waves
in correspondence to the eigenstates of $H$.  Proposition
 \ref{prop:bddst} describes the well known fact that   the NLD \eqref{eq:NLS},    at least at small energies, has    invariant  disks  formed by standing waves. The question arises then on what should be the behavior of other  initially small
 solutions of the NLD.  In the case of the NLS the answer given in \cite{CM1}  is that
 the union of these invariant disks is an attractor for all small solutions.
 This is by no means obvious in view of the fact that  in the  linear equation   the situation is different. In fact in the linear case   discrete modes  and the continuous part
 are all independent from each other  so that solutions
 of the linear equation contain a quasi periodic component. There are examples of
 nonlinear systems where these linear like  patterns persist.  Notably   the   NLS
 where $x\in \Z$,
  which has
 small quasi--periodic solutions, see for example in \cite{Maeda} and references therein.
 There are also equations in $\R ^n$ which admit families of  quasiperiodic solutions containing
 solitary waves as a special case, an example being exactly the \eqref{eq:NLS}  with potential
 $V=0$ in \eqref{Eq:op} where solitary waves are special cases of more general quasiperiodic
 solutions. This is related to the existence of special   eigenvalues
 of the linearization uncoupled to radiation, as the eigenvalues in claim 4 Lemma 6.1 in \cite{com}.
 Notice also that the stabilization  displayed in
  \cite{CM1} is a very slow phenomenon, non-detectable easily numerically:
  see \cite{Kevrekidis} for comments on  a somewhat different but related context.
  So results such as those in \cite{CM1}, from the earlier papers
 \cite{NPT, SW4,TY1,TY2,TY3,SW3} which treat special cases,
 to  the solutions in generic setting contained in
  \cite{bambusicuccagna,CM1}, are far from obvious.

Here we transpose partially to the NLD the  result proved for the NLS in \cite{CM1}.   No particular deep new insight is needed since the proof for the NLD is similar to that for the NLS, with some
 difference related to the dispersion theory of radiation. However,
 results such as Theorem \ref{thm:small en} are important  given that the literature
 gives a very fragmentary picture of the asymptotic analysis of
 solutions of    the  NLD.

  While prior to   \cite{CM1} there had been a large body of work for the NLS, very little has been written   for the NLD  about our problem, that is the analysis
  of the NLD for with a mixture of discrete and continuous components.
  It is worth comparing the known  results in the literature with Theorem \ref{thm:small en}.

     One study is  \cite{boussaid}.  \cite{boussaid} contains a number of useful results on the
     dispersion for the group $e^{\im tH}$ which are  used here.  In terms of analysis of small solutions of the  NLD with a linear potential, \cite{boussaid}   considers   the case of  $\sigma_p(H)=\{e_1<e_2  \} $  with $e_2-e_1<\min ( \mathscr{M} -e_1, e_1+\mathscr{M} )$ and then proves the existence of a hypersurface of initial data perpendicular  at the origin to
the eigenspace of $e_2$, and whose corresponding solutions  of \eqref{eq:NLS} converge
to the 1st family of standing waves $Q _{1z_1}$.

    Another study  is \cite{PelinovskyStefanov}.  \cite{PelinovskyStefanov}  proves
the asymptotic stability of the standing waves $Q _{1z_1}$ when $\sigma_p(H)=\{e_1   \} $,
that is an analogue of \cite{SW2,PiW}.    \cite{PelinovskyStefanov}, like \cite{SW2,PiW}, does not examine a whole neighborhood of 0
but  proves that if an initial datum starts very close to a $Q _{1z_1}$ then, in a sense similar to Theorem \ref{thm:small en},  it will converge to a nearby $Q _{1z_1}$. Notice that  while here we don't prove that the $Q _{1z_1}$
are stable, under the hypothesis $\sigma_p(H)=\{e_1   \} $ it can be
proved to be stable also in our 3 D set up. In fact this is much simpler to prove than
what we do here.

  In \cite{boussaidcuccagna} is  discussed
the asymptotic stability of standing waves of  the NLD in a different context, which is somewhat closer in spirit to Theorem \ref{thm:small en} since the emphasis is on the case when many
discrete modes are present. For another result, see also \cite{Comech}.

   Here we give   more comprehensive conclusions than in \cite{boussaid,PelinovskyStefanov}
   because we treat all small solutions of \eqref{eq:NLS}. Maybe we could have  analyzed
   the stability of specific standing waves, applying the theory in \cite{boussaidcuccagna}
   which provides some tools to characterize standing waves  of the \eqref{eq:NLS} through an
   analysis of the linearization of \eqref{eq:NLS} at the standing wave,
   but we don't do this here. So  unfortunately we do not   identify
   stable standing waves and we don't produce a criterion to define
   ground states for   \eqref{eq:NLS}. Recall that in \cite{CM1}  it is proved for the NLS that while     the $Q_{1z_1}$  are stable (which was well known since \cite{RW}),    the  $Q_{jz_j}$  for $j\ge 2$  are unstable. No similar analysis
   is done here.

After adding to the Dirac equation a nonlinearity such as  in \eqref{eq:NLS} as we have already mentioned    the lines   of
standing waves    persist,  only in  the form of  topological  disks.
  Here we briefly explain how the asymptotics claimed in Theorem \ref{thm:small en} comes about. The mechanism is the same of the NLS and to be seen requires the identification
  of an \textit{effective} hamiltonian.   In an appropriate system of coordinates, this turns out to be,  heuristically,  of the form
\begin{equation*}       \begin{aligned} &
  \mathcal{H} (z,\eta )= \sum _{j=1}^n
	e_j |z_j|^2+  \langle   H  \eta,  { \eta}^* \rangle  +   \sum  _{| (\mu -\nu )\cdot \mathbf{e}| > \mathscr{M} } ({z}^{\mu}   \overline{{z}}^{\nu}  \langle
  G_{\mu \nu } , \eta \rangle  +  \overline{{z}}^{\mu}  {{z}}^{\nu}  \langle
  G_{\mu \nu }^* , \eta ^*\rangle ), \text{ with $\mathbf{e}=(e_1,...., e_n),$ } \end{aligned}
\end{equation*}
where the 2nd summation is on an appropriate finite set of multi--indexes.
The coordinates $(z,\eta)$, with $z=(z_1,...., z_n)$, representing the discrete modes and $\eta
\in \mathcal{H}_c[0]$ representing the radiation, are canonical (or Darboux) coordinates (but not so the  initial  coordinates in Lemma \ref{lem:systcoo}).  This in particular means that  the equation for $\eta$ is
   \begin{equation}\label{intstr0}
  \im \dot \eta =H\eta + \sum  \overline{{z}}^{\mu}  {{z}}^{\nu}
  G_{\mu \nu }^*.
  \end{equation} Then, succinctly,
  \begin{equation}\label{intstr1}
   \| \eta \| _{\text{Strichartz}}\le \epsilon +\sum \|   {z} ^{\mu+\nu} \| _{L^2 },
  \end{equation}
  for appropriate Strichartz norms of $\eta$ which are proven in  Sect.
  \ref{subsec:prop}    using \cite{boussaid,Boussaid2,boussaidcuccagna}.
There are various other  sets of Strichartz like estimates for Dirac potentials in the literature, see for example \cite{Fanelli2,Cacciafesta,Fanelli1}, but they do not apply to our case since they involve Dirac  operators without eigenvalues.
The equations for the discrete modes are, heuristically,
 \begin{equation} \label{intstr11}  \begin{aligned} &
\im \dot z _j
=  e_j z_j   +   \sum _{| (\mu -\nu )\cdot \mathbf{e}| > \mathscr{M}  }     \nu _j  \frac{z  ^{\mu  }
 \overline{ {z}}^ { {\nu}   } }{\overline{z}_j}
\langle  \eta   ,
  {G}_{\mu \nu }  \rangle  +   \sum _{ | (\mu -\nu )\cdot \mathbf{e}| > \mathscr{M}}         \mu _j  \frac{z  ^{\nu   }
 \overline{ z}^ { {\mu}   } }{\overline{z}_j}
\langle  \eta ^*  ,
  {G}^*_{\mu  \nu  }  \rangle   .
\end{aligned}  \end{equation}
By an argument introduced in \cite{BP2,SW3} we write  \begin{equation} \label{intstr12}   \begin{aligned} &
 \eta   =  g-Y,  \text{  where } Y:=
  \sum _{| (\alpha  -\beta  )\cdot \mathbf{e}| > \mathscr{M}   }   \overline{{z}}^\alpha  {{z}}^\beta   R_H^{+}( {\textbf{e} }     \cdot (\beta  -
\alpha  ))
   {G}_{\alpha \beta  }  ^*.
\end{aligned}\end{equation}
We observe that \eqref{intstr12} is a decomposition of $\eta$ into $-Y$, which is the
part of $\eta$ that mostly affects the $z$'s, and $g$ which is small. Substituting in \eqref{intstr11} and ignoring   $g$ we get an autonomous system in $z$. Ignoring smaller terms, we have \begin{equation}\label{fgrintr}  \begin{aligned} &
     \frac{d}{dt}\sum _j  \  |  z _j| ^2
 =
     -2\pi \sum _{ \substack{  | (\mu -\nu )\cdot \mathbf{e}| > \mathscr{M} }} |  z ^{\mu }
 \overline{ {z }}^ { {\nu}}|^2 \langle \delta ({H-(\nu-\mu)\cdot \mathbf{e}})   G^*_{\mu \nu}  ,
    G_{\mu \nu}   \rangle       .
\end{aligned}\end{equation}
The r.h.s. is negative, see Lemma   \ref{plemeljDir} in the Appendix.
Notice that  in \cite{BP2,SW4}  the structure of the r.h.s. is as clear as \eqref{fgrintr}
only under very restrictive hypotheses on the discrete spectrum.

\noindent
We assume that for an appropriate set of pairs $(\mu ,\nu)$  in \eqref{fgrintr} we have  $\langle \delta ({H-(\nu-\mu)\cdot \mathbf{e}})   G^*_{\mu \nu}  ,
    G_{\mu \nu}   \rangle   >0$. This is   the content of hypothesis (H5)
    and is an exact analogue of  inequality (1.8), which is   the main
    hypotheses  in \cite{SW4}.
    Then   integrating \eqref{fgrintr}  we obtain
\begin{equation}  \label{intstr2}\begin{aligned} &
       \sum _j    |  z _j(t)| ^2 + \sum _{ \substack{  | (\mu -\nu )\cdot \mathbf{e}| > \mathscr{M} }}\int _0^t |  z ^{\mu }
 \overline{ {z }}^ { {\nu}}|^2 \le C \sum _j    |  z _j(0)| ^2\le C \epsilon ^2,
\end{aligned}\end{equation}
for a fixed $C$. This allows to close \eqref{intstr1} and to prove that $\eta$ scatters.
From  \eqref{intstr2} and the fact that the $\dot z (t)$ is uniformly bounded, we get that
the $\lim _{t\to +\infty }z ^{\mu+{\nu}} (t) =0$  for the $(\mu , \nu)$ in \eqref{intstr2}.
This allows to conclude that $ \lim_{t\to +\infty}  z_j(t)   =0$ for all except for at most one $j$, because the pairs of  multi--indexes $(\mu , \nu)$ involve cases such as $ z^{\mu _j}_j
\overline{z}^{\nu _k}_k$ for any pair $j\neq k$ (while cases of the form
$ z^{\mu _j}_j
\overline{z}^{\nu _j}_j$ are not allowed).

What has happened is that when we have substituted   \eqref{intstr12}  inside \eqref{intstr11} and ignored $g$, we got a system on the $z$'s which has some degree of friction,
ultimately due to the coupling of the discrete modes with radiation. This can be proved
thanks to the square structure of the r.h.s. of \eqref{fgrintr} where what is crucial
is the fact that each $G_{\mu \nu}$  is paired with its complex conjugate $G_{\mu \nu}^*$.
This comes about and can  be seen transparently because we have a hamiltonian system in Darboux coordinates,
which tells us that the $G_{\mu \nu}^*$'s in the r.h.s. of \eqref{intstr0}
are the same of the coefficients in the \eqref{intstr11}.

  To be able to implement the above intuition gets some work
  because, as we mentioned, the most natural coordinates are not  Darboux  coordinates,
  that is the symplectic form is complicated when expressed in these coordinates (this would
  make the search of the  effective hamiltonian very difficult in these initial coordinates) and the equations are not as simple as $\im \dot z_j= \partial _{\overline{z}_j}\mathcal{H}$ and $\im \dot \eta = \partial _{\eta ^*}\mathcal{H}$.

  There are various papers, such as \cite{Cu0,Cu2}, that discuss
  how to first produce Darboux  coordinates and how to find the effective hamiltonian
  in an appropriate way that guarantees that the system remains a semilinear Dirac equation.
  In fact this part of the proof is here the same of \cite{CM1}, since in terms of the hamiltonian
  formalism,
  the NLS and the NLD are   the same. So   Darboux  coordinates and search
  of the effective hamiltonian by means of Birkhoff normal forms are accomplished in \cite{CM1}.

The only  part of the proof   where there is difference between NLD and NLS is in the study
of dispersion.  The NLD requires its own set of technical machinery
about  linear dispersion theory, Strichartz and smoothing estimates,
which is not the same of the NLS and which is somewhat  tricker and less  well understood.
Nonetheless, all the linear theory we need here has been already developed
in the literature, in particular  in \cite{boussaid,Boussaid2,boussaidcuccagna}.

Since the main ideas on the hamiltonian structure
and coordinate changes come from \cite{CM1}, we will refer to \cite{CM1} for an extensive discussion of the main issues.   The rest of the proof consists in proving dispersion of the continuous component of the solution (and here we use
the technology of Dirac operators in  \cite{boussaid,Boussaid2}), i.e. \eqref{intstr1},
 and the so called Nonlinear Fermi Golden Rule,   i.e. \eqref{fgrintr}.  This latter part of the proof is similar
to \cite{CM1}, but requires   some   modification in the spirit of \cite{Cu3} because of
the possible presence of pairs of eigenvalues with different signs.

Here as in \cite{CM1}   we do not prove, as   done in \cite{bambusicuccagna}  in an easier setting, that hypothesis (H5) holds for generic pairs $(V, g(u\overline{u}))$. While in
\cite{CM1} what was missing to repeat the argument in  \cite{bambusicuccagna} was a meaningful mass term, here we have mass $\mathscr{M}$ but the dependence of the   linear operator $H$ on $\mathscr{M}$ requires new ideas.

\section{Further notation and   coordinates}\label{section:set up}

\subsection{Notation}
\label{subsec:notation}
\noindent
For~$k\in\R$ and~$1\leq p,q\leq\infty$, the Besov
space~$B^k_{p,q}(\R^3,\C^d)$\index{$B^k_{p,q}$} is the space of all   tempered
distributions $f\in {\mathcal S}'(\R^3,\C^d)$  such that
\[
   \|f\|_{B^k_{p,q}}=  (\sum_{j\in\N}2^{jkq} \|\varphi_j *
f\|_p^q )^{\frac{1}{q}}<+\infty,
\]
with~$\mathcal F (\varphi) \in {\mathcal C}^\infty_0(\R^n\setminus
\left\{0\right\})$ such that~$\sum_{j\in
\Z}\mathcal F (\varphi)(2^{-j}\xi)=1$ for all~$\xi
\in\R^3\setminus\left\{0\right\}$,
$\mathcal F (\varphi_j)(\xi)=\mathcal F (\varphi)(2^{-j} \xi)$ for all
$j\in\N^*$ and for all~$\xi \in\R^3$, and
$\mathcal F (\varphi_0)=1-\sum_{j\in \N^*}\mathcal F (\varphi_j)$. It is
endowed with the norm $   \|f\|_{B^k_{p,q}}$.

\begin{itemize}
\item We denote by $\N =\{ 1,2,...\}$  the set of natural numbers  and  set $\N_0= \N\cup \{ 0\}  $.

\item Given a Banach space $X$, $v\in X$ and $\delta>0$ we set
$$
B_X(v,\delta):=\{ x\in X\ |\ \|v-x\|_X<\delta\}.
$$

\item
We denote $z=( z_1,\dots ,z_n)$, $|z|:=\sqrt{\sum_{j=1}^n|z_j|^2}$.

\item We set $\partial _{ l}:=\partial _{z_l}$  and  $\partial _{\overline{l} }:=\partial _{\overline{z}_l}$.
Here as customary   $\partial _{z_l}  = \frac 12 (D _{lR}-\im  D _{lI} )$ and    $\partial _{\overline{z}_l}  = \frac 12 (D _{lR}+\im  D _{lI} )$.

\item Occasionally we use a single  index $\ell =j, \overline{j}$. To define $\overline{\ell}$  we use the convention $\overline{\overline{j}}=j$.
We will also write   $z_{\overline{j}}=\overline{z}_j$.

\item  We will consider vectors $z=(z_1, ..., z_n)\in \C^n$ and  for  vectors $\mu , \nu \in (\N \cup \{ 0 \} ) ^{n}$ we set $ z^\mu \overline{z}^\nu  := z_1^{\mu _1} ...z_n^{\mu _n}\overline{z}_1^{\nu _1} ...\overline{z}_n^{\nu _n}$.  We will set $|\mu |=\sum _j \mu _j$.

\item   We have  $dz_j  =dz_{jR}+\im dz_{jI}$,    $d\overline{z}_j  =dz_{jR}-\im dz_{jI}$.

\item We consider the vector $ \mathbf{e}=(e_1,...., e_n)$ whose entries are the eigenvalues of $H$.

\end{itemize}

\begin{remark}\label{rem:att}  We draw the attention of the reader to the fact that
the complex conjugate of $v\in \C^4$  is $v^*$ with $\overline{v}=\beta v^*$  while for $\zeta \in \C$ the complex conjugate is $\overline{\zeta}$ which is more convenient notation  in some
later formulas
than writing  $\zeta ^*$.

\end{remark}

\subsection{Coordinates}
\label{subsec:coord}

The first thing we need is the following well standard ansatz,  see for example in  Lemma 2.6
\cite{CM1}.

\begin{lemma}\label{lem:decomposition}
There exists $c _0 >0$ such that  there exists a $C>0$ such that for all $u \in H^1$   with $\|u\|_{H^1}<c  _0 $, there exists a unique  pair $(z,\Theta )\in   \C^{n }\times  ( H^1 \cap H_{c}[z])$
such that
\begin{equation}\label{eq:decomposition1}
\begin{aligned} &
u=\sum_{j=1}^nQ_{j z_j}+\Theta   \text{ with }  |z |+\|\Theta \|_{H^4}\le C \|u\|_{H^4} .
\end{aligned}
\end{equation}
 Finally,
 the map  $u \to (z,\Theta )$  is $C^\infty (B _{H^4}(0, c _0 ),
 \C ^n \times  H^4 )$   and  satisfies  the   gauge property  \begin{equation}\label{eq:decomposition3}
\begin{aligned} &
 \text{$z(e^{\im \vartheta} u)=e^{\im \vartheta} z( u)
$ and    $\Theta (e^{\im \vartheta} u)=e^{\im \vartheta} \Theta ( u).$ }
\end{aligned}
\end{equation}

\end{lemma}  \qed

We now recall from \cite{CM1} the following definitions.
\begin{definition}\label{def:comb0} Given $z\in \C^n$, we denote by $ \widehat{Z} $ the vector with entries $(  z_{i}  \overline{{z}}_j)$ with $i,j\in [1,n]$   ordered in lexicographic order (that is we write $z_{i}  \overline{{z}}_j$  before $z_{i'}  \overline{{z}}_{j'}$
if either $i<j$ or, when $i=j$, if $j<j'$). We denote by $\textbf{Z}$ the vector     with entries $(  z_{i}  \overline{{z}}_j)$ with $i,j\in [1,n]$   ordered in lexicographic order      but only with pairs of  indexes  with $i\neq j$.  Here  $\textbf{Z}\in   L$ with $L$ the subspace of $ \C ^{n_0} =\{  (a_{i,j} ) _{i, j =1, ..., n} : i\neq j   \}   $ where $n_0=n (n-1),$ with $(a_{i,j} ) \in  L $ iff $a_{i,j} = \overline{a}_{j,i}$ for all $i,j$.
 For  a multi index $\textbf{m}=\{  {m}_{ij}\in \N _0 :i\neq j  \}$ we set $\textbf{Z}^{\textbf{m}}=\prod  (z_{i}  \overline{{z}}_j) ^{ {m}_{ij}} $  and $ |\mathbf{m}| :=\sum _{i,j}m  _{ij} $.

\end{definition}

\begin{definition}\label{def:setM}   Consider the set of multiindexes  $\textbf{m}$
as in Definition  \ref{def:comb0}. Consider for any $k\in \{ 1, ..., n  \}$ the set
\begin{equation} \label{eq:setM0} \begin{aligned} &  \mathcal{M}_k (r)=
\{   \textbf{m} :   \left |\sum _{i=1} ^{n} \sum _{j=1} ^{n} m  _{ij} (e_i -e_j) - e_k\right |>\mathscr{M}   \text{ and  }  |\mathbf{m}|   \le r \},  \\& \mathcal{M}_0 (r)=
\{   \textbf{m} :     \sum _{i=1} ^{n} \sum _{j=1} ^{n} m  _{ij} (e_i -e_j) =0   \text{ and  }  |\mathbf{m}|   \le r \}  .
\end{aligned}
\end{equation}
Set now
\begin{equation} \label{eq:setM1} \begin{aligned} &   {M}_k(r)=
\{    (\mu , \nu )\in \N _0 ^{n}  \times  \N _0 ^{n} :   \exists \textbf{m}   \in \mathcal{M}_k (r) \text{ such that }   z^{\mu} \overline{z}^{\nu} = \overline{z}_k   \mathbf{Z} ^{\textbf{m} }  \}  , \\& M(r)= \cup  _{k=1} ^{n} {M}_k(r).
\end{aligned}
\end{equation}
We also set    $M= M (2N+4)$ and
\begin{equation} \label{eq:setMhat} \begin{aligned} &   M_{min}=
\{    (\mu , \nu )\in M :      (\alpha  , \beta  )\in M  \text{ with } \alpha _j\le \mu _j \text{ and } \beta _j\le \nu _j \ \forall \ j \Rightarrow  (\alpha  , \beta  ) = (\mu , \nu )  \}  .
\end{aligned}
\end{equation}

\end{definition}

The following simple lemma is used in Lemma \ref{lem:chcoo}.
\begin{lemma} \label{lem:comb-1}  The following facts hold.
\begin{itemize}
  \item[(1)]
If $(\mu , \nu )\in M_{min}$ then for any $j$ we have $\mu  _{j} \nu _{j}=0$.
\item[(2)] Suppose that   $(\mu , \nu )\in M_{min}$,   $(\alpha , \beta )\in M_{min}$  and
$(\mu - \nu ) \cdot \mathbf{e} =(\alpha - \beta ) \cdot \mathbf{e}  $. Then  $(\mu , \nu )=(\alpha , \beta ) $.
\end{itemize}

\end{lemma}
\proof First of all it is easy to show  that  $(\mu , \nu )\in M(r)$    if and only if $|\nu |= |\mu |+1$, $|\mu | \le r$  and $|(\mu - \nu )\cdot \mathbf{e}|>\mathscr M$.

Suppose  that  $\mu  _{j} \ge 1$ and $\nu  _{j} \ge 1$.
Then consider $(\alpha , \beta )    $ \begin{equation*}
\begin{aligned}
&    \alpha _{  k}  =\left\{\begin{matrix}
          \mu _k    \text{ for $ k\neq j $} ,\\
  \mu _j-1    \text{ for $ k=j$}
\end{matrix}\right.     \qquad    \beta _k  =\left\{\begin{matrix}
          \nu _k    \text{ for $ k\neq j $}  ,\\
  \nu _j-1    \text{ for $ k=j$}
\end{matrix}\right.  .
\end{aligned}
\end{equation*}
Since $|\nu |= |\mu |+1$ we have $|\beta  |= | \alpha |+1$. Furthermore
$(\mu - \nu )\cdot \mathbf{e} = (\alpha - \beta )\cdot \mathbf{e}$. This implies
$(\alpha , \beta )  \in M $ and so   $(\mu , \nu )\not \in M_{min}$.
This proves claim (1).

\noindent  By  $(\mu - \nu ) \cdot \mathbf{e} =(\alpha - \beta ) \cdot \mathbf{e}  $   and (H4) we obtain $\mu - \nu = \alpha - \beta$, which by claim (1) yields  $(\mu , \nu )=(\alpha , \beta ) $.

\qed

\begin{lemma} \label{lem:M0} Assuming $(H4)$ then the following properties are fulfilled.

 \begin{itemize}
\item[(1)]
 For $\mathbf{Z}^{\mathbf{m}}=z^{\mu}\overline{z}^{\nu}$, then $\mathbf{m}\in \mathcal{M}_0 (2N+4)$ implies  $\mu =\nu$. In particular  $\mathbf{m}\in \mathcal{M}_0 (2N+4)$ implies $\mathbf{Z}^{\mathbf{m}}=|z_1| ^{2l_1}...|z_n| ^{2l_n}$ for some $(l_1,..., l_n)\in \N _0^n$.

 \item[(2)] For $|\textbf{m}  |  \le  2N+3$  and any $j$ we   have $\sum _{a,b}(e_a-e_b)  {m}_{ab} -e_j \neq 0$.

 \end{itemize}
\end{lemma}
\proof The following  proof is in   \cite{CM1}.
If  $\mu =\nu$ then $z^{\mu}\overline{z}^{\nu}=|z_1| ^{2\mu _1}...|z_n| ^{2\mu _n} $. So the first sentence in claim (1)  implies the second sentence in claim (1).
  We have
\begin{equation*}     \begin{aligned} &  \mathbf{Z}^{\mathbf{m}} = \prod _{i,l=1} ^{n}(z_i\overline{z}_l)^{m_{il}}= \prod _{i =1}^{n} z_i^{\sum _{l=1}^{n}m_{il} } \overline{z}_i^{\sum _{l=1}^{n}m_{li} } =  z^{\mu}\overline{z}^{\nu}.
\end{aligned}\end{equation*}
The pair $(\mu , \nu)$ satisfies $|\mu |=|\nu|\le 2N+4$ by
\begin{equation*}     \begin{aligned} &
 |\mu |=\sum _l \mu _l=\sum _{i,l} m_{il}  = |\nu| .
\end{aligned}\end{equation*}
We have $(\mu -\nu )\cdot \mathbf{e}=0$ by $\mathbf{m}\in \mathcal{M}_0 (2N+4)$ and
\begin{equation*}     \begin{aligned} &  \sum _i \mu _i e_i -\sum _l \nu _l e_l =  \sum _{i,l} m_{il}(e_i - e_l) = 0   .
\end{aligned}\end{equation*}
We conclude by (H4) that $\mu -\nu =0$.   This proves the 1st sentence of  claim (1).

\noindent The proof of claim (2)  is similar. Set
\begin{equation*}     \begin{aligned} &  \mathbf{Z}^{\mathbf{m}}\overline{z}_j = \prod _{i,l=1} ^{n}(z_i\overline{z}_l)^{m_{il}} \overline{z}_j = \prod _{i =1}^{n} z_i^{\sum _{l=1}^{n}m_{il} } \overline{z}_i^{\sum _{l=1}^{n}m_{li} } \overline{z}_j =  z^{\mu}\overline{z}^{\nu}.
\end{aligned}\end{equation*}
  We have
\begin{equation*}     \begin{aligned} & (\mu - \nu )\cdot \mathbf{e}=  \sum _i \mu _i e_i -\sum _l \nu _l e_l  =  \sum _{i,l} m_{il}(e_i - e_l)-e_j  ,
\end{aligned}\end{equation*}
in addition we have also
\begin{equation}  \label{eq:in}   \begin{aligned} &
 |\mu |=\sum _l \mu _l=\sum _{i,l} m_{il}  = |\nu|-1 .
\end{aligned}\end{equation}
If  $(\mu - \nu )\cdot \mathbf{e}=0$  then by $|\mu -\nu |\le 4N+5$
and by (H4) we would have $\mu =\nu$,   impossible by \eqref{eq:in}.

\qed

The   following elementary  lemma is very similar to Lemma 2.4 \cite{CM1}
and is used to bound the 2nd and 3rd term in the 1st line of
\eqref{eq:rest1}  in terms of the $z ^{\mu}\overline{z}^{\nu}$ with $(\mu ,\nu)\in M _{min}$
and $\eta$.

\begin{lemma} \label{lem:comb1} We have the  following facts.

\begin{itemize}
\item[(1)]
Consider a vector $\mathbf{m}=( m  _{ij}) \in \N _0 ^{n_0}  $ such that
$\sum _{i<j}m  _{ij} >N$ for   $N >\mathscr{M} (\min \{ e_j -e_i : j>i \}) ^{-1}$, see (H4). Then for any eigenvalue $e_k$  we have

\begin{equation}\label{eq:comb3}
\begin{aligned}
&  \sum _{i<j}  m  _{ij} (e_i -e_j) - e_k<-\mathscr{M} .
\end{aligned}
\end{equation}

\item[(2)] Consider $\mathbf{m}\in \N _0 ^{n_0} $   and the monomial $ z _j \mathbf{Z}^{\mathbf{m}}$. Suppose   $|\mathbf{m}|   \ge 2N+3$.
    Then there are  $\mathbf{a}, \mathbf{b}\in \N _0 ^{n_0} $ such  that we have
\begin{equation}\label{eq:comb4}
\begin{aligned}
& \sum _{i<j}  a  _{ij}    =N+1=\sum _{i<j}  b  _{ij} ,\\&
  a  _{ij}    =  b  _{ij}  =0  \text{ for all }   i>j ,
  \\&  a  _{ij}   +  b  _{ij} \le m  _{ij} + m  _{ji}   \text{ for all } (i,j)
\end{aligned}
\end{equation}
   and moreover there are   two indexes $ (k,l)$ such that
 \begin{equation}\label{eq:comb5}
\begin{aligned}
& \sum _{i<j}  a  _{ij} (e_i -e_j) - e_k< -\mathscr{M}, \ \  \ \sum _{i<j}  b  _{ij} (e_i -e_j) - e_l< -\mathscr{M}
\end{aligned}
\end{equation}
 and such that  for $|z|\le 1$
  \begin{equation}\label{eq:comb6}
\begin{aligned}
& |z _j \mathbf{Z}^{\mathbf{m}}|\le |z_j| \  |z _k \mathbf{Z}^{\mathbf{a}}|  \  |z _l \mathbf{Z}^{\mathbf{b}}| .
\end{aligned}
\end{equation}

\item[(3)] For $\mathbf{m}$ with $|\mathbf{m}|   \ge 2N+3$  there exist $(k,l)$ and $\mathbf{a} \in \mathcal{M}_k$ and $\mathbf{b} \in \mathcal{M}_l$ such that  \eqref{eq:comb6}  holds.

    \end{itemize}
\end{lemma}
\proof  The proof is very similar to Lemma 2.4 in \cite{CM1}. For example,
\eqref{eq:comb3} follows immediately from
\begin{equation*}
\begin{aligned}
& \sum _{i<j}  m  _{ij} (e_i -e_j) - e_k\le -\min \{ e_j -e_i : j>i \} N -e_1 <-\mathscr{M},
\end{aligned}
\end{equation*}
 with the latter inequality due to the definition of $N$. All the other
claims can be proved like the rest of Lemma 2.4 in \cite{CM1}.

Given  $\mathbf{a}, \mathbf{b}\in \N _0 ^{n_0} $ satisfying  \eqref{eq:comb4},
by claim (1) they satisfy   \eqref{eq:comb5}  for any pair of indexes  $ (k,l)$.  Consider now the monomial  $ z _j \mathbf{Z}^{\mathbf{m}}$.   Since  $|\mathbf{m}|\ge 2N+3$, there are vectors
$\mathbf{c}, \mathbf{d}\in \N _0 ^{n_0} $ such that  $|\mathbf{c}|=|\mathbf{d}| = N+1$ with $c  _{ij}   +  d  _{ij} \le m  _{ij}  $   for all  $(i,j)$. Furthermore
we have
 \begin{equation}\label{eq:comb7}
\begin{aligned}
& z _j \mathbf{Z}^{\mathbf{m}} = z _j   z^\mu \overline{z} ^\nu \mathbf{Z}^{\mathbf{c}}  \mathbf{Z}^{\mathbf{d}} \text{ with $|\mu | >0 $ and  $|\nu | >0 $.}
\end{aligned}
\end{equation}
So, for $z_k$ a factor of $z^\mu $ and  $\overline{z}_l$ a factor of $\overline{z}^\nu $,
and for
\begin{equation}\label{eq:comb8}
\begin{aligned}
&    a_{ij}  =\left\{\begin{matrix}
          c_{ij}+c_{ji}  \text{ for $i<j$} \,\\
  0     \text{ for $i>j$}
\end{matrix}\right.  ,  \qquad    b_{ij}  =\left\{\begin{matrix}
          d_{ij}+d_{ji}  \text{ for $i<j$} \,\\
  0     \text{ for $i>j$}
\end{matrix}\right.,
\end{aligned}
\end{equation}
for $|z|\le 1$ we have from \eqref{eq:comb7}
\begin{equation*}
\begin{aligned}
& |z _j \mathbf{Z}^{\mathbf{m}}| \le | z _j  | \ |  z_k   \mathbf{Z}^{\mathbf{c}} | \  | \ z_l \mathbf{Z}^{\mathbf{d}}  | =  | z _j  | \ |  z_k   \mathbf{Z}^{\mathbf{a}} | \  | \ z_l \mathbf{Z}^{\mathbf{b}}  | .
\end{aligned}
\end{equation*}
Furthermore,  \eqref{eq:comb4} is satisfied.

\noindent Since our $(\mathbf{a},\mathbf{b})$ satisfy
$\mathbf{a} \in \mathcal{M}_k$ and $\mathbf{b} \in \mathcal{M}_l$,
claim (3) is a consequence of claim (2).

\qed

Since $(z,\Theta  )$
in   \eqref{eq:decomposition1}  are not a system of independent coordinates we need
the  following,  see Lemma 2.5   \cite{CM1}.

\begin{lemma} \label{lem:contcoo}
There exists
$d _0>0$ such that for all $z\in\C$ with $|z|<d _0$  there exists $R[z]:\mathcal{H}_c[0]\to \mathcal{H}_c[z]$ such that  $ \left. P_c\right|_{\mathcal{H}_c[z]}=R[z]^{-1}$,
with $P_c$ the orthogonal projection of $L^2$ onto $\mathcal{H}_c[0]$, see   Definition \ref{def:contsp}.  Furthermore, for $|z|<d_0$ and $\eta \in  \mathcal{H}_c[0] $, we have the following properties.

\begin{itemize}
\item[(1)]
$R[z]\in C^\infty (B _{\C^n}  (0, \delta _0 ), B  (
H^r,H^ r ) ),$  for any $r\in \R$.

\item[(2)]  For any $r>0$, we have $\|(R[z]-1)\eta \|_{\Sigma _r}\le c_r |z	 |^2 \|\eta \|_{\Sigma _{-r}}$  for a fixed $c_r$.

\item[(3)]   We have the covariance property  $R[e^{\im \vartheta }z]= e^{\im \vartheta }R[z] e^{-\im \vartheta }.$

\item[(4)]    We have, summing on repeated indexes,
\begin{equation} \label{eq:contcoo21}
 R[z]\eta =\eta  +   (\alpha_j[z]\eta )\phi_j,  \text{  with }    \alpha_j[z]\eta =\langle B_j (z), \eta \rangle + \langle C_j(z),  {\eta }^*\rangle,
\end{equation}
where, for $ \widehat{Z} $ as in Definition \ref{def:comb0},  we have $B_j (z)=\widehat{B}_j (\widehat{Z} )$ and $C_j(z) = z_{i}  {z}_\ell  \widehat{C}_{i\ell j}(\widehat{Z})$, for $\widehat{B} _j$  and $\widehat{C}_{i\ell j}$  smooth in $\widehat{Z}$  with values in $\Sigma _r$.

\item[(5)] We have, for $r\in \R$, with $\textbf{Z}$   as in Definition \ref{def:comb0}\begin{equation} \label{eq:contcoo2}
\begin{aligned}  &   \| {B} _j  (z) + \partial _{\overline{z}_j}   {q}_{jz_j}^* \|_{\Sigma _r}  + \| C _j  (z)-\partial _{ \overline{z}_j}   {q}_{jz_j} \|_{\Sigma _r}   \le c_r |\textbf{Z}	 |^2.
\end{aligned}
\end{equation}

\end{itemize}
\end{lemma}

\qed

Then Lemma 2.6 gives us a  system of coordinates near the origin in $H^4$.
The simple proof is the same of Lemma 2.6 \cite{CM1}.

\begin{lemma} \label{lem:systcoo}  For the $d  _0>0$ of Lemma \ref{lem:contcoo}
the map $(z,\eta )\to u$  defined
by
\begin{equation} \label{eq:systcoo1}
 u=\sum_{j=1}^nQ_{j z_j}+R[z] \eta , \text{ for $(z,\eta )\in B_{\C^n}(0 , d _0) \times ( H^4\cap \mathcal{H}_c[0]) ,$}
\end{equation}
is  with values in   $H^4$  and  is   $C^\infty$.
Furthermore, there is a $d_1>0$ such that    for $(z,\eta )\in B_{\C^n}(0 , d _1) \times  (B_{H^4}(0 , d _1)\cap \mathcal{H}_c[0]),$
the above map is a  diffeomorphism and
\begin{equation} \label{eq:coo11}
  |z|+\| \eta \| _{H^4} \sim   \| u \| _{H^4}.
\end{equation}

\noindent
	Finally, we have the gauge properties  $u ( e^{\im \vartheta } z, e^{\im \vartheta } \eta )= e^{\im \vartheta } u (z,\eta ) $
	and    \begin{equation}\label{eq:detion31}
\begin{aligned} &
 \text{$z(e^{\im \vartheta} u)=e^{\im \vartheta} z( u)
$ and    $\eta (e^{\im \vartheta} u)=e^{\im \vartheta} \eta ( u)$ } .
\end{aligned}
\end{equation}

\end{lemma}

\qed

We end this section exploiting the notation introduced  in claim (5) of Lemma \ref{lem:contcoo}  to introduce two classes  of functions.   First of all notice that    the linear maps $\eta \to  \< {\eta},\phi _j^* \>$
extend into bounded linear maps $\Sigma _r\to \R$ for any $r\in \R $.   We set
\begin{equation}\label{eq:phsp1}
\begin{aligned}
\Sigma _r^c & :=\left\{\eta\in \Sigma _r :\  \< {\eta},\phi _j ^*\>=0,\ j=1,\cdots,n\right\} .
\end{aligned}
\end{equation}
The following two classes of functions  will be used in the rest of the paper.
Recall  that in   Definition \ref{def:comb0}  we introduced
$\textbf{Z}\in L$ with   $\dim L= n (n-1).$

\begin{definition}\label{def:scalSymb}
We will say that   $F(t,   z, Z ,\eta )\in C^{M}(I\times   \mathcal{A},\R)$, with
$I$ a neighborhood of 0 in $\R$ and
 $\mathcal{A}$   a neighborhood of 0 in  $  \C ^n \times L \times \Sigma _{-K}^c $
is   $F=\mathcal{R}^{i, j}_{ K,M} (t,  z,\textbf{Z},\eta)$,
 if    there exist    a $C>0$   and a smaller neighborhood  $\mathcal{A}'$ of 0   such that
 \begin{equation}\label{eq:scalSymb}
  |F(t,  z,\mathbf{Z},\eta)|\le C (\|  \eta \| _{\Sigma   _{-K}}+|\textbf{Z} |)^j (\|  \eta \| _{\Sigma   _{-K}}+ |\mathbf{Z}   |+|z |)^{i} \text{  in $I\times  \mathcal{A}  '$} .
\end{equation}
We will specify      $F=\mathcal{R}^{i, j}_{ K,M} (t, z,\textbf{Z})$    if
\begin{equation}\label{eq:scalSymb1}
  |F(t,  z,\mathbf{Z},\eta)|\le C |\textbf{Z} | ^j  |z | ^{i}
\end{equation}
and       $F=\mathcal{R}^{i, j}_{ K,M} (t,  z,\eta  )$    if
\begin{equation}\label{eq:scalSymb2}
  |F(t,  z,\mathbf{Z},\eta)|\le C  \|  \eta \| _{\Sigma   _{-K}} ^j (\|  \eta \| _{\Sigma   _{-K}}+|z |)^{i} .
\end{equation}
We will  omit $t$    if there is no dependence on such variable.
We write  $ F=\mathcal{R}^{i, j} _{K, \infty}$  if $F=\mathcal{R}^{i,j}_{K, m}$ for all $m\ge M$.
We write    $F=\mathcal{R}^{i, j}_{\infty, M} $       if   for all   $k\ge K$    the above   $F$ is the restriction  of an
$F(t,  z,\eta )\in C^{M}(I\times   \mathcal{A}_{k },\R)$ with  $\mathcal{A}_k$   a neighborhood of 0 in
$  \C  ^{n }\times L\times \Sigma _{-k}^c $ and
 which is
$F=\mathcal{R}^{i,j}_{k, M}$.
Finally we write
$F=\mathcal{R}^{i, j} _{\infty, \infty} $   if $F=\mathcal{R}^{i, j} _{k, \infty}$  for all $k$.

\end{definition}

\begin{definition}\label{def:opSymb}  We will say that an   $T(t, z,\eta )\in C^{M}(I\times   \mathcal{A},\Sigma   _{K}  (\R^3, \C ))$,  with the above notation,
 is   $T= \mathbf{{S}}^{i,j}_{K,M} (t,  z,\textbf{Z},\eta)$,
 if     there exists  a $C>0$   and a smaller neighborhood  $\mathcal{A}'
$ of 0   such that
 \begin{equation}\label{eq:opSymb}
  \|T(t,  z,\mathbf{Z},\eta)\| _{\Sigma   _{K}}\le    C (\|  \eta \| _{\Sigma   _{-K}}+|\textbf{Z} |)^j (\|  \eta \| _{\Sigma   _{-K}}+ |\mathbf{Z}   |+|z |)^{i}   \text{  in $I\times  \mathcal{A}'$}.
\end{equation}
We  use notations
$ \mathbf{{S}}^{i,j}_{K,M} (t,  z,\textbf{Z})$, $ \mathbf{{S}}^{i,j}_{K,M} (t, z,\eta  )$ etc.  as above.

\end{definition}

 \begin{remark}\label{rem:sym}
 For  given  functions   $F(t,  z, \eta)$ and $T(t,  z, \eta)$ we   write
  $F(t,  z, \eta)=\resto ^{i,j}_{K,M} (t,  z,\textbf{Z},\eta)$ and  $T(t, z, \eta)=\mathbf{{S}}^{i,j}_{K,M} (t,  z,\textbf{Z},\eta)$ when they are restrictions
  to the set of vectors  $\mathbf{Z}\in \{ (z_i \overline{z}_j) _{i,j =1,...,n}:i\neq j\}   $
 of functions satisfying the two above definitions.

\noindent  Furthermore  later, when we write $\resto ^{i,j}_{K,M} $ and $\mathbf{{S}}^{i,j}_{K,M}$,    we   mean $\resto ^{i,j}_{K,M} ( z,\textbf{Z},\eta) $ and $\mathbf{{S}}^{i,j}_{K,M}  ( z,\textbf{Z},\eta)$.

\noindent   Notice that $F= \resto ^{i,j}_{K,M} ( z,\textbf{Z} ) $ or  $S= \mathbf{S} ^{i,j}_{K,M} ( z,\textbf{Z} ) $ do not mean independence by the variable $\eta$.
\end{remark}

\section{Invariants}
\label{sec:invariants}

Equation \eqref{eq:NLS}  admits the  energy and mass  invariants,
   defined as follows for $G(0)=0$ and $G'(s)=g (s)
$:
\begin{equation}\label{eq:inv}
\begin{aligned}
&E(u ):= E_ K(u)   + E_P(u) \text{, where }  E_ K(u) := \langle  D_\mathscr{M} u, {u} ^* \rangle \text{ and } \\&
 E_P(u)= \frac 12 \int _{\R ^3}G( u \overline{u}) dx; \   \quad
Q(u ):= \langle    u, {u} ^* \rangle  .
\end{aligned}
\end{equation}
   We have   $ {E}\in C^\infty ( H^4( \R ^3, \C  ), \R  )$ and $Q\in C^\infty (  L^2(
\R ^3, \C ) , \R  )$.    We denote by $dE$ the Frech\'et derivative of $E$.
 We define $ \nabla {E}\in C^\infty ( H^4( \R ^3, \C  ), H^4( \R ^3, \C  )  )$ by  $dE(X)= \Re \langle   \nabla {E},  {X} ^{*}\rangle$,
  for any $X\in   H^4$. We define also $\nabla _{u}{E}$ and  $\nabla _{ {u}^{*}}{E}$
 by
 \begin{equation*}
\begin{aligned}
& dEX =   \langle   \nabla _{u} {E},  {X} \rangle  + \langle   \nabla _{{u}^{*}} {E}, {X} ^{*} \rangle,  \text{ that is }
 \nabla _{u} {E} =2^{-1}  (\nabla {E}) ^*\text{  and } \nabla _{{u}^{*}} {E} =2^{-1}  {\nabla {E}}.
\end{aligned}
\end{equation*}
Notice that $ \nabla   {E} =2Hu+2 g(u\overline{u} )\beta u$.
Then equation \eqref{eq:NLS} can be
interpreted as
\begin{equation}\label{eq:NLSham}
\begin{aligned}
& \im \dot u = \nabla _{  {u}^{*}} E(u ) .
\end{aligned}
\end{equation}

We recall that normal forms arguments consist in making Taylor expansions
of the hamiltonian and in the cancellation of the \textit{non--resonant} terms
of the expansion.
The following proposition identifies the kind of expansion we have in mind.
We should think of \eqref{eq:enexp1} as an expansion in the variables $z$, $\eta$ and the
auxiliary variable $\mathbf{Z}$. Eventually the effective hamiltonian will
contain terms  in the r.h.s. such as the 1st and 2nd  in the 1st line and the terms of the 2nd line. The cancellations will occur later in the 2nd line.
\begin{proposition}
  \label{prop:EnExp}
We have the following expansion of the energy for   any preassigned $r_0\in \N$:
\begin{align}
   &E (u)=     \sum _{j=1}^{n}E (Q_{j z_j}) + \langle   H   \eta,   { \eta} ^*\rangle
    +\resto ^{1,  2  }_{r_{0} , \infty} ( z,\eta  ) +
\resto ^{0,  2N+5 }_{r_{0} , \infty} (z,\textbf{Z}  )
       \nonumber   \\& \nonumber  +
	 \sum _{j =1}^n\sum _{l =0} ^{2N+3}	 \sum _{
	  |\textbf{m}|=l+1 }  \textbf{Z}^{\textbf{m}}   a_{j \textbf{m} }( |z_j |^2 )+
   \sum _{j,k=1}^n \sum _{l =0} ^{2N+3}	 \sum _{
	  |\textbf{m}|=l  }
     ( \overline{z}_j \textbf{Z}^{\textbf{m}}  \langle
  G_{jk\textbf{m}}(|z_k|^2  ), \eta \rangle  +c.c. )   \\&   \nonumber +   \Re \langle
  \textbf{S} ^{0,   2N+4 }_{r_{0} , \infty} (z,\textbf{Z}  ) ,  {\eta}^* \rangle
		+ \sum _{ i+j  = 2}   \sum _{
		  |\textbf{m}|\le 1 }  \textbf{Z}^{\textbf{m}}    \langle G_{2\textbf{m} ij } ( z ),   \eta ^{  \otimes i}\otimes ({\eta}^*) ^{\otimes j}\rangle\\&
		+  	   \sum _{  d  = 2  }^{3}   \sum _{ i  = 1}^d \mathcal R ^{0,  3-d }_{r_{0} , \infty} (z, \eta ) \int _{\R ^3} G_{d i } ( x,z, \eta , \eta (x) )    \eta ^{  \otimes i}(x)\otimes ({\eta}^*(x)) ^{\otimes (d-i)} dx          + E _P( \eta),  \label{eq:enexp1}
  \end{align}
where $c.c.$ is the complex conjugate of the term right before in   the $( \ )$ and where:
  \begin{itemize}

\item[(1)] $  ( a_{j\textbf{m} } ,  G_{jk\textbf{m}}  , G_{2\textbf{m} ij })  \in C^\infty (B_{\R}(0,d  _0), \C\times  \Sigma _{r_{0}}(\mathbb{R}^3, \mathbb{C}   )\times
\Sigma _{r_{0}} (\mathbb{R}^3, B^{i+j}(\C ^4,\mathbb{C})   )) $;

 \item[(2)]
 $G_{d i }( \cdot ,z, \eta , \zeta ) \in C^{\infty } ( B_{\C ^n}(0,d _0) \times \Sigma _{-r_{0}} (\mathbb{R}^3, \C)\times \C^4,
\Sigma _{r_{0}} (\mathbb{R}^3, B^{d}(\C ^4,\mathbb{C})   ))) $;

\item[(3)]
for $|\textbf{m}|=0$ we have $G_{2\textbf{0} ij } (0)=0$ and
\begin{equation}\label{eq:b02}
\begin{aligned}
  &    \sum _{ i +j =  2}           \langle G_{2\textbf{0} i j} ( z ),   \eta ^{  \otimes i} ({\eta}^*) ^{\otimes  j}\rangle = 2^{-1} \sum _{j=1}^{n}
  \langle g ( Q_{j z_{j}} \overline{Q}_{j z_{j}}) \eta  , {\eta }^* \rangle \\& +  \sum _{j=1}^{n}  \Re \langle  g'(Q_{j z_{j}} \overline{Q}_{j z_{j}} ) \Re (Q_{j z_{j}} \overline{\eta}) \beta    Q_{j z_{j}}  ,    {\eta } ^*\rangle .\end{aligned}
\end{equation}

\end{itemize}
\end{proposition}

In  order to prove Proposition \ref{prop:EnExp}
 we set
\begin{equation*} \begin{aligned} &
K(z,\eta):=E(\sum_{j=1}^nQ_{j z_j}+R[z] \eta)=K_K(z,\eta)+K_P(z,\eta), \text{ with} \\&
K_K(z,\eta):=E_K(\sum_{j=1}^nQ_{j z_j}+R[z] \eta),\quad K_P(z,\eta):=E_P(\sum_{j=1}^nQ_{j z_j}+R[z] \eta).  \end{aligned}
\end{equation*}

\noindent
By Taylor expansion, we   write
 \begin{align}
&K(z,\eta)=K(z,0)+\Re \< \partial _{\eta } K (z,0), \eta ^* \> + \frac  1 2 \Re  \< \partial _{\eta }  ^2 K(z,0)  \eta,\eta ^*\> + K_3(z,\eta),\label{eq:KP0}\\& K_3(z,\eta):=\frac 1 2 \int_0^1(1-t)^2 \Re  \< \partial _{\eta } ^3K(z,t\eta)  \eta ^2,\eta ^* \>\,dt \label{eq:KP1}.\end{align}

\noindent
We  expand
\begin{equation}\label{eq:KP2}
K_3(z,\eta)=K_3(0,\eta)+R_P(z,\eta), \text{ where $K_3(0,\eta)=K_P(0,\eta)=E_P(\eta )$}
\end{equation}
 and
\begin{equation}
\begin{aligned}
R_P(z,\eta)&=\int_0^1 \partial_zK_3(tz,\eta) z\,dt:=\int_0^1\sum_{j=1}^n\sum_{A=R,I}D_{jA}K_3(tz,\eta)z_{jA}\,dt\\&=\sum_{j=1}^n\int_0^1\(\partial_jK_3(tz,\eta)z_j+\partial_{\bar j}K_3(tz,\eta)\bar z_j\)\,dt.
\end{aligned}
\end{equation}

To prove {Proposition}
  \ref{prop:EnExp} we compute the terms of
\begin{equation}\label{eq:5termcompute}
K(z,\eta)=K(z,0)+\Re \<\partial _{\eta }  K (z,0), \eta  ^*\> + \frac  1 2 \Re \< \partial _{\eta }^2 K(z,0)  \eta,\eta ^*\> +E_P(\eta )+R_P(z,\eta),
\end{equation}
starting with $\partial _{\eta } K$, $\partial _{\eta } ^2K$ and $\partial _{\eta } ^3K$.

\begin{lemma}\label{lem:expKP} Set   $u=u(z,\eta )= \sum_{j=1}^n Q_{jz_j}+R[z] \eta $. We have the following equalities:
\begin{equation}
\begin{aligned}\label{eq:lemexpK}
&\partial _{\eta } K_K(z,\eta)=2R[z]^*H u ,\\&
\partial _{\eta } ^2 K_K(z,\eta)=2R[z]^*H R[z],\quad
\partial _{\eta } ^3 K_K(z,\eta)=0;
\end{aligned}
\end{equation}
\begin{equation}\label{eq:lemexpP}
\begin{aligned}
&\partial _{\eta } K_P(z,\eta)= R[z]^*\(g (u \overline{u })\beta u \)  ,\\
&\partial _{\eta } ^2 K_P (z,\eta)\nu= R[z]^*g (u \overline{u })\beta R[z]\nu +2R[z]^*g'(u \overline{u })\mathrm{Re}\(u \ \overline{R[z]\nu}\) \beta u ,\\&
\(\partial _{\eta } ^3 K_P (z,\eta)\nu\)\nu= 6R[z]^* g'(u \overline{u })\mathrm{Re}\(u \overline{R[z]\nu}\)\beta R[z]\nu  +
4R[z]^*g''(u \overline{u })\(\mathrm{Re}\(u \overline{R[z]\nu}\)\)^2\beta u .
\end{aligned}
\end{equation}
In particular  we have
\begin{equation*}
\begin{aligned}
 \partial _{\eta } K_K(z,0)&=2R[z]^*H \sum_{j=1}^nQ_{j z_j},\ \quad
\partial _{\eta }^2 K_K(z,0)=2R[z]^*H R[z] ,\\  \partial _{\eta}  K_P(z,0)&=  g( \sum_{j,k }  Q_{jz_j}   \overline{Q}_{kz_k} )R[z]^* \beta  \sum_{j=1}^n Q_{jz_j}   ,\\
\partial _{\eta } ^2 K_P(z,0)&=   g( \sum_{j,k }  Q_{jz_j}   \overline{Q}_{kz_k} )R[z]^* \beta R[z]\\& +2g'(  \sum_{j,k }  Q_{jz_j}   \overline{Q}_{kz_k} ) \mathrm{Re}\( \sum_{j=1}^n Q_{jz_j} \overline{R[z]\ \cdot\ }\)     R[z]^*\beta \sum_{j=1}^n Q_{jz_j}  .
\end{aligned}
\end{equation*}
\end{lemma}

\proof
  We get \eqref{eq:lemexpK}    by
\begin{equation*}
\begin{aligned}
K_K(z,\eta+\varepsilon\nu)&=\Re \<H (u(z,\eta ) +\varepsilon R[z]  \nu ), (u(z,\eta ) +\varepsilon R[z]  \nu )^*\>\\&
=K_K(z,\eta)+2 \varepsilon \Re  \<H u(z,\eta ) , (R[z] \nu ) ^*\>+\varepsilon^2\Re  \<H R[z] \nu, (R[z] \nu ) ^*\> .
\end{aligned}
\end{equation*}
   Moreover  we arrive at  \eqref{eq:lemexpP}  by
\begin{equation*}
\begin{aligned}
&K_P(z,\eta+\varepsilon\nu)=2^{-1}\int G((u(z,\eta )+\varepsilon R[z] \nu )(\overline{u(z,\eta )+\varepsilon R[z] \nu}) )\,dx,
\end{aligned}
\end{equation*}
and   by
\begin{equation*}
\begin{aligned}
&G((u(z,\eta )+\varepsilon R[z] \nu )(\overline{u(z,\eta )+\varepsilon R[z] \nu}) )\\&=
G\( u(z,\eta )\overline{u(z,\eta )}+2\varepsilon \mathrm{Re}\(u(z,\eta )\overline{R[z]\nu}\)+\varepsilon^2 R[z]\nu \overline{R[z]\nu}\) \\&=  o(\varepsilon^3) +
 G(u(z,\eta )\overline{u(z,\eta )} ) +
2\varepsilon    g (u(z,\eta )\overline{u(z,\eta )})\mathrm{Re}\( \beta u(z,\eta ) (R[z]\nu )^* \)\\&+
\varepsilon^2 \left [ g(u(z,\eta )\overline{u(z,\eta )} )R[z]\nu \overline{R[z]\nu} +
2  g'(u(z,\eta )\overline{u(z,\eta )} )\( \mathrm{Re}\(u(z,\eta )\overline{R[z]\nu}\) \)^2\right ] + \\&   \varepsilon ^3  \big  [
2g'(u(z,\eta )\overline{u(z,\eta )} )\mathrm{Re}\(u(z,\eta ) \overline{R[z]\nu}\)R[z]\nu \overline{R[z]\nu}
+\frac 4 3    g''(u(z,\eta )\overline{u(z,\eta )})\Re \(u(z,\eta )\overline{R[z]\nu}\)^3
 \big ] .
\end{aligned}
\end{equation*}

\qed

 We now examine the    r.h.s. of
 \eqref{eq:5termcompute}.

 \begin{lemma} \label{lem:1exp}  Consider the first two terms in the r.h.s. of
 \eqref{eq:5termcompute}. We then have
\begin{align} &
K(z,0)= \sum _{j=1}^{n}E (Q_{j z_j})   +
	 \sum _{j =1}^n\sum _{l =0} ^{2N+3}	 \sum _{
	  |\textbf{m}|=l+1 }  \textbf{Z}^{\textbf{m}}   a_{j \textbf{m} }( |z_j |^2 )+
\resto ^{0,  2N+5 }_{\infty , \infty} (z,\textbf{Z}  ),\label{eq:1exp}\\&
\Re \<\partial _{\eta }  K (z,0), \eta  ^*\> =
   \sum _{j,k=1}^n \sum _{l =0} ^{2N+3}	 \sum _{
	  |\textbf{m}|=l  }
     ( \overline{z}_j \textbf{Z}^{\textbf{m}}  \langle
  G_{jk\textbf{m}}(|z_k|^2  ), \eta \rangle  +c.c. )+
        \Re \langle
  \textbf{S} ^{0,   2N+4 }_{\infty , \infty} (z,\textbf{Z}  ) ,  {\eta}^* \rangle   \label{eq:2exp},
\end{align}
 where the coefficients in the r.h.s.'s have the properties listed in claim (1) in  Prop. \ref{prop:EnExp}.

\end{lemma}
  \proof  First of all, both l.h.s.'s of \eqref{eq:1exp}--\eqref{eq:2exp}
  are gauge invariant. Then
  \eqref{eq:2exp} is an immediate consequence of claim (3) of Lemma
  \ref{lem:gaugesmooth} below.

  We have
  \begin{equation*}  \begin{aligned}  & K(z,0)= E (\sum _{j=1}^{n} Q_{j z_j}) = E (  Q_{1 z_1})+E (\sum _{j>1} Q_{j z_j})+\int   _{[0,1]^2}  \frac{\partial ^2}{\partial s\partial t} E ( sQ_{1 z_1} +t\sum _{j>1}  Q_{j z_j})dtds \\& =\sum _{k=1}^{n} E ( Q_{k z_k}) + \alpha _k(z)  \text{ with }
  \alpha _k(z):=\sum _{k=1}^{n}\int   _{[0,1]^2}  \frac{\partial ^2}{\partial s\partial t} E (s Q_{k z_k} +t\sum _{j>k}  Q_{j z_j})dtds.
  \end{aligned}
 \end{equation*}
  The $\alpha _k(z)$ are gauge invariant, so that we can apply to them
  claim (2) of Lemma
  \ref{lem:gaugesmooth} below. Furthermore, since $\alpha _k(z)=O(|\mathbf{Z}|)$ we conclude that  in the expansion \eqref{eq:zj1}
  for  $\alpha _k(z) $ we have equalities
  $b_{j\mathbf 0}(|z_j|^2)=0$.

  \qed

\begin{lemma}\label{lem:gaugesmooth}   The following facts hold:

\begin{itemize}

 \item[(1)]
For $a( \zeta  )$   smooth from $B_{\C}(0,\delta)$ to $\R$  such that  $a(e^{\im\theta}\zeta )=a(\zeta )$ for any $\theta\in\R$  there exists $\alpha\in \C^\infty([0,\delta^2);\R)$ such that $\alpha(|\zeta |^2)=a(\zeta )$.

\item[(2)]
Let $a\in C^\infty (B_{\C^{n}}(0,\delta),\R ) $ satisfy $a(e^{\im\theta}z_1,\cdots,e^{\im\theta}z_n)=a(z_1, \cdots, z_n)$ for all $\theta \in \R $ and $a(0)=0$.
Then  for any $M>0$  there exist smooth   $b_{j\mathbf{0}}$ such that $b_{j\mathbf 0}(|z_j|^2)=a(0,\cdots,0,z_j,0,\cdots,0)$ and
\begin{equation}\label{eq:zj1}
a(z_1,\cdots,z_n)= \sum_{|\mathbf{m}|\leq M-1}\mathbf Z^{\mathbf m}b_{j\mathbf m}(|z_j|^2)+\mathcal R^{0,M}_{\infty ,\infty}(z,\mathbf Z).
\end{equation}

\item[(3)] Let    $a\in C^\infty (B_{\C^{n}}(0,\delta),\Sigma_r)$  $\forall$ $r\in \R$   such that $a(e^{\im\theta}z_1,\cdots,e^{\im\theta}z_n)=e^{\im\theta}a(z_1,\cdots,z_n)$.
Then  for any $M>0$ $\exists$  $G_{j\mathbf{m}}$ such that  $G_{j\mathbf{m}}\in C^\infty (B_{\C^{n}}(0,\delta),\Sigma_r)$  $\forall$  $r$,  $z_jG_{j\mathbf 0}  (|z_j|^2)=a(0,\cdots,0,z_j,0,\cdots,0)$    and
\begin{equation}\label{eq:zj2}
a(z_1,\cdots,z_n)= \sum _{j=1}^{n}\sum_{|\mathbf{m} |\leq M-1}z_j\mathbf Z^{\mathbf m}G_{j\mathbf m}(|z_j|^2)+\mathcal S^{1,M}_{\infty , \infty}(z,\mathbf Z).
\end{equation}

\end{itemize}
\end{lemma}
\proof This elementary lemma is proved in   \cite{CM1}. \qed

The 3rd term in the r.h.s. of \eqref{eq:5termcompute} is dealt by the
following lemma.

\begin{lemma}\label{lem:expand3}
There exist     $ G_{2\textbf{m} i (2-i)} ( z )  $  as in the statement of Prop. \ref{prop:EnExp}     such that\eqref{eq:b02}  holds and
\begin{equation}\label{eq:lemexpand3}
\Re \<    \partial _{\eta}^2K(z,0)\eta,\eta ^*\>=\<H\eta,\eta ^*\>+ \mathcal R ^{1,  2  }_{\infty ,\infty} (z, \eta  )
   +    \sum _{ i  = 0}^{2}   \sum _{
		  |\textbf{m}|\le 1 }  \textbf{Z}^{\textbf{m}}    \langle G_{2\textbf{m} i (2-i)} ( z ),   \eta ^{  \otimes i} ({\eta}^*) ^{\otimes  (2-i)}\rangle.
\end{equation}
\end{lemma}

\proof
By Lemma \ref{lem:expKP} and
by claim (2) in Lemma \ref{lem:contcoo}  we have
\begin{equation*}
\begin{aligned}
\frac 1 2 \Re  \< \partial _\eta^2K(z,0)\eta,\eta\>=&\<HR[z]\eta,(R[z]\eta )^* \>+\langle g  ( \sum_{j,k} Q_{jz_j} \overline{Q}_{jz_j}  )R[z]\eta,(R[z]\eta )^* \rangle \\ +&
2\Re \langle g'(\sum_{j,k} Q_{jz_j} \overline{Q}_{jz_j})\mathrm{Re}\(\sum_{j=1}^nQ_{jz_j}\overline{R[z]\eta}\)
\sum_{j=1}^nQ_{jz_j},(R[z]\eta )^*\rangle \\=& \<H \eta, \eta  ^* \> +\mathcal R^{1,2}_{\infty , \infty}(z,\eta) \\  +\langle g  ( \sum_{j,k} Q_{jz_j} \overline{Q}_{kz_k}  ) \eta, \eta  ^* \rangle &  +
2\Re \langle g'(\sum_{j,k} Q_{jz_j} \overline{Q}_{kz_k})\mathrm{Re}\(\sum_{j=1}^nQ_{jz_j}\overline{ \eta}\)
\sum_{j=1}^nQ_{jz_j}, \eta  ^*\rangle.
\end{aligned}
\end{equation*}
The last line yields   the last term of  \eqref{eq:lemexpand3}  and it is straightforward to
see that it has the required properties.

\qed

To finish the discussion of the r.h.s. of \eqref{eq:5termcompute} we need
to   compute $R_P$.    We consider first  a preparatory lemma.

\begin{lemma}\label{lem:DR}
Let $l=1,\cdots,n$ and $A=R,I$.
Then we have
$(D_{lA}R[z])\eta=\sum_{j=1}^n\mathcal R^{1,1}_{\infty , \infty}(z,\eta)\phi_j$.
\end{lemma}
\begin{proof}
One has $R[z]\eta=\eta+\sum_{j=1}^n(\alpha_j[z]\eta)\phi_j$.
So $D_{lA}R[z]\eta=\sum_{j=1}^n\((D_{lA}\alpha_j[z])\eta\)\phi_j$.
Now, $D_{lA}\alpha[z]\eta=\int D_{lA}B_j(z)\eta\,dx+\int D_{lA}C_j (z) {\eta}^*\,dx$, where $B$ and $C$ are given in Lemma \ref{lem:contcoo} and we have $D_{lA}B_j(z)=\mathcal S^{1,0}_{\infty , \infty}$ and $D_{lA}C_j(z)=\mathcal S^{1,0}_{\infty , \infty}$.
This yields the lemma.
\end{proof}

We set $u:=\sum_{j=1}^n Q_{j  z_j}+  R[ z]\eta$.
Then, we have $D_{lA}u= D_{lA}Q_{l z_l}+  (D_{lA}R [ z])\eta$.
For $l=1,\cdots,n$ and $A=R,I$ we have
\begin{equation*}
\begin{aligned}
&D_{lA} \Re \<\(\partial ^3_{\eta} K_P ( z, \eta)\eta\)\eta, \eta ^*\>= 6 \mathscr{A}+4\mathscr{B}
\\& \mathscr{A} :=
D_{lA}\Re \langle  g'( u\overline{u}  )\mathrm{Re}\( \overline{R[ z]\eta}\) \beta R[ z]\eta, (R[ z]\eta )^*\rangle ,\\& \mathscr{B} :=
 D_{lA}\Re \langle   g''( u \overline{u})\(\mathrm{Re}\( u \overline{R[ z]\eta}\)\)^2\beta u ,(R[ z]\eta  )^*\rangle .
\end{aligned}
\end{equation*}
We have
\begin{equation*}\label{eq:expand41}
\begin{aligned}
  &\mathscr{A}  =
2\langle  g''( u \overline{u} ) \mathrm{Re}\( \overline{D_{lA}u}\)\mathrm{Re}\(u\overline{R[z]\eta}\) \beta R[z]\eta, (R[z]\eta  )^* \rangle\\&+
\langle  g '( u \overline{u}  )\mathrm{Re}\(D_{lA}u\overline{R[z]\eta}\) \beta R[z]\eta, (R[z]\eta  )^* \rangle +
\langle  g '( u \overline{u}  ) \mathrm{Re}\( u\overline{(D_{lA}R[z])\eta}\)\beta R[z]\eta, (R[z]\eta  )^* \rangle \\& +2
\langle  g '( u \overline{u}  ) \mathrm{Re}\(u\overline{R[z]\eta}\)\beta (D_{lA}R[z])\eta, (R[z]\eta  )^*\rangle ,
\end{aligned}
\end{equation*}
\begin{equation*}\label{eq:expand42}
\begin{aligned}
  &\mathscr{B} =
2\langle g'''(u \overline{u} )\mathrm{Re}(u\overline{D_{lA}u})\(\mathrm{Re}\( u \overline{R[z]\eta}\)\)^2 \beta u ,(R[z]\eta  )^* \rangle \\&+2\langle  g''(u \overline{u} )\mathrm{Re}\( u \overline{R[z]\eta}\)\mathrm{Re}\(D_{lA}u \overline{R[z]\eta}\)\beta  u ,(R[z]\eta  )^* \rangle \\&+2\langle g''(u \overline{u} )\mathrm{Re}\( u  \overline{R[z]\eta}\)\mathrm{Re}\( u \overline{(D_{lA}R[z])\eta}\) u ,(R[z]\eta  )^* \rangle \\&+
\langle g''(u \overline{u} )\(\mathrm{Re}\( u \overline{R[z]\eta}\)\)^2\beta  D_{lA} u ,(R[z]\eta  )^* \rangle  +
\langle g''(u \overline{u} )\(\mathrm{Re}\( u\overline{R[z]\eta}\)\)^2 \beta u ,((D_{lA}R[z] )\eta )^*\rangle.
\end{aligned}.
\end{equation*}
All the terms in  the formulas for $\mathscr{A}$ and $\mathscr{B}$ can be expressed as
\begin{equation}\label{eq:R_P}
\sum _{  d  = 2, 3 }   \sum _{ i+j  = d}   \langle G_{d ij } ( z, \eta ),   \eta ^{ \otimes  i} \otimes ({\eta ^*})  ^{\otimes  j}\rangle  \mathcal R ^{2,  3-d }_{\infty ,\infty} (z, \eta )  +  \mathcal R ^{2,  2  }_{\infty ,\infty}(z, \eta  ).
\end{equation}
Therefore $R_P$ admits an expansion of the form \eqref{eq:R_P}, which   is absorbed in terms of the r.h.s. of \eqref{eq:enexp1}.

\proof[Proof of Proposition \ref{prop:EnExp}]
We have just seen that $R_P$ is absorbed   in  the r.h.s. of \eqref{eq:enexp1}.
The other terms of the r.h.s. of \eqref{eq:5termcompute} are treated by
Lemmas \ref{lem:1exp} and \ref{lem:expand3}.
\qed

\section{Effective Hamiltonian}
\label{sec:darboux}

In this section we apply the theory of Sect. 4 and 5 in \cite{CM1}
which yields an effective Hamiltonian  in an appropriate coordinate system,
which in turn will be used to prove Theorem \ref{thm:mainbounds} which yields
Theorem \ref{thm:small en}. Finding an effective Hamiltonian entails  canceling as many terms
as possible from the 2nd line of \eqref{eq:enexp1} through appropriate changes of
variables. This process is called Birkhoff normal form  argument and    is done by means of a recursive procedure where each time we need to cancel
a term from the hamiltonian we find an appropriate coordinate change by first solving
an  equation, the \textit{homological} equation.   It is easier
to implement this procedure using coordinates which are Darboux. In  a finite dimensional setting this would mean that  the symplectic form
is equal to a simple model, like   $\omega _0:=\sum _j\im dz_{j  }\wedge d\overline{z}_{j  }$. This corresponds to
diagonalizing the  {homological} equations. Furthermore, it is important that the new coordinates
remain Darboux. This means that the change of coordinates should leave $\omega _0$ invariant.
 One way to do this is to make changes of coordinates using    flows  of hamiltonian
 vectorfields.
See for example Sect. 1.8 \cite{Hofer} for a general introduction to the subject.

\noindent The   system  \eqref{eq:NLSham}    is   Hamiltonian   with respect to the symplectic form
\begin{equation}
\label{eq:Omega} \Omega (X,Y
):=   -2\Im \langle  {X} , Y^*\rangle .
\end{equation}
The first thing to notice  is  that the coordinates in Lemma \ref{lem:systcoo},
  initially the most natural coordinates in our problem,
  do not form a system
of Darboux coordinates for \eqref{eq:Omega} in any reasonable sense. Indeed $\Omega$  is
rather complicated in this coordinate system.

 \noindent We consider as a local model  the  symplectic form
\begin{equation}\label{eq:defOm0}
\begin{aligned} &
\Omega_0  :=  \im \sum_{j=1}^n (1+ \gamma _j (|z_j|^{2}))  dz_{j  }\wedge d\overline{z}_{j  }  + \im \<d \eta  , d {\eta}^*  \>   -\im \<d   {\eta}^*  , d {\eta}  \>,    \\&
  \text{where }     \gamma _j(|z_j|^2)
:=-
 \<  \widehat{{q}}_{j  } (|z_j|^2), \widehat{{q}}_{j  } ^*  (|z_j|^2) \> +  2|z_j|^2   \Re \<  \widehat{{q}}_{j  }^* (|z_j|^2), \widehat{{q}}_{j  } ' (|z_j|^2) \>,
\end{aligned}
\end{equation}
 with $\widehat{{q}}_{j  } '  (t)= \frac{d}{dt}\widehat{{q}}_{j  }(t) $.
 By Proposition \ref{prop:bddst}  and Definition \ref{def:scalSymb} we have $\gamma _j(|z_j|^2) =\resto ^{2,0}_{\infty,\infty}(|z_j|^2)$.

 \begin{remark}\label{rem:remOm0}    $\Omega_0$  is the same  {local model} symplectic form  of \cite{CM1}. We do not know if
 Proposition \ref{prop:darboux}  below   holds when choosing $\gamma _j\equiv 0$ because  we do not know if in \eqref{eq:syst11} below  we would still have   $S _{j  }  =\mathcal{R}^{1,1}_{r, \infty} $   and
  $ S _{\eta}   =\textbf{S}^{1,1}_{r, \infty} $, which is crucial. Notice, incidentally, that the Darboux Theorem is an abstract result. But   in \cite{CM1}
   it is proved with an \textit{ad hoc} argument exactly because we want this change of coordinates,
    as well as all the other coordinates changes in   the paper,
   to have this crucial property. The fact that all coordinate changes have this property guarantees that the limits \eqref{eq:small en1} in one coordinate system  imply the same limit in any other coordinate system.

   While $\Omega_0$  would be simpler if $\gamma _j\equiv 0$, nonetheless  it is simple enough
   for  a normal form argument.
  \end{remark}
 In Sect. 4 \cite{CM1}   the following proposition is proved.
\begin{proposition}[Darboux Theorem]\label{prop:darboux} Fix any $r\in \N$.  There exists a    $\delta _0 \in (0, d_0) $
such that the following facts hold.
\begin{itemize} \item[(1)]
 There exists a gauge invariant  1--form
  $\Gamma = \Gamma _{jA}  dz_{jA} + \langle \Gamma _{\eta}, d\eta \rangle  + \langle \Gamma _{ {\eta}^*}, d {\eta}^* \rangle,
   $  with
 \begin{align} \label{eq:alphaest1}
  &          \Gamma _{jA}=   \mathcal{R}^{1,1}_{\infty, \infty}(z,\textbf{Z},\eta )  \text{ and } \Gamma _{\xi }= \mathbf{S} ^{1,1}_{\infty, \infty}(z,\textbf{Z},\eta )  \text{ for $\xi = \eta , {\eta} ^* $,}
\end{align}
 such that $d\Gamma =\Omega - \Omega _0 $.

\item[(2)]   For any     $(t, z, \eta )\in (-4, 4)\times B_{\C ^n} (0,\delta _0 )
    \times   B_{\Sigma _{-r}^c} (0,\delta _0 )$    there exists
  exactly one solution
    $\mathcal{X}^t  (z, \eta )\in L^2$ of the  equation     $
 i_{\mathcal{X}^t} \Omega _t=-  \Gamma
   $. Furthermore,  $\mathcal{X}^t  (z, \eta )$ is gauge invariant, $\mathcal{X}^t  (z, \eta )\in \Sigma _{ r}$ and if we set $\mathcal{X}^t _{jA} (z, \eta )=dz_{jA}\mathcal{X}^t  (z, \eta )$ and  $\mathcal{X}^t _{\eta} (z, \eta )=d\eta \mathcal{X}^t  (z, \eta )$,   we have
  $\mathcal{X}^t _{jA} (z, \eta )=\mathcal{R}^{1,1}_{r,\infty }(t,z,\textbf{Z}, \eta )$ and
  $ \mathcal{X}^t _{\eta} (z, \eta )=\textbf{S}^{1,1}_{r,\infty}(t,z,\textbf{Z},\eta )$.

  \item[(3)] Consider
         the   following system  in
     $(t, z, \eta )\in (-4, 4)\times B_{\C ^n} (0,\delta _0 )
    \times   B_{\Sigma _{ k}^c} (0,\delta _0 )$  for all  $k\in \Z\cap [-r,r]$:
\begin{equation}\label{eq:syst1}
\begin{aligned}
  &    \dot z_{j } = \mathcal{X}^t _{j }  (z, \eta ) \text{  and }     \dot \eta = \mathcal{X}^t _{\eta}  (z, \eta ) .
     \end{aligned}
\end{equation}
Then the following facts hold  for the corresponding flow $\mathfrak{F}^t$.\begin{itemize}
\item[(3.1)]  For   $\delta _1\in (0,\delta _0 ) $  sufficiently small we have
\begin{align} \nonumber
  &   \mathfrak{F}^t \in C ^\infty ((-2, 2)\times B_{\C ^n} (0,\delta _1 )
    \times   B_{\Sigma _{ k}^c} (0,\delta _1 ) ,B_{\C ^n} (0,\delta _0 )
    \times   B_{\Sigma _{ k}^c} (0,\delta _0 )) \text{ for all $k\in \Z\cap [-r,r]$} \\&  \mathfrak{F}^t \in C ^\infty ((-2, 2)\times B_{\C ^n} (0,\delta _1 )
    \times   B_{H^1 \cap \mathcal{H}_{c}[0]} (0,\delta _1 ) ,B_{\C ^n} (0,\delta _0 )
    \times   B_{H^1 \cap \mathcal{H}_{c}[0]} (0,\delta _0 ) ) .\label{eq:syst122}
     \end{align}
 In particular we have
\begin{equation}\label{eq:syst11}
\begin{aligned}
  &      z_{j  }  ^{t}= z_{j }+ S _{j  }  (t,z, \eta ) \text{  and }       \eta  ^t= \eta+ S _{\eta}  (t,z, \eta ),\end{aligned}
\end{equation}
with
  $S _{j  } (t,z, \eta )=\mathcal{R}^{1,1}_{r, \infty}(t,z,\textbf{Z}, \eta )$   and
  $ S _{\eta}  (t,z, \eta )=\textbf{S}^{1,1}_{r, \infty}(t,z,\textbf{Z}, \eta )$.

 \item[(3.2)]      The map $\mathfrak{F}=\mathfrak{F}^1$ is a local diffeomorphism
 of $H^1$ into itself near  the origin, and we have $\mathfrak{F}^*\Omega =\Omega _0$.

  \item[(3.3)]      We have  $S _{j }  (t,e^{\im \vartheta}z,e^{\im \vartheta} \eta )=e^{\im \vartheta}S _{j }  (t,z, \eta )$,
 $S _{\eta }  (t,e^{\im \vartheta}z,e^{\im \vartheta} \eta )=e^{\im \vartheta}S _{\eta }  (t,z, \eta )$.
 \end{itemize}

 \end{itemize}

\end{proposition}
  \qed

We now consider the pullback $K:=E\circ   \mathfrak{F} $.
\begin{lemma}
  \label{lem:KExp1}    Consider the $\delta _1>0$ and    $\delta _0>0$ of Prop. \ref{prop:darboux}
   and set      $r_0=r $ with  $r  $  the index   in Prop. \ref{prop:EnExp}.   Then we have
   \begin{equation}\label{eq:KExp11}
\begin{aligned} &
   \mathfrak{{F}}(B_{\C^n}(0 , \delta  _1) \times ( B_{H^1}(0 , \delta  _1)\cap \mathcal{H}_c[0]) )\subset  B_{\C^n}(0 , \delta   _0) \times ( B_{H^1}(0 , \delta  _0)\cap \mathcal{H}_c[0])
\end{aligned}
\end{equation}
   and $  \mathfrak{{F}}|_{B_{\C^n}(0 , \delta  _1) \times ( B_{H^1}(0 , \delta  _1)\cap \mathcal{H}_c[0])} $ is a diffeomorphism between domain and   an open neighborhood of the origin in
   $\C^n \times (H^1 \cap \mathcal{H}_c[0])$ and furthermore
the functional  $K  $ admits an expansion for $r_1=r_0-2$
\begin{align} \nonumber
  & K (z,\eta )=H_2(z,\eta )  + \sum _{j=1,...,n} \lambda _j(|z_j| ^{2}  )
 \\&   \nonumber  +
	 	 \sum _{l=1}^{2N+4 } \sum _{
		  |\textbf{m}|=   l+1 }  \textbf{Z}^{\textbf{m}}   a_{  \textbf{m} }^{(1)}(|z _{1}|^2,..., |z _{n}|^2 )+ \sum _{j =1}^n \sum _{l=1}^{ 2N+3 }	\sum _{   |\textbf{m}  |  =l} ( \overline{z}_j \textbf{Z}^{\textbf{m}}  \langle
  G_{j \textbf{m}}^{(1)}(|z_j|^2  ), \eta \rangle  +   c.c. )
		\\&  \nonumber
	 	  +     \resto ^{1,  2  }_{{r_1}, \infty} (z , \eta  )  +
	 	  \resto ^{0,  2N+5 }_{{r_1}, \infty} (z,\textbf{Z}  ,\eta )+   \Re \langle
  \textbf{S} ^{0,   2N+4 }_{{r_1}, \infty} (z,\textbf{Z},\eta ) , {\eta} ^{*} \rangle \nonumber
		\\&
		+       \sum _{  d  = 2  }^{3}   \sum _{ i  = 1}^d \mathcal R ^{0,  3-d }_{r_{0} , \infty} (z, \eta ) \int _{\R ^3} G_{d i }^{(1)} ( x,z, \eta , \eta (x) )    \eta ^{  \otimes i}(x)\otimes ({\eta}^*(x)) ^{\otimes (d-i)} dx     \nonumber \\&  +
 \sum _{ i+j  = 2}   \sum _{
		  |\textbf{m}|\le 1 }  \textbf{Z}^{\textbf{m}}    \langle G_{2\textbf{m} ij }^{(1)} ( z ),   \eta ^{  \otimes i}\otimes ({\eta}^*) ^{\otimes j}\rangle  + E _P( \eta), \   \label{eq:enexp10}
  \end{align}
  where
  \begin{equation*}
   H_2(z,\eta ) = \sum _{j=1}^n
	e_j |z_j|^2+  \langle   H  \eta,  { \eta}^* \rangle
  \end{equation*}
    and  where:    $G_{j \textbf{m}}^{(1)}$,     $G^{(1)}_{2\textbf{m} ij }    $     are   $\textbf{S} ^{0,  0} _{{r_1}, \infty }  $;   $a_{  \textbf{m}}^{(1)}(|z _{1}|^2,..., |z _{n}|^2 )  =\resto ^{0,  0} _{\infty, \infty  }  (z) $;     c.c. means complex conjugate;  $\lambda _j(|z_j| ^{2}  )
=\resto ^{2,  0  }_{\infty, \infty} (|z_j| ^{2}  )$;    $$G_{d i }^{(1)}( \cdot ,z, \eta , \zeta ) \in C^{\infty } ( B_{\C ^n}(0,\delta _1) \times \Sigma _{-r_{1}} (\mathbb{R}^3, \C)\times \C^4,
\Sigma _{r_{1}} (\mathbb{R}^3, B^{d}(\C ^4,\mathbb{C})   ))) .$$ For $|\textbf{m}|=0$,  $G^{(1)} _{2\textbf{m} ij }( z,\eta  ) = G_{2\textbf{m} ij }( z   ) $ is the same of \eqref{eq:b02}.
Finally, we have the invariance  $\resto ^{1,  2  }_{r_{1} , \infty} (e^{\im \vartheta} z, e^{\im \vartheta} \eta )\equiv \resto ^{1,  2  }_{r_{1} , \infty} (  z,   \eta )  $.

\end{lemma}
\proof The proof of the above statement with possibly nonzero  terms also corresponding to $l=0$ in both summations in the 2nd line of \eqref{eq:enexp10} is
elementary, see Lemma 4.10 in \cite{CM1} and Lemma 4.3 \cite{Cu0}.

\noindent The key fact that in the 2nd line of \eqref{eq:enexp10}   both summations start from $l=1$  and there are no $l=0$ is proved in the cancellation Lemma 4.11 in \cite{CM1}

\qed

Consider now the    symplectic form $\Omega _0$   in \eqref{eq:defOm0}.
 We introduce an index $\ell =j, \overline{j}$, for $\overline{\overline{j}}=j$
  with $j=1,...,n$. We write $\partial _j=\partial _{z_j}$ and  $\partial _{ \overline{j}}=\partial _{\overline{z}_j}$, $z_{\overline{j}}=\overline{z}_j$.
Given   $F\in C^1(U,\C )$
with $U$ an open subset of $\C^n\times \Sigma _{r}^c$,  its Hamiltonian vector field $X_F$ is defined by
$i_{X_F}\Omega _0=dF$.     We have summing on $j$
\begin{equation}  \label{eq:hamf-1}\begin{aligned}        &      i_{X_F}\Omega _0=    \im     (1 +   \gamma _j(|z_j|^2) ) ((X_F) _{j }   d\overline{z}_{j } -(X_F) _{\overline{j} }   d {z}_{j })  + \im \<(X_F) _{\eta  }  , d\overline{\eta}  \>   -\im \<(X_F) _{\overline{\eta} }    , d {\eta}  \>   \\&  =  \partial _j F    d {z}_j   +     \partial _{\overline{j} } F  d\overline{z}_{j }    +\<\nabla _{\eta  } F , d {\eta}  \> +  \<\nabla _{ {\eta}^*  } F   , d {\eta}^*  \>    ,
\end{aligned}\end{equation}
where \eqref{eq:hamf-1} is used also to define $\nabla _{\xi  } F$ for $\xi = \eta , \eta ^*$.

Comparing the components of the two sides of \eqref{eq:hamf-1}  we get   for $1 +   \varpi _j(|z_j|^2)=  (1 +   \gamma _j(|z_j|^2) )^{-1}$  where    $\varpi  _j(|z_j|^2) =\resto ^{2,0}_{\infty, \infty }(|z_j|^2)$
:
\begin{equation} \label{eq:hamf1}\begin{aligned}        &    (X_F)_j    =-\im   (1 +   \varpi  _j(|z_j|^2))
 \partial _{\overline{j}} F   \ ,  \quad    (X_F)_{\overline{j}}   = \im   (1 +  \varpi  _j(|z_j|^2)) \partial _j F ,         \\&    (X_F)_\eta   =-\im    \nabla  _{{\eta}^* } F  \,  , \quad     (X_F)_{{\eta}^* }  =  \im    \nabla  _\eta F .
\end{aligned}\end{equation}

Given    $G\in C^1(U,\C )$  and   $F\in C^1(U, \textbf{E} )$,  with $\textbf{E}$ a Banach space, we set
$  \{   F, G  \}  :=dF X_G$.

\begin{definition}[Normal Forms]
\label{def:normal form}  Recall Definition \ref{def:setM}  and    \eqref{eq:setM0}.
Fix
$r\in \N _0$.
A real valued function $Z( z,\eta  )$ is in normal form if  $
  Z=Z_0+Z_1
$
with $Z_0$ and $Z_1$   finite sums of the following type,  for $\mathbf l \ge 1$:  for $ G_{j \textbf{m}}( |z_j|^2 ) =S^{0,0}_{r,\infty}(|z_j|^2  )$,   c.c. the complex conjugate and $ a_{  \textbf{m} }( |z _1|^2,...,|z _n|^2  )  =\resto ^{0,0}_{r,\infty}(|z _1|^2,...,|z _n|^2  )$,
\begin{equation}
\label{e.12a}Z_1 (z,\mathbf{Z},\eta )=      \sum _{j =1}^n	\sum _{   \substack{|\textbf{m}  |   = \mathbf l \\ \mathbf{m}  \in \mathcal{M}_{j}(\mathbf l)} } \(\overline{z}_j \textbf{Z}^{\textbf{m}}  \langle
  G_{j \textbf{m}}(|z_j|^2  ), \eta \rangle  +   \text{c.c.}\), \
\end{equation}
 \begin{equation}
\label{e.12c}Z_0 (z,\mathbf{Z},\eta ) =  	\sum _{   \substack{|\textbf{m}  |   = \mathbf l +1\\ \mathbf{m}  \in \mathcal{M}_{0}(\mathbf l+1)} } \textbf{Z}^{\textbf{m}}   a_{  \textbf{m} }( |z _1|^2,...,|z _n|^2 ) . \end{equation}

 \end{definition}

\begin{remark}\label{rem:H3}
  By Hypothesis (H4), in particular by \eqref{h42},  for any $\textbf{m} \in  \mathcal{M}_{0}(2N+4)$  we have $\textbf{Z}^{\textbf{m}} =|z_1| ^{2m_1}...|z_n| ^{2m_n} $ for an $m \in  \N  _0^n$  with $2|m|=|\mathbf{m}| $.
Similarly  by (H4), in particular by \eqref{h41},    for $|\textbf{m}  |  \le  2N+4$ we   have $|\sum _{a,b}(e_a-e_b)  {m}_{ab} -e_j |\neq M$.
\end{remark}

For $ \mathbf{l} \le 2N+4$ we will consider   flows   associated to   Hamiltonian   vector-fields $X_\chi$
with real valued
   functions  $\chi $ of the following form,  with    $ b_{  \textbf{m} }  =\resto ^{0,0}_{\mathbf r ,\infty }(|z _1|^2,...,|z _n|^2  )$  and    $ B_{j \textbf{m}} =S^{0,0}_{\mathbf{r},\infty}(|z_j|^2  )$  for some $\mathbf{r}\in \N $ defined in $B _{\C^n} (0, \mathbf{d})$ for some $\mathbf{d}>0$:
\begin{equation} \label{eq:chi1}\begin{aligned}
\chi   &=     	\sum _{   \substack{|\textbf{m}  |   = \mathbf{l} +1\\   \mathbf{{m}}
 \not \in \mathcal{M} _0 (\mathbf{l}+1) } } \textbf{Z}^{\textbf{m}}   b_{  \textbf{m} }( |z _1|^2,...,|z _n|^2 )    +
 \sum _{j =1}^n	\sum _{   \substack{|\textbf{m}  |   = \mathbf{l} \\ \mathbf{{m}}
 \not \in \mathcal{M} _j (\mathbf{l}) } } (\overline{z}_j \textbf{Z}^{\textbf{m}}  \langle
  B_{j \textbf{m}}(|z_j|^2  ), \eta \rangle  +   \text{c.c.})   \ .
\end{aligned}\end{equation}

The following result is proved in \cite{CM1}

\begin{proposition}[Birkhoff normal forms]
\label{th:main} For any $\iota \in \N \cap [2,  2 {N}+4]$
there are a $\delta _\iota >0$,   a  polynomial
$\chi _{\iota}$  as in \eqref{eq:chi1}    with   $\mathbf{l}=\iota$,  $\mathbf{d}=\delta _\iota $ and   $\mathbf{r}= {r_{\iota}} =r_0- 2(\iota +1)$ such that
for all $k\in \Z\cap [-r(\iota),r(\iota)]$ we have for each $\chi _{\iota}$
a flow (for $\delta _1>0$ the constant in Prop. \ref{prop:darboux})
 \begin{align}
  &   \phi ^t_{\iota} \in C ^{\infty} ((-2, 2)\times B_{\C ^n} (0,  \delta _\iota  )
    \times   B_{\Sigma _{ k}^c} (0,\delta _\iota  ) ,B_{\C ^n} (0,\delta _{\iota -1 } )
    \times   B_{\Sigma _{ k}^c} (0,\delta _{\iota -1 }    )) , \label{eq:bsyst122}\\&   \phi ^t_{\iota} \in C ^{\infty} ((-2, 2)\times B_{\C ^n} (0,\delta _\iota  )
    \times   B_{H^1 \cap \mathcal{H}_{c}[0]} (0,\delta _\iota  ) ,B_{\C ^n} (0,\delta _{\iota -1 }   )
    \times   B_{H^1 \cap \mathcal{H}_{c}[0]} (0,\delta _{\iota -1 }   ) )   \nonumber
     \end{align}
 and  such that, if we set  $\mathfrak{F} ^{( \iota )} := \mathfrak{F}  \circ \phi _2\circ ...\circ \phi _ \iota   $, with   $\mathfrak{F} $  the transformation in
Prop.  \ref{prop:darboux} and
the $\phi _j=  \phi ^1_{\iota}   $, then
  for $(z,\eta )\in  B_{\C^n}(0 , \delta _\iota) \times ( B_{H^1}(0 ,\delta _\iota  )\cap \mathcal{H}_c[0])$ we have  the following expansion   \begin{equation}  \label{eq:enexp40}    \begin{aligned} &
   H^{(\iota )} (z,\eta ):= E\circ \mathfrak{F} ^{( \iota )} (z,\eta )=H_2(z,\eta )  +\sum _{j=1}^{n} \lambda _j (|z_j| ^{2}  )+ Z  ^{(\iota )} (z,\mathbf{Z},\eta ) \\& +
	 	 \sum _{l=\iota  }^{2N+3 } \sum _{
		  |\textbf{m}|=   l+1 }  \textbf{Z}^{\textbf{m}}   a_{  \textbf{m} }^{(\iota )}( |z_1 |^2,...,|z_n |^2 )     + \sum _{j =1}^n \sum _{l=\iota  }^{ 2N+3 }	 \sum _{   |\textbf{m}  |  =l} ( \overline{z}_j \textbf{Z}^{\textbf{m}}  \langle
  G_{j \textbf{m}}^{(\iota )}(|z_j|^2  ), \eta \rangle  +   c.c. )
		\\&          +
	 	  \resto ^{1,  2  }_{{r_{\iota}},\infty} (z, \eta )+
	 	  \resto ^{0,  2N+5 }_{{r_{\iota}},\infty} (z,\textbf{Z}  ,\eta )+   \Re \langle
  \textbf{S} ^{0,   2N+4 }_{{r_{\iota}},\infty} (z,\textbf{Z},\eta ) , {\eta} ^* \rangle +
		\\&
		+       \sum _{  d  = 2  }^{3}   \sum _{ i  = 1}^d \mathcal R ^{0,  3-d }_{r_{0} , \infty} (z, \eta ) \int _{\R ^3} G_{d i }^{(\iota )} ( x,z, \eta , \eta (x) )    \eta ^{  \otimes i}(x)\otimes ({\eta}^*(x)) ^{\otimes (d-i)} dx     \\&  +
 \sum _{ i+j  = 2}   \sum _{
		  |\textbf{m}|\le 1 }  \textbf{Z}^{\textbf{m}}    \langle G_{2\textbf{m} ij }^{(\iota )} ( z ),   \eta ^{  \otimes i}\otimes ({\eta}^*) ^{\otimes j}\rangle  + E _P( \eta) ,\end{aligned}
\end{equation}
where, for coefficients like in  Definition \ref{def:normal form}  for $(r ,m )=({r_{\iota}},\infty)$,
\begin{equation}  \label{eq:enexp41}
\begin{aligned}
&   Z ^{(\iota )} = \sum _{   \mathbf{m}  \in \mathcal{M}_{0}(\iota )  } \textbf{Z}^{\textbf{m}}   a_{  \textbf{m} }( |z _1|^2,...,|z _n|^2 )    + \sum _{j =1}^n	 ( \sum _{   \mathbf{m}  \in \mathcal{M}_{j}(\iota- 1)  } \overline{z}_j \textbf{Z}^{\textbf{m}}  \langle
  G_{j \textbf{m}}(|z_j|^2  ), \eta \rangle  +   \text{c.c.} ).
     \end{aligned}
\end{equation}
We have
$\resto ^{1,  2 }_{ r _{\iota}, \infty }  = \resto ^{1,  2 }_{ r _{2}, \infty }
 $ and
 $\resto ^{1,  2 }_{ r _{2}, \infty }
 (e^{\im \vartheta} z , e^{\im \vartheta} \eta )\equiv \resto ^{1,  2 }_{ r _{2}, \infty } (  z ,   \eta )  $.

 In particular we have for $\delta  _f:=  \delta  _{2 {N}+4} $ and for the $\delta _0$
 in Prop. \ref{prop:darboux},
   \begin{equation}\label{eq:diff1}
\begin{aligned} &
  \mathcal{F}^{(2 {N}+4)}(B_{\C^n}(0 , \delta  _f) \times ( B_{H^1}(0 , \delta  _f)\cap \mathcal{H}_c[0]) )\subset  B_{\C^n}(0 , \delta  _0) \times ( B_{H^1}(0 , \delta  _0)\cap \mathcal{H}_c[0])
\end{aligned}
\end{equation}
   with $ \mathcal{F}|_{B_{\C^n}(0 , \delta  _f) \times ( B_{H^1}(0 , \delta  _f)\cap \mathcal{H}_c[0])} $   a diffeomorphism between its domain and   an open neighborhood of the origin in
   $\C^n \times (H^1 \cap \mathcal{H}_c[0])$.

   Furthermore, for $r=r_0- 4N-10$
   there is a pair $\mathcal{R}^{1, 1} _{r, \infty}$ and  $\mathbf{S}^{1, 1} _{r, \infty}$ such that
for $(z',\eta ')=\mathcal{F}^{(2 {N}+4)}(z ,\eta  )$
we have
\begin{equation}\label{eq:diff2}
\begin{aligned} &  z'=z+ \mathcal{R}^{1, 1} _{r, \infty}(z,\mathbf{Z},\eta ), \, \quad  \eta'= \eta+ \mathbf{S} ^{1, 1} _{r, \infty}(z,\mathbf{Z},\eta ) .
\end{aligned}
\end{equation}
Furthermore, by taking  all the  $\delta  _\iota >0$ sufficiently small, we can assume that all the symbols in the proof, i.e. the symbols in \eqref{eq:diff2}
and the symbols in the expansions \eqref{eq:enexp40}, satisfy the estimates
of Definitions \ref{def:scalSymb} and \ref{def:opSymb} for $|z|<\delta  _\iota $ and $\|  \eta \| _{\Sigma_{-r(\iota)}}<\delta  _\iota$ for their respective
$\iota$'s.

\end{proposition}
\qed

\section{Dispersion} \label{sec:disp}

We apply Proposition \ref{th:main},  set $  \mathcal{H} = H^{(2N+4 )}$ so that for some $r\in \N$ which we can take arbitrarily large,
\begin{equation}  \label{eq:enexp401}    \begin{aligned} &
  \mathcal{H} (z,\eta )=H_2(z,\eta )  +\sum _{j=1}^{n} \lambda _j (|z_j| ^{2}  ) + \mathcal{Z}  (z,\mathbf{Z},\eta )   +\resto , \end{aligned}
\end{equation}
with $ \mathcal{Z}  (z,\mathbf{Z},\eta )=  \mathcal{Z}^{(2N+4 )}  (z,\mathbf{Z},\eta )     $
and
\begin{equation}  \label{eq:rest1}  \begin{aligned} & \resto =
	 	  \resto ^{1,  2  }_{{r_{\iota}},\infty} (z, \eta )+
	 	  \resto ^{0,  2N+5 }_{{r_{\iota}},\infty} (z,\textbf{Z}  ,\eta )+   \Re \langle
  \textbf{S} ^{0,   2N+4 }_{{r_{\iota}},\infty} (z,\textbf{Z},\eta ) , {\eta} ^* \rangle +
		\\&
		+       \sum _{  d  = 2  }^{3}   \sum _{ i  = 1}^d \mathcal R ^{0,  3-d }_{r_{0} , \infty} (z, \eta ) \int _{\R ^3} G_{d i }^{(\iota )} ( x,z, \eta , \eta (x) )    \eta ^{  \otimes i}(x)\otimes ({\eta}^*(x)) ^{\otimes (d-i)} dx     \\&  +
 \sum _{ i+j  = 2}   \sum _{
		  |\textbf{m}|\le 1 }  \textbf{Z}^{\textbf{m}}    \langle G_{2\textbf{m} ij }^{(\iota )} ( z ),   \eta ^{  \otimes i}\otimes ({\eta}^*) ^{\otimes j}\rangle  + E _P( \eta) . \end{aligned}
\end{equation}
Our ambient space is $H^{4} (\R ^3, \C ^4)$.  So under (H1) the functional $u\to
g(u\overline{u}) \beta  u$ is locally Lipschitz and    \eqref{eq:NLS},
 \eqref{eq:NLSham} and the equivalent  system  with Hamiltonian  $\mathcal{H} (z,\eta )$ and symplectic form $\Omega _0$,
are locally
well posed, see pp. 293--294 volume III  \cite{taylor}.

\noindent By standard arguments, see  \cite{CM1},  Theorem \ref{thm:mainbounds} below implies Theorem \ref{thm:small en}.
      \begin{theorem}[Main Estimates]\label{thm:mainbounds} Consider the    $ \epsilon   $  of Theorem \ref{thm:small en}.
      Then   there exists  $\epsilon _0>0$     and a
$C _{0}>0$ such that  if $ \epsilon <\epsilon _0$ then   for $I= [0,\infty )$     we have
the following inequalities:
\begin{align}
&   \|  \eta \| _{L^p_t( I,B^{ 4 -\frac{2}{p}} _{q,2} (\R ^3, \C ^4))}\le
  C   \epsilon ,\text{ for all the pairs $(p,q)$ as of \eqref{eq:numbers1}},
  \label{Strichartzradiation} \\& \|  \eta \| _{L^2_t( I,H^{ 4 ,-10}(\R ^3, \C ^4))}\le
  C \epsilon,
  \label{Smootingradiation}\\& \|  \eta \| _{L^2_t( I,L^{ \infty}(\R ^3, \C ^4))}\le
  C \epsilon,
  \label{endStrich}
\\& \| z_j \mathbf{Z}  ^{\mathbf{m}} \| _{L^2_t(I)}\le
  C   \epsilon, \text{ for all   $(j,\mathbf{m})$
  with  $\mathbf{m} \in \mathcal{M}_j (2N+4)  $,} \label{L^2discrete}\\& \| z _j  \|
  _{W ^{1,\infty} _t  (I )}\le
  C   \epsilon, \text{ for all   $j\in \{ 1, \dots ,  {n}\}$ } \label{L^inftydiscrete}
   .
\end{align}
Furthermore,  there exists  $\rho  _+\in [0,\infty )^n$ such that  there exist  a   $j_0$  with $\rho_{+j}=0$ for $j\neq j_0$
and   there exists $\eta _+\in L^\infty $   such that
  $| \rho  _+   | \le C   \epsilon $ and   $\|  \eta _+\| _{L^\infty }\le C    \epsilon $  such that
\begin{align}&     \lim_{t\to +\infty}\| \eta (t,x)-
e^{-\im tD_{\mathscr{M}} }\eta  _+ (x)   \|_{L^\infty _x(\R ^3, \C ^4)}=0 ,\label{eq:small en31} \\&
  \lim_{t\to +\infty} |z_j(t)|  =\rho_{+j}  .\label{eq:small en32}
\end{align}

\end{theorem}

By an elementary continuation argument (see \cite{boussaidcuccagna, CM1} or \cite{So}, end of the proof of Theorem 2.1, Sect. II), the estimates
\eqref{Strichartzradiation}--\eqref{L^inftydiscrete} for  $I= [0,\infty )$ are a consequence of the following proposition.

\begin{proposition}\label{prop:mainbounds} There exist  a  constant $c_0>0$  such that
for any  $C_0>c_0$ there is a value    $\epsilon _0= \epsilon _0(C_0)   $ such that   if   the inequalities  \eqref{Strichartzradiation}--\eqref{L^inftydiscrete}
hold  for $I=[0,T]$ for some $T>0$, for $C=C_0$  and for $0< \epsilon < \epsilon _0$,
then in fact for $I=[0,T]$  the inequalities  \eqref{Strichartzradiation}--\eqref{L^inftydiscrete} hold  for   $C=C_0/2$.
\end{proposition}
\proof  By Lemma \ref{lem:conditional4.2}, there exists a fixed $c_1>0$ such that given any
$C_0$ if $\epsilon _0>0$  is small enough we have for all admissible pairs $(p,q)$, for the ${M}$ of Definition \ref{def:setM} and for a preassigned $\tau _0>1$,
\begin{equation}\label{4.5}
    \|   \eta  \| _{L^p_t( [0,T ],B^{  4-\frac{2}{p}} _{q,2})\cap L^2_t([0,T],H^{  4 ,-\tau _0}_x)\cap  L^2_t( [0,T],L^{
\infty}_x)}\le   c_1  \epsilon   + c_1    \sum _{(\mu , \nu )\in {M}  }|   z ^{\mu}\overline{z}  ^{\nu}  | _{L^2_t( 0,T  )}.
\end{equation}
The aim of Sect.  \ref{subsec:prop} is to prove that   there exists a fixed $c_2>0$ such that if $\epsilon _0>0$ for any  given any
$C_0$ if $\epsilon _0>0$  is small enough we have \begin{equation} \label{eq:crunch0}\begin{aligned}&  \sum _j   \| z
_j \|^2_{L^\infty_t( 0,T  )}  +\sum _{ ( \mu , \nu )
   \in M}  \| z^{\mu +\nu } \| _{L^2(0,T)}^2\le c_2
\epsilon ^2+ c_2 C_0\epsilon ^2.\end{aligned}
\end{equation}
This implies that we can replace $C_0$ with $C_1$ with     $C_1  = \sqrt{c_2(1+C_0)} \le   2\sqrt{c_2 C_0 }$. We have $ C_1 \le C_0/2$ if $C_0\ge c_0:=16c_2 $. The proof is now completed.\qed

 \subsection{Completion of the proof of Proposition \ref{prop:mainbounds}} \label{subsec:prop}
\subsubsection{Bounds on the continuous modes} \label{subsec:cont}
We start this section by listing some results known in the literature, then we prove some auxiliary tools required for the proof of Proposition \ref{prop:mainbounds}. The following theorem is Theorem 1.1 \cite{boussaid}
\begin{theorem}
\label{Thm:bouss1}  Under hypotheses (H2)--(H3)  for $s>5/2$ and any $k\in \R$ we have
\begin{equation}\label{eq:bouss11}
   \| e^{\im H t}P_c u_0\| _{H^{k,-s}(\R ^3, \C ^4)}\le C _{s,k}\langle t \rangle ^{-\frac{3}{2}} \|  P_c u_0\| _{H^{k, s}(\R ^3, \C ^4)}.
\end{equation}

\end{theorem}
\qed

\noindent The subsequent  is  Theorem 1.1 in  \cite{Boussaid2}.
\begin{theorem}[Smoothness estimates]
\label{Thm:Smoothness} For any $\tau > 1 $ and $k\in \R$
$\exists $ $C$ such that
  \begin{align} &%
 \|  e^{-\rmi t H}P_c\psi \|_{L_t^ 2(\R,H^{k, -\tau}(\R ^3, \C ^4))} \leq
C \| P_c\psi \|_{H^{k }(\R ^3, \C ^4)},\label{eq:Smooth1}
 \\&
  \| \int_{\R}e^{\rmi t H}  P_cF(t)\;dt \|_{H^k} \leq
C \| P_c F \|_{L_t^2(\R,H^{k,  \tau}(\R ^3, \C ^4))},\label{eq:Smooth2}
 \\&
 \|\int_{t'<t}  e^{-\rmi(t-t')
H}  P_c F(t')\;dt'  \|_{L_t^2(\R,H^{k, -\tau}(\R ^3, \C ^4))} \leq C \| P_cF
\|_{L_t^2(\R,H^{k,  \tau}(\R ^3, \C ^4))}.\label{eq:Smooth3}
  \end{align}
\end{theorem}

\noindent The following  is  Theorem 3.1 in  \cite{boussaid}.

\begin{theorem}
\label{Thm:dispersion} For   $  p \in [1,2]$,
$\theta\in[0,1]$, with
$ k-k'\ge (2 +\theta )    ( \frac{2}{p}-1) $ and  $q\in [1,\infty ]$,
there is a constant $C$  such that for $p'= \frac{p}{p-1}$,
\begin{equation*}
  \| e ^{\im t D _\mathscr{M}} \| _{B^{k}_{p,q}\to B^{k'}_{p',q }} \le C (K(t))^{\frac{2}{p}-1},  \text{  where }K(t):=
\left\{\begin{matrix}
|t|  ^{-1+\theta /2}  \text{  if $|t|\le 1$}\\
|t|  ^{-1-\theta /2}\text{  if $|t|\ge 1$.}\end{matrix}\right.
\end{equation*}

\end{theorem}
\qed

\noindent The following  is  Theorem 1.2 in  \cite{Boussaid2}.

\begin{theorem}[Strichartz estimates]
\label{Thm:Strichartz} For any $2\leq p,q \leq \infty$,
$\theta\in[0,1]$, with
$(1-\frac{2}{q})(1\pm\frac{\theta}{2})=\frac{2}{p}$ and
$(p,\theta)\neq(2,0)$, and for any reals $k$, $k'$ with $k'-k\geq
\alpha(q)$, where $\alpha(q)=(1+\frac{\theta}{2})(1-\frac{2}{q})$,
there exists a positive constant $C$ such that
  \begin{align} &
\left\|e^{-\rmi t
H }P_c\psi\right\|_{L_t^{p}(\R,B^k_{q,2}(\R^3,\C^4))} \leq
 C\left\| P_c \psi\right\|_{H^{k'}(\R^3,\C^4)}, \label{eq:Stri1}
\\&
  \left\|\int _{\R}e^{\rmi t H} P_c F(t)\,dt\right\|_{H^k(\R ^3, \C ^4)} \leq
 C\left\|   P_c F\right\|_{L_t^{p'}(\R,B^{k'}_{q',2}(\R^3,\C^4))}, \label{eq:Stri2}
\\&
  \left\| \int_{t'<t} e^{-\rmi(t-t')H}P_c F(t')\,dt'
\right\|_{L_t^{p}(\R,B^{ k}_{q,2}(\R^3,\C^4))} \leq  C
 \left\|   P_cF\right\|_{L_t^{a'}(\R,B^{h}
 _{b',2}(\R^3,\C^4))}, \label{eq:Stri3}
\end{align}
for any  $(a, b)$ chosen like $(p,q)$ and
$h-k\geq\alpha(q)+\alpha(b)$.
\end{theorem}
\qed

We have the following facts concerning the resolvent of the operator $D_\mathscr{M}$, see \cite{boussaid,Boussaid2,boussaidcuccagna}.
\begin{lemma}
  \label{lem:absflat} The following facts are true.
 \begin{itemize}
 \item[(1)] For $z\not \in \sigma (D_\mathscr{M})$ for the integral kernel  we have  $R  _{D_\mathscr{M}
}(x,y,z)=  R  _{D_\mathscr{M}
}(x-y,z )  $      with
\begin{equation} \label{eq:flat R1} \begin{aligned}&
 R  _{D_\mathscr{M} }(x ,z  )= \begin{pmatrix}
 (z +\mathscr{M}) I_2 &   \rmi \sqrt{\mathscr{M}^2-z ^2}\sigma \cdot \widehat{x} \\
 \rmi \sqrt{\mathscr{M}^2-z ^2}\sigma \cdot \widehat{x} &   (z
-\mathscr{M}) I_2
\end{pmatrix} \frac{e^{-\sqrt{\mathscr{M}^2-z ^2} |x|}}{4\pi |x| } +
\rmi   \frac{\alpha  \cdot \widehat{x}} {4\pi |x|
^2}e^{-\sqrt{\mathscr{M}^2-z ^2} |x|},
\end{aligned}
\end{equation}
where  $\widehat{x}=x/|x|$ and  where for $\zeta =e^{\rmi
\vartheta} r$ with $r\ge 0$ and $\vartheta \in (-\pi , \pi )$ we
set $\sqrt{\zeta} =e^{\rmi \vartheta/2} \sqrt{r}$.

\item[(2)] For any $\tau > 1 $   there exists $C$
  such that $
    \|  R _{D_\mathscr{M}}(z)  \psi  \| _{L^{2,-\tau }(\R ^3, \C ^4)}
     \leq C \|\psi \| _{L^{2, \tau }(\R ^3, \C ^4)}
    $,  for all $z\not \in \R$.
      \item[(3)]
  For any   $\tau >1$   the following limits  exist  in $B(H^{1, \tau   }(\R ^3, \C ^4) , L^{2, -\tau   } (\R ^3, \C ^4))$,
   \begin{equation} \label{eq:absfree5}\index{$R_{D_m}^+$}
   R_{D_\mathscr{M}}^+
    (\lambda )=\lim _{\varepsilon \searrow 0} R_{D_\mathscr{M}}
    (\lambda \pm \rmi \varepsilon  )
    \text{ for  }   \lambda \in \R \backslash (-\mathscr{M},\mathscr{M})
   \end{equation}
   and the
   convergence is uniform for $\lambda $ in compact subsets of $\R \backslash (-\mathscr{M},\mathscr{M})$.

   \item[(4)]  $ R  _{D_\mathscr{M} }^{+}(x ,z  )$ for $z
>\mathscr{M}  $   (resp. $z <-\mathscr{M}  $) is obtained substituting
$\sqrt{\mathscr{M}^2-z
^2}$  in  \eqref{eq:flat R1}  with  $\displaystyle -\rmi \sqrt{
z ^2 -\mathscr{M}^2}= \lim _{\varepsilon \searrow 0}\sqrt{\mathscr{M}^2-(z
+\rmi \varepsilon )^2} $ (resp. $\displaystyle  \rmi \sqrt{ z
^2 -\mathscr{M}^2}= \lim _{\varepsilon \searrow 0}\sqrt{\mathscr{M}^2-(z +\rmi
\varepsilon )^2} $).

   \item[(5)]  We have
\begin{equation} \label{eq:smooth23}  \begin{aligned} &  R_{D_\mathscr{M} }  ^{\pm }(\lambda     )  =R_{-\Delta +\mathscr{M} ^2 }  ^{\pm }(  \lambda    ^2  ) \mathcal{A}(\lambda , \nabla ) \text{ with }  \mathcal{A} (\lambda , \nabla ) :=\begin{pmatrix}
{\lambda   +\mathscr{M}}
  &
   -\rmi{\sigma \cdot \nabla }   \\
    -\rmi {\sigma \cdot \nabla }    &    {\lambda   -\mathscr{M}}
\end{pmatrix}      . \end{aligned}
   \end{equation}

   \end{itemize}
\end{lemma}

By Lemma \ref{lem:absflat} above we are able to deal with the resolvent of the perturbed Dirac operator $H$.

\begin{lemma}  \label{lem:smooth1} For any preassigned  $\tau >1$
the following facts hold.

\begin{itemize}

\item[(1)]  The  limits
   $\displaystyle
 R_{H}^\pm (\lambda)=R_{H}
    (\lambda \pm \rmi 0  ) :=\lim _{\varepsilon \searrow 0}
R_{H}
    (\lambda \pm \rmi \varepsilon  )
  $, for $\lambda \in (-\infty ,-\mathscr{M} ) \cup (\mathscr{M} , \infty )$,
   exist  in $B(L^{2, \tau   }  , L^{2, -\tau   } )$ and the
   convergence is  uniform   in   compact subsets of $(-\infty ,-\mathscr{M} ) \cup (\mathscr{M} , \infty )$.

\item[(2)]
  There exists a
constant $ C_1=C _1(\tau  )$
such that  for any $u _0  \in L^2(\R ^3,\mathbb{C}^4)$ and any
$\varepsilon
\ge 0$ we have
 \begin{align}  &
   \| \langle x \rangle ^{-\tau}  R_{ H}(\lambda \pm \rmi
\varepsilon )
    P_c  u _0\| _{L ^2 _{\lambda} (\R  , L^2 _{ x}  (\R ^3) ) } \le C _{1} \|
P_c  u _0\| _{  L^2(\R ^3) }  .  \label{eq:smooth10}\end{align}

 \item[(3)]  Let $ \Lambda$ be a
compact subset of $(-\infty ,-\mathscr{M} )\cup ( \mathscr{M} ,\infty ) $.
  There exists a
constant $ C_1=C _1(\tau , \Lambda  )$
such that
 \begin{align}  &  \| \langle x \rangle ^{-\tau}  R_{ H}^{\pm }(\lambda   )
    P_c  u _0\| _{L ^\infty _{\lambda} ( \Lambda  , L^2 _{ x}  (\R ^3) ) } \le C _{1} \|
P_c  u _0\| _{  L^2(\R ^3) } . \label{eq:smooth11}
   \end{align}

   \end{itemize}

 \end{lemma}
\proof Claim (1) is an immediate consequence of Theorem \ref{Thm:bouss1} in the case $\tau >5/2$, as observed on p. 783 \cite{boussaid}. The extension  to the case $\tau >1$
 follows by the proof of     Proposition 3.10 \cite{boussaid}.   \eqref{eq:smooth10}
  is equivalent to \eqref{eq:Smooth1} for $k=0$.

\noindent We prove now \eqref{eq:smooth11}.
Let
  $u_0= P_c  u _0$,
    $A (x) =\langle x \rangle ^{-\tau}$
  and    $
B(x)\in \mathcal{S}( \mathbb{R}^3,S_4 (\C ))  $
  such that $B^*A=V  $. Then
 \begin{equation*} \label{eq:smooth101}
 A R_{H} (z  ) u _0=
 ( 1+ AR_{D_\mathscr{M} } (z  )B^*   )^{-1}
  AR_{D_\mathscr{M}} (z  )u _0    \text{  for $z\in \C \backslash \R$.}
 \end{equation*}
This equality continues to hold  on $\R \pm \im 0$ by     Lemmas \ref{lem:absflat} and \ref{lem:smooth1}.   We  then have
\begin{equation} \label{eq:smooth113} \begin{aligned} &\|  R_{H }^+ (\lambda
)P_c   \| _{ B( L^{2,\tau}_x,
L^{2,-\tau}_x) }\le  \| ( 1+ AR_{D_\mathscr{M}} ^{+}(\lambda  )B^*   )^{-1}\| _{ B( L^{2 ,\tau }_x,
L^{2 ,\tau }_x) } \|
  R_{D_\mathscr{M}}^+ (\lambda
)\| _{ B( L^{2 ,\tau }_x,
L^{2,-\tau}_x) }.
\end{aligned}
\end{equation}
By  \cite{Agmon}  there is a $C'(\tau )>0$ such that for all $\lambda \in \R$
$$\| \lambda
  R_{-\Delta}^+ (\lambda ^2
)\| _{ B( L^{2,\tau}_x,
L^{2,-\tau}_x) }+ \|  \nabla
  R_{-\Delta}^+ (\lambda ^2
)\| _{ B( L^{2,\tau}_x,
L^{2,-\tau}_x) }\le C'(\tau ). $$
 Then by \eqref{eq:smooth23}  we have $\|
  R_{D_\mathscr{M}}^+ (\lambda
)\| _{ B( L^{2 ,\tau }_x,
L^{2,-\tau}_x) }\le C (\tau )$  for all $\lambda \in \R$.

\noindent We obtain \eqref{eq:smooth11} from
\begin{equation} \label{eq:smooth114} \begin{aligned} & \sup _{\lambda \in \Lambda} \| ( 1+ AR_{D_\mathscr{M}} ^{+}(\lambda  )B^*   )^{-1}\| _{ B( L^{2 ,\tau }_x,
L^{2 ,\tau }_x) }  < \infty,
\end{aligned}
\end{equation}
which follows from the analytic Fredholm alternative.

\qed

\begin{remark}
\label{rem:egs}     Notice that \eqref{eq:smooth114} is in fact true for $\Lambda =\R$
by (H3) and, for large $\lambda$, by \cite{EGS}, see Appendix A \cite{boussaidcuccagna}.
\end{remark}

\noindent The next lemma is proved by  an argument of \cite{M1} reviewed in Lemma 5.7 \cite{boussaidcuccagna}.
\begin{lemma}
\label{lem:surrogate}  Consider pairs $(p,q)$ as in Theorem
\ref{Thm:Strichartz} with $p>2$, $k\in \R$ arbitrary and
$k'-k\ge \alpha (q)$. Then for any $\tau >1$ there is a constant
$C_0=C_0(\tau , k,p,q)$ such that
\begin{equation}
\label{eq:surrogate} \left\|  \int _{0} ^t e^{\im H(t'-t)} P_cF(t')
dt'\right \| _{L^p_tB^{k} _{q,2}} \le C_0    \|
 P_cF\|_{L_t^2H ^{   k', \tau } }\quad .
\end{equation}
\end{lemma}
\proof  For $F(t,x
)\in C^\infty _0(  \mathbb{R}\times  \mathbb{R}^3)$ set
\begin{eqnarray}T F(t):=\int _0^{+\infty}
e^{\im (t'-t)H} P_c F(t') dt' \, , \quad f:= \int _0^{+\infty} e^{\im
t' H} P_c F(t') dt' .\nonumber
\end{eqnarray} Theorem \ref{Thm:Strichartz} implies
$\left\|  T F\right \| _{L^p_tB^{k} _{q,2}} \le \| f \| _{H^{k'}}$
for $k'-k=\alpha (q)$. By Theorem \ref{Thm:Smoothness}  we have  $\| f \|
_{H^{k'}} \le C  \|
 F\|_{L_t^2H ^{   k', \tau } }.$ By $p>2$
  a lemma by Christ and Kiselev \cite{ChristKiselev}, see
 Lemma 3.1 \cite{SmithSogge},
  yields Lemma \ref{lem:surrogate}.

   \qed

\begin{lemma}\label{lem:conditional4.2} Assume
the hypotheses of Prop. \ref{prop:mainbounds} and recall the definition of ${M}$ in Definition \ref{def:setM}.  Let $\tau _0>1$.   Then  there is a fixed $c_1$
such that for all admissible pairs $(p,q)$ inequality \eqref{4.5} holds.

\end{lemma}
 \proof
 By picking $\epsilon _0>0$ sufficiently small and $\epsilon =\| u (0)\| _{H^4}<\epsilon _0$, for a fixed $c_1>0$
    for the final coordinates $(z  (0), \eta   (0))$   of $u (0)$ we have
     \begin{equation}
  |z  (0)|+\| \eta   (0) \| _{H^4} \le  c_1    \epsilon ,
  \label{4.5id}
\end{equation}
We have  for $ {G}_{j \textbf{m}}^*   =  {G}_{j \textbf{m}}^* (0  )$
\begin{equation} \label{eq:eq f} \begin{aligned} &
 \im \dot \eta =\im \{ \eta , \mathcal{H}  \} =    H \eta  + \sum _{j =1}^n \sum _{l=1 }^{ 2N+3 }	\sum _{   |\textbf{m}  |  =l}  {z}_j \overline{\textbf{Z}}^{\textbf{m}}
  {G}_{j \textbf{m}} ^* + \mathbb{A}, \text{  where}\\&
    \mathbb{A}:=
    \sum _{j =1}^n \sum _{l=1 }^{ 2N+3 }	\sum _{   |\textbf{m}  |  =l}  {z}_j \overline{\textbf{Z}}^{\textbf{m}}
  [ {G}_{j \textbf{m}}^* (|z_j|^2  )  -  {G}_{j \textbf{m}}^*  ] + \nabla _{  \eta ^*} \resto   .
\end{aligned}\end{equation}
We rewrite
\begin{equation} \label{4.9} \begin{aligned} &
 \sum _{j =1}^n \sum _{l=1 }^{ 2N+3 }	\sum _{   |\textbf{m}  |  =l}  {z}_j \overline{\textbf{Z}}^{\textbf{m}}
   {G}_{j \textbf{m}}^*= \sum _{(\mu , \nu )\in {M}  }    \overline{z} ^{\mu} {z}  ^{\nu}  {G} _{\mu \nu}^*.
\end{aligned}\end{equation}
Notice that \eqref{L^2discrete} is the same as
 \begin{align}
&    \|   {z} ^{\mu} \overline{{z}}  ^{\nu} \| _{L^2_t(I)}\le
  C   \epsilon \text{ for all     $(\mu , \nu )\in {M} . $ } \label{L^2disbis}
\end{align}
The proof of Lemma \ref{lem:conditional4.2} is a consequence of Lemmas \ref{lem:rem1} , \ref{lem:conditional4.21} and \ref{lem:conditional4.22} below. \qed

 \begin{lemma} \label{lem:rem1}
For $I_T:= [0,T]$ and for   $S\in \R   $ and  $\epsilon _0>0$  small enough
 then for a constant $C (S,C_0) $   independent from $T$ and $\epsilon$ we have
\begin{equation} \label{eq:eq A} \begin{aligned} &
 \| \mathbb{A} \|  _{ L^2 (I_T, H ^{4, S}) + L^1 (I_T , H^4) } \le C (S,C_0) \epsilon ^2.
\end{aligned}\end{equation}
\end{lemma}
\proof We have   $r-1\ge S $,
\begin{equation} \label{4.88} \begin{aligned} &
\| {z}_j \overline{\textbf{Z}}^{\textbf{m}}
  [ {G}_{j \textbf{m}}^* (|z_j|^2  )  -  {G}_{j \textbf{m}}^*  ]   \|  _{  L^2 (I_T, H ^{4, S})} \le \| {z}_j \overline{\textbf{Z}}^{\textbf{m}}
    \|  _{  L^2 (I_T, \C)}    \|   {G}_{j \textbf{m}} (|z_j|^2  )  -  {G}_{j \textbf{m}}    \|  _{  L^\infty (I_T, H ^{4, S})} \\& \le C_0  \epsilon
       \sup \{ \|  {G}_{j \textbf{m}} '(|z_j|^2  )\| _{\Sigma _{r}}:  |z_j|\le   \delta _{0} \}
        \|   z_j ^2     \|  _{  L^\infty (I_T, \C)} \le   C C_0^3 \epsilon ^{3}   .
\end{aligned}\end{equation}
Furthermore we get for a fixed $c_1>0$
\begin{equation} \label{4.8} \begin{aligned} &
 \| \nabla  _{\eta ^*} E_P (\eta )   \|  _{L^1 (I_T , H^4) } = 2 \|  g(\eta \overline{\eta} )   \eta    \|  _{L^1 (I_T , H^4) }   \le   c_1  \|    \eta    \|  _{L^\infty  (I_T , H^4) } \|    \eta    \| ^2  _{L^2 (I_T , L^\infty) }\le c_1 C_0^3   \epsilon ^3  .
\end{aligned}\end{equation}
The rest of Lemma  \ref{lem:rem1}  follows by the fact that for arbitrarily preassigned $S>2$,
\begin{align} &
  \| R_1  \|  _{ L^2 (I_T, H ^{4, S})} \le C (S,C_0) \epsilon ^2  \text{ for $R_1 =\nabla_{ \eta ^*} ( \resto  -E_P (\eta ))$.} \label{4.6}
\end{align}
This inequality is proved in \cite{CM1}    (for $H ^{4, S}$  replaced by $H ^{1, S}$,
but the proof is  the same).  Then  \eqref{4.88}--\eqref{4.6}
  imply  \eqref{eq:eq A}.


 \qed

\begin{lemma}\label{lem:conditional4.21} Consider $\im \dot \psi -
 {H}\psi =F$  where $P_c $  and $\psi= P_c\psi$. Let $k\in
\Z$  with $k\ge 0$ and $\tau _0>1$.  Then for $(p,q) $ as in  \eqref{eq:numbers1}  and $\tau _0>1$  for a constant $C=C(p,q,k,\tau _0)$  we have
\begin{equation} \label{eq:421}  \begin{aligned}
 &     \|  \psi \| _{L^p_t( [0,T ],B^{  k -\frac{2}{p}} _{q,2})\cap
L^2_t([0,T],H^{  k  ,-\tau _0}_x) }\le C\| \psi(0) \| _{H^{  k }}+ C \|  F \|
_{L^1_t([0,T],H^{ k } ) + L^2_t([0,T],H^{ k  , \tau _0} )},  \end{aligned}
\end{equation}
\end{lemma}
\begin{proof} We split $F=F_1+F_2$ with $F_1\in L^1 ([0,T],H^{ k } )$ and
$F_2\in  L^2 ([0,T],H^{ k  , \tau _0} ) $ and we write
\begin{equation}\label{eq:duh1}
 \psi (t) = e^{-\im tH}\psi (0) -\im \sum _{j=1}^{2}\int _{0}^te^{-\im (t-s)H}F_j(s) ds.
\end{equation}
Estimate \eqref{eq:421} in the special case $F=0$ is a consequence of \eqref{eq:Stri1} and \eqref{eq:Smooth1}. The case $\psi_0=0$, $F_2=0$ follows by \eqref{eq:Stri3} and \eqref{eq:Smooth1}. Finally, the  case $\psi_0=0$, $F_1=0$ follows by  \eqref{eq:surrogate} and \eqref{eq:Smooth3}.

 \end{proof}

\begin{lemma}\label{lem:conditional4.22} Using the notation of Lemma \ref{lem:conditional4.21},  but this time picking $\tau _0>3/2$,     we have
\begin{equation} \label{eq:endpoint1}  \begin{aligned}
 &     \|  \psi \| _{L^2_t( [0,T ],L^{ \infty} ) }\le C\| \psi(0) \| _{H^{  4 }}+ C \|  F \|
_{L^1_t([0,T],H^{ 4 }_x) + L^2_t([0,T],H^{ 4  , \tau _0}_x)}  \end{aligned}
\end{equation}
\end{lemma}
 \begin{proof} The argument is the same of Lemma 10.5 \cite{boussaidcuccagna}.  We
 consider
 \begin{equation}\label{eq:duh2}
 \psi (t) = e^{-\im tD_{\mathscr{M}}}\psi (0)  +\im  \int _{0}^te^{-\im (t-s)D_\mathscr{M}}V \psi (s) ds-\im \sum _{j=1}^{2}\int _{0}^te^{-\im (t-s)D_\mathscr{M}}F_j(s) ds.
\end{equation}
     We have
    for $k\in ( 1/2  , 3]$
\begin{equation}    \begin{aligned}
 &  \| e^{-\im D_\mathscr{M} t}\psi  (0)\| _{L^2_t L^{ \infty}_x  }\le C \| e^{-\im D_\mathscr{M} t}\psi  (0)\| _{L^2_t  B _{6,2} ^{k} }\le C '\| \psi (0)\| _{H ^{k+1} } \le C '\| \psi (0)\| _{H ^{4} } \end{aligned} \nonumber
 \end{equation}
by  the flat version of \eqref{eq:Stri1}, which holds by \cite{boussaid}. Similarly we have
\begin{equation}    \begin{aligned}
 &  \| \int _0^t e^{ \im D_\mathscr{M}  (t'-t)}  F_1 (t')dt'\| _{L^2_t L^{ \infty}_x  }\le C \| \int _0^t e^{ \im D_\mathscr{M} (t'-t)}   F_1 (t')dt'\| _{L^2_t  B _{6,2} ^{k} }\\& \le C '\| F_1\| _{L^1_t H ^{k+1} } \le C '  \| F_1\| _{L^1_t H ^{4} } .\end{aligned} \nonumber
 \end{equation}
Using $B ^{k}_{\infty ,2} \subset L^{ \infty} $  for $k>1/2$ and picking $k<1$,   by   Theorem \ref{Thm:dispersion} we have
\begin{equation} \label{eq:H^4}   \begin{aligned}
 &  \| \int _0^t e^{ \im D_\mathscr{M} (t'-t)}   F_2 (t')dt'\| _{L^2_t L^{ \infty}_x  }\le C \left \| \int_0^t  \min \{ |t-t'|   ^{-\frac{1}{2}},
 |t-t'|   ^{-\frac{3}{2}} \}  \|    F_2(t') \|_{B^{k+3}_{1, 2  }} dt' \right \|
_{L_t^{2} } \\& \le C '\| F_2\| _{L^2_t B^{ 4}_{1, 2  } } \le C ''  \| \langle x \rangle ^{ \tau _0}F_2\| _{L^2_t B^{4}_{2, 2  } } =C ''  \|  F_2\| _{L^2_t H^{4, \tau _0}  }  ,\end{aligned}
 \end{equation}
 where we have used $\|   \varphi _j* F_2 \| _{L^1_x}\le \| \langle x \rangle ^{- \tau _0}   \| _{L^2_x}  \| \langle x \rangle ^{ \tau _0} \varphi _j* F_2 \| _{L^2_x}
 \le C''' \| \varphi _j* ( \langle \cdot  \rangle ^{ \tau _0}   F_2 )\| _{L^2_x}$
 for fixed $C'''>0$  and fixed $\tau _0>3/2$.  With $F_2$ replaced by
 $V\psi$ we get a similar estimate. This yields inequality \eqref{eq:endpoint1}.

\end{proof}

 Setting $M= M (2N+4)$, see  Definition \ref{def:setM},
  we now introduce a new variable $g$ setting
\begin{equation} \label{eq:def g} \begin{aligned} &
 g   =  \eta +Y \text{   with } Y:=
  \sum _{(\alpha , \beta )\in M   }   \overline{{z}}^\alpha  {{z}}^\beta   R_H^{+}( {\textbf{e} }     \cdot (\beta  -
\alpha  ))
   {G}_{\alpha \beta  }  ^*   .
\end{aligned}\end{equation}
This can be traced in \cite{BP2,SW4} and has the following   meaning. When we write $\eta  =-Y+g$  the term $-Y$ is the part of $\eta$ which has the most significant
effect on the variables $z_j$. Substituting $\eta$ by $-Y+g$ in the equations for the $z_j$,
these equations reduce to equations dependent only on $z$, up to a perturbation.

The following lemma is an easier version of Lemma 10.7 \cite{boussaidcuccagna} and so we give the proof in few lines.
\begin{lemma}\label{lem:bound g}  Assume the hypotheses of Prop. \ref{prop:mainbounds} and fix $S>9/2$. Then  there is a $c_1(S)>0$
 such that
 for any $C_0$ there is a   $\epsilon _0=\epsilon _0(C_0,S) >0$   such that for $\epsilon \in (0, \epsilon _0)$ in Theorem \ref{thm:small en} we have
\begin{equation} \label{bound:auxiliary}\| g
\| _{L^2 ([0,T], L^{2,-S}  )}\le c_1(S) \epsilon  .\end{equation}
\end{lemma}
\proof  Substituting \eqref{eq:def g}  in \eqref{eq:eq f} and using \eqref{4.9} we obtain
\begin{equation} \label{eq:eq g1}  \begin{aligned} &
 \im \dot g =     H g  +   \im \dot Y -HY + \sum _{(\alpha , \beta )\in M   }   \overline{{z}}^\alpha  {{z}}^\beta
   {G}_{\alpha \beta  }  ^*+   \mathbb{A}.
\end{aligned}\end{equation}
We then compute
\begin{equation} \label{eq:eq g2}  \begin{aligned} &
   \im \dot Y =\sum _{(\alpha , \beta )\in M   }  {\textbf{e} }     \cdot (\beta  -
\alpha  ) \overline{{z}}^\alpha  {{z}}^\beta   R_H^{+}( {\textbf{e} }     \cdot (\beta  -
\alpha  ))
   {G}_{\alpha \beta  }  ^* + \mathbf{T} \text{ where} \\&     \textbf{T} :=\sum _j \left [\partial _{z_j}Y (\im \dot z_j- {e}_jz_j)+\partial _{\overline{z}_j}Y (\im \dot {\overline{z}}_j+ {e}_j\overline{z}_j)
 \right ] .
\end{aligned}\end{equation}
Then in \eqref{eq:eq g1} we have the cancellation
\begin{equation*}    \begin{aligned} &
   \sum _{(\alpha , \beta )\in M   }  {\textbf{e} }     \cdot (\beta  -
\alpha  ) \overline{{z}}^\alpha  {{z}}^\beta   R_H^{+}( {\textbf{e} }     \cdot (\beta  -
\alpha  ))
   {G}_{\alpha \beta  }  ^* -HY +    \sum _{(\alpha , \beta )\in M   }   \overline{{z}}^\alpha  {{z}}^\beta
   {G}_{\alpha \beta  }  ^*=0.
\end{aligned}\end{equation*}
So \eqref{eq:eq g1} becomes  \begin{equation} \label{eq:eq g} \begin{aligned} &
 \im \dot g =     H g  +   \mathbb{A} + \mathbf{T}   .
\end{aligned}\end{equation}
We then have
\begin{equation} \label{eq:exp g} \begin{aligned} &
 g(t)= e^{-\im Ht} \eta (0) + e^{-\im Ht} Y (0)-\im \int _0^t  e^{-\im H(t-s)}  (\mathbb{A} (s) +\textbf{T} (s) ) ds .
\end{aligned}\end{equation}
We have $ \|  e^{-\im Ht} \eta (0)
\| _{L^2 (\R , L^{2,-S}  )}\le    c   \| \eta (0) \| _{L^2} \le  c   \epsilon$ by
\eqref{eq:Smooth1}.  The rest of the proof of Lemma \ref{lem:bound g}
is exactly the same of Lemma 6.4  \cite{CM1}, where the auxiliary Lemma 6.5  \cite{CM1} needs to be replaced by
 Lemma \ref{lem:lemg9} below. \qed

Following the proof of Lemma 5.8 \cite{boussaidcuccagna} we obtain Lemma \ref{lem:lemg9} below, a lemma which is standard ingredient
in this type of proofs. For example in \cite{SW4} the analogous ingredient is
Proposition 2.2, but versions appear in \cite{BP2,TY1} (see also references therein) just to name a few.
\begin{lemma}\label{lem:lemg9}
Let $ \Lambda$ be a
compact subset of $(-\infty ,-\mathscr{M} )\cup ( \mathscr{M} ,\infty ) $
and let $S >7/2$. Then  there exists a fixed $ c(S,\Lambda)$ such that
  for every $t \ge 0$ and  $\lambda \in  \Lambda$,
\begin{equation} \label{eq:lemg91}
\|  e^{-\im H  t}R_{ H }^{+}( \lambda )
P_cv_0 \|_{L^{ 2, - S}(\R^3)} \le c(S,\Lambda)\langle  t\rangle ^{-\frac 32} \| P_cv_0  \|_{L^{ 2,   S}(\R^3)}  \text{  for all $v_0\in L^{ 2,   S}(\R^3)$} .
\end{equation}
\end{lemma}
  \proof
  We expand $ R_{H }^{+}(\lambda )=
R_{D_\mathscr{M}}^{+}(\lambda )   -  R_{D_\mathscr{M}}^{+}(\lambda ) V R_{H}^{+}(\lambda )$. For $\tau _1 >5/2$,  by Theorem \ref{Thm:bouss1}  and by
 \cite[Theorem 2]{BerthierGeorgescu} we   have
 \begin{equation*}\label{eq:w-estflat}
\|  e^{- \rmi  t D_\mathscr{M}
}R_{D_\mathscr{M}}^{+}(\lambda )   \psi  _{0}\|_{L^{2,
-\tau _1}(\R^3)} \le C\langle  t \rangle ^{-\frac 32} \|
R_{D_\mathscr{M}}^{+}(\lambda )   \psi _{0}
\|_{L^{2,  \tau _1}(\R^3)} \leq  C_1\langle  t \rangle ^{-\frac 32} \|
\psi _{0}
\|_{L^{2,  \tau _1+1}(\R^3)},
\end{equation*}
with $C_1$ locally bounded in $\lambda$ and
$\tau_1$.  Hence, by the rapid
decay of
$V$ and by Lemma \ref{lem:smooth1}
\begin{equation*} \begin{aligned} & \|  e^{-
\rmi t D_\mathscr{M}} R_{D_\mathscr{M}}^{+}(\lambda )
V
R_{H}^{+}(\lambda )  P_c    \psi  _{0}\|_{L^{2,
\tau _1} }\\
& \le  C_1\langle  t \rangle ^{-\frac 32}
\left\|V \right\|_{B(L^{2,-\tau_1},L^{2,\tau_1+1})}
\left\|R_{H}^{+}(\lambda )  P_c
\right\|_{B(L^{2,\tau_1},L^{2,-\tau_1})}\|  \psi
_{0}
\|_{L^{2,  \tau_1} } \le C ' \langle  t \rangle ^{-\frac 32}.
\end{aligned} \nonumber
\end{equation*}
\qed

\subsubsection{The analysis of the discrete modes}
\label{subsec:fgr}
Let us turn now to the analysis of the Fermi Golden Rule (FGR).
We have
\begin{equation}\label{eq:FGR01} \begin{aligned} &
\im \dot z _j
=(1 +   \varpi  _j(|z_j|^2)) (e_jz_j+
\partial _{\overline{z}_j}\mathcal{Z}_0(|z _1|^2, ..., |z _n|^2) +
\partial _{  \overline{z} _j}   \mathcal{R}) \\  &
+  (1 +   \varpi  _j(|z_j|^2))[   \sum _{ (\mu , \nu )\in M }    \nu _j  \frac{z  ^{\mu }
 \overline{ {z }}^ { {\nu}  } }{\overline{z}_j}
\langle \eta  ,
  {G}_{\mu \nu }  \rangle      +  \sum _{(\mu ', \nu' )\in M}    \mu _j ' \frac{z  ^{\nu '}
 \overline{ {z }}^ { {\mu}'  } }{\overline{z}_j}
\langle \eta^*  ,
 {{G^*}}_{\mu' \nu' }  \rangle ]\\&
+(1 +   \varpi  _j(|z_j|^2))[\sum_{\mathbf m \in \mathcal M_j (2N+3)}|z_j|^2\mathbf Z^{\mathbf m}\<G_{j\mathbf m}',\eta\>+ z_j^2 \overline{\mathbf{Z}}^{\mathbf m}\<G_{j\mathbf m}^{\prime *}, \eta^*\>]  .
\end{aligned}  \end{equation}
We use \eqref{eq:def g}  to substitute   $\eta$ in the equations above getting, for $M_{min}$ defined as in \eqref{eq:setMhat},
\begin{equation}\label{eq:FGR02}
\begin{aligned}
\im \dot z _j -e_j z_j
&=     \varpi  _j(|z_j|^2))  e_j z_j+ (1 +   \varpi  _j(|z_j|^2))
\partial _{\overline{z}_j}\mathcal{Z}_0(|z _1|^2, ..., |z _n|^2)    +
   \mathcal{A} _j +\mathcal{B}_j  +X_j(M _{min}), \end{aligned}
\end{equation}
    where  for   $\widehat{M}\subset M$     we set
    \begin{equation}\label{eq:XN}
\begin{aligned}&
    X_j(\widehat{M}):=  -\sum _{ \substack{  (\mu , \nu )\in \widehat{M} \\
(\alpha , \beta )\in \widehat{M}}}     \nu _j  \frac{z  ^{\mu +\beta}
 \overline{ {z }}^ { {\nu} +\alpha } }{\overline{z}_j}
\langle  R_H^{+}( {\textbf{e} }     \cdot (\beta -
\alpha ))  G^*_{\alpha \beta}  ,
  {G}_{\mu \nu }  \rangle  \\&
-  \sum _{ \substack{  (\mu ', \nu' )\in \widehat{M} \\
(\alpha ', \beta ')\in \widehat{M}}}         \mu _j ' \frac{z  ^{\nu '+ \alpha '}
 \overline{ {z }}^ { {\mu}'+\beta '  } }{\overline{z}_j}
\langle  R_H^{-}( {\textbf{e} }     \cdot (\beta '-
\alpha' ))  {G}_{\alpha ' \beta '}  ,
 {{G}}^*_{\mu' \nu' }  \rangle,
\end{aligned}
\end{equation}
 with
 \begin{align} &
 \mathcal{A} _j  =(1 +   \varpi  _j(|z_j|^2))
\partial _{  \overline{z} _j}   \mathcal{R} + \varpi  _j(|z_j|^2) [   \sum _{ (\mu , \nu )\in M }    \nu _j  \frac{z  ^{\mu }
 \overline{ {z }}^ { {\nu}  } }{\overline{z}_j}
\langle \eta  ,
  {G}_{\mu \nu }  \rangle      +  \sum _{(\mu ', \nu' )\in M}    \mu _j ' \frac{z  ^{\nu '}
 \overline{ {z }}^ { {\mu}'  } }{\overline{z}_j}
\langle \eta^*  ,
 {G}^*_{\mu' \nu' }  \rangle  ]
\nonumber\\&+     \sum _{ (\mu , \nu )\in M }    \nu _j  \frac{z  ^{\mu }
 \overline{ {z }}^ { {\nu}  } }{\overline{z}_j}
\langle g  ,
  {G}_{\mu \nu }  \rangle      +  \sum _{(\mu ', \nu' )\in M}    \mu _j ' \frac{z  ^{\nu '}
 \overline{ {z }}^ { {\mu}'  } }{\overline{z}_j}
\langle g^*  ,
  {G}^*_{\mu' \nu' }  \rangle \label{eq:FGRrem1}\\&
+(1 +   \varpi  _j(|z_j|^2))[\sum_{\mathbf m \in \mathcal M_j (2N+3)}|z_j|^2\mathbf Z^{\mathbf m}\<G_{j\mathbf m}',\eta\>+ z_j^2 \overline{\mathbf{Z}}^{\mathbf m}\< G_{j\mathbf m}^{\prime*}, \eta^*\>]  \nonumber
\end{align}
and
\begin{equation}\label{eq:FGRX}
\begin{aligned} &  \mathcal{B} _j= X_j(M )-X_j(M _{min}).
\end{aligned}
\end{equation}
We notice that the r.h.s. of the identity \eqref{eq:XN} is well defined by Definition \ref{def:setM} in combination with Lemma \ref{lem:lemg9}. This observation allows also the introduction of the variable $\zeta$ defined by
\begin{equation}\label{eq:FGR21}  \begin{aligned}   &
 \zeta _j = z _j  + {T}_j (z) \text{ where}
\\&  {T}_j (z):=
 \sum _{ \substack{  (\mu , \nu )\in M_{min} \\
(\alpha , \beta )\in M_{min}}}       \frac{\nu _jz  ^{\mu +\beta}
 \overline{ {z }}^ { {\nu} +\alpha } }{((\mu - \nu)\cdot \mathbf{e}-(\alpha - \beta )\cdot \mathbf{e}  )\overline{z}_j}
\langle  R_H^{+}( {\textbf{e} }     \cdot (\beta -
\alpha ))   {G}^*_{\alpha \beta}  ,
  {G}_{\mu \nu }  \rangle  \\&   + \sum _{ \substack{  (\mu ', \nu' )\in M_{min} \\
(\alpha ', \beta ')\in M_{min}}}         \frac{\mu _j ' z  ^{\nu '+ \alpha '}
 \overline{ {z }}^ { {\mu}'+\beta '  } }{((\alpha ' - \beta ')\cdot \mathbf{e}-(\mu ' - \nu ' )\cdot \mathbf{e}  )\overline{z}_j}
\langle  R_H^{-}( {\textbf{e} }     \cdot (\beta '-
\alpha' ))  {G}_{\alpha ' \beta '}  ,
  {G}^*_{\mu' \nu' }  \rangle,
  \end{aligned}
\end{equation}
with the summation   performed over the pairs where the formula makes sense, that is $(\mu -\nu )\cdot \mathbf{e}\neq (\alpha -\beta )\cdot \mathbf{e}.$
We   have the following  lemma, (see also \cite{CM1}).

\begin{lemma}\label{lem:chcoo} Assume \eqref{L^2disbis}.  We have
\begin{equation}  \label{equation:FGR3} \begin{aligned}   & \| \zeta  -
 z  \| _{L^2(0,T)} \le  c(N,C_0) \epsilon ^2  \text{  and }  \| \zeta  -
 z \| _{L^\infty (0,T)} \le  c(N,C_0) \epsilon ^2
\end{aligned}
\end{equation}
and $\zeta$   equations
 \begin{equation}\label{eq:FGR251} \begin{aligned} &
\im \dot \zeta _j
=(1+\varpi(|z_j|^2))(e_j\zeta _j+\lambda_j'(|z_j|^2)\zeta_j+
\partial _{\overline{\zeta}_j}\mathcal{Z}_0(\zeta))\\& -   \sum _{ \substack{  (\mu , \nu )\in M _{min}}}     \nu _j  \frac{|\zeta  ^{\mu  }
 \overline{ {\zeta }}^ { {\nu}   }|^2 }{\overline{\zeta}_j}
\langle  R_H^{+}(   \mathbf{e}     \cdot (\nu -
\mu )) {G}^*_{\mu \nu}  ,
  {G}_{\mu \nu }  \rangle  \\&   -   \sum _{ \substack{  (\mu ', \nu' )\in M _{min}}}         \mu _j ' \frac{|\zeta  ^{\nu ' }
 \overline{ {\zeta }}^ { {\mu}'   } |^2}{\overline{\zeta}_j}
\langle  R_H^{-}( { \mathbf{{e}} }     \cdot (\nu  '-
\mu ' ))  {G}_{\mu ' \nu '}  ,
  {G}^*_{\mu' \nu' }  \rangle      +
    \mathcal{G} _j  ,
\end{aligned}  \end{equation}
where
there are fixed $c_4$  and $ \epsilon _0>0$ such that
 for $T>0$ and $\epsilon \in (0, \epsilon _0)$ we have
 \begin{equation}\label{4.33}
  \begin{aligned}  & \| \mathcal{G}_j \overline \zeta _j\| _{L^1 [0,T]}\le (1+C_0) c_4 \epsilon ^2   .
\end{aligned}
\end{equation}

\end{lemma}

\proof
We have   \begin{equation*}  \begin{aligned} &
\im \dot \zeta _j
=  \im \dot  {T}_j  -e_j  {T}_j +e_j \zeta _j +\text{r.h.s. of \eqref{eq:FGR02}.}
\end{aligned}  \end{equation*}
By elementary computation we have
\begin{equation}\label{eq:FGR2--51}    \begin{aligned} &
   \im \dot{T}_j  -e_j  {T}_j  =-X_j(M _{min})+ \mathbb{T}_j ,\text{  where}
    \\& \mathbb{T}_j :=
    \sum _l \left [\partial _{z_l}{T}_j (\im \dot z_l- {e}_lz_l)+\partial _{\overline{z}_l}{T}_j (\im \dot {\overline{z}}_l+ {e}_j\overline{z}_l)
 \right ] .
\end{aligned}\end{equation}
This leads to the cancellation of $X_j(M _{min})$  and, for a $\mathcal{G} _j$
which we write explicitly below, to
   \begin{equation}\label{eq:FGR2-51} \begin{aligned} &
\im \dot \zeta _j
=(1+\varpi(|z_j|^2))(e_j\zeta _j+
\partial _{\overline{j} }\mathcal{Z}_0(|\zeta _1|^2, ..., |\zeta _n|^2))\\& -    \sum _{ \substack{  (\mu , \nu ), (\alpha , \beta )\in M  _{min}\\
(\alpha - \beta )\cdot \mathbf{e}=(\mu - \nu )\cdot \mathbf{e  }}  }   \nu _j  \frac{\zeta  ^{\mu +\beta}
 \overline{ {\zeta }}^ { {\nu} +\alpha } }{\overline{\zeta}_j}
\langle  R_H^{+}( {\textbf{e} }     \cdot (\beta -
\alpha ))  G^*_{\alpha \beta}  ,
  {G}_{\mu \nu }  \rangle  \\&   -  \sum _{ \substack{  (\mu ' , \nu '), (\alpha ', \beta ')\in M  _{min}\\
(\alpha '- \beta ')\cdot \mathbf{e}=(\mu '- \nu ')\cdot \mathbf{e  }}  }      \mu _j ' \frac{\zeta  ^{\nu '+ \alpha '}
 \overline{ {\zeta }}^ { {\mu}'+\beta '  } }{\overline{\zeta}_j}
\langle  R_H^{-}( {\textbf{e} }     \cdot (\beta '-
\alpha' ))  {G}_{\alpha ' \beta '}  ,
  G^* _{\mu' \nu' }  \rangle      +
    \mathcal{G} _j .
\end{aligned}  \end{equation}
By Lemma \ref{lem:comb-1}
 we achieve that $(\alpha  - \beta  )\cdot \mathbf{e}=(\mu  - \nu  )\cdot \mathbf{e  } $
 implies $(\alpha  , \beta  )  =(\mu  ,\nu  ) $. Similarly $(\alpha  ', \beta  ')  =(\mu  ',\nu  ') $. Hence \eqref{eq:FGR2-51}  can be written as
\begin{equation*}
\begin{aligned} &
\im \dot \zeta _j
=(1+\varpi(|z_j|^2))(e_j\zeta _j+
\partial _{\overline{\zeta}_j}\mathcal{Z}_0(|\zeta _1|^2, ..., |\zeta _n|^2))\\& -    \sum _{    (\mu , \nu )\in M _{min} }     \nu _j  \frac{|\zeta  ^{\mu  }
 \overline{ {\zeta }}^ { {\nu}   } |^2}{\overline{\zeta}_j}
\langle  R_H^{+}( { \mathbf{{e}} }     \cdot (\nu  -
\mu ))  G^*_{\mu \nu}  ,
  {G}_{\mu \nu }  \rangle  \\&   -   \sum _{    (\mu  , \nu  )\in M _{min} }         \mu _j   \frac{|\zeta  ^{\mu  }
 \overline{ {\zeta }}^ { {\nu}   } |^2 }{\overline{\zeta}_j}
\langle  R_H^{-}( \mathbf{{e}}     \cdot (\nu -
\mu ))  {G}_{\mu \nu}  ,
 G^*_{\mu  \nu }  \rangle      +
    \mathcal{G} _j  ,
\end{aligned}  \end{equation*}
that is the equation \eqref{eq:FGR251}, where, recalling  $ \mathcal{A} _j $ as in \eqref{eq:FGRrem1} and $\mathcal{B}_j$ as in \eqref{eq:FGRX},
\begin{equation*}  \label{eq:restG}
\begin{aligned}
   \mathcal{G} _j &= \mathcal{B}_j+ \mathcal{G} _j ' ,\text{  where}\\ \mathcal{G} _j '&:=\mathcal{A} _j   +(1+\varpi(|z_j|^2))[  {\partial} _{ \overline{j}}\mathcal{Z}_0(|z_1|^2, ..., |z_n|^2) -{\partial} _{ \overline{j}}\mathcal{Z}_0(|\zeta _1|^2, ..., |\zeta _n|^2)]
-e_j \varpi(|z_j|^2)  {T}_j (z) + \mathbb{T}_j.
\end{aligned}
\end{equation*}
 The  proof of
  \eqref{equation:FGR3}, an easy consequence of the estimate \eqref{L^2disbis}
(or equivalently \eqref{L^2discrete}),  and the proof of   $\| \mathcal{G}_j '\overline \zeta _j\| _{L^1 [0,T]}\le (1+C_0) c_4 \epsilon ^2 $ are  in \cite{CM1}. The proof
 of \eqref{4.33} is then a consequence of
\begin{equation}\label{eq:inB}
  \begin{aligned}  & \|  \mathcal{B}_j \zeta _j\| _{L^1_t} \le  \|  \mathcal{B} _j z _j\| _{L^1_t}+ \|  \mathcal{B}_j\| _{L^ {2}_t} \| z_j- \zeta _j\| _{L^2_t} \le C(C_0) \epsilon^3.
\end{aligned}
\end{equation}
  To prove  \eqref{eq:inB} we use
\begin{equation*}
  \begin{aligned}  & \|  \mathcal{B} _j z _j\| _{L^1_t} \lesssim \sum _{ \substack{  (\mu , \nu )\in M \backslash M _{min}  \\
(\alpha , \beta )\in M   }} \|   z ^{\mu }   \overline{z}^{\nu } \| _{L^ {2}_t} \| z ^{\alpha }   \overline{z}^{\beta }\| _{L^2_t}  \le C (C_0)\epsilon^3,
\end{aligned}
\end{equation*}
by  the definition of $M_{min}$. In an analogous way it is possible to prove the
remaining estimates needed for  the second term on the r.h.s. of \eqref{eq:inB}.

\qed

By multiplying the identity \eqref{eq:FGR251} above by $ \overline{\zeta}_j$ and summing over the index $j$ we achieve, in like manner as in \cite{CM1},
\begin{align}\label{eq:FGR252}
   2 ^{-1} \sum _j\frac{d}{dt}   \  |  \zeta _j| ^2
= - \sum _j \Im \big [  \sum _{ \substack{  (\mu , \nu )\in M_{min}}}     \nu_j   \   |\zeta  ^{\mu  }
 \overline{ {\zeta }}^ { {\nu}   }|^2
\langle  R_H^{+}((\nu-\mu)\cdot \mathbf{e})  G^*_{\mu \nu}  ,
  {G}_{\mu \nu }  \rangle & \nonumber\\
  \\  +    \sum _{ \substack{  (\mu  , \nu )\in M_{min}}}         \mu _j   \ | \zeta  ^{\mu '}
 \overline{ {\zeta }}^ { {\nu}  }|^2
\langle  R_H^{-}( (\nu-\mu)\cdot \mathbf{e})  {G}_{\mu  \nu }  ,
  G^*_{\mu \nu }  \rangle    \big ] +  \sum _j   \Im [
  \mathcal{G} _j  {\overline\zeta} _j ] .
 \nonumber
  \end{align}
Thus by using the substitution, for any $(\nu-\mu)\cdot \mathbf{e}\in \Lambda$ (see the formulas \eqref{eq:plem} and \eqref{eq:pv} in Lemma \ref{plemeljDir} of the Appendix \ref{app:Pformula}),
\begin{equation}\label{eq:plemelj}
  \begin{aligned} &
  R_H^{\pm }( (\nu-\mu)\cdot  \mathbf{e}) = P.V. \frac{1}{H-(\nu-\mu)\cdot \mathbf{e}}\pm \im \pi \delta ( H-(\nu-\mu)\cdot \mathbf{e}),
\end{aligned}
\end{equation}
we can state the following lemma.
 \begin{lemma}
  \label{lem:can1}   For any $(\mu , \nu )\in M_{min}$,   we
  have
  \begin{equation} \label{eq:can2} \begin{aligned} &
\sum_j \Im [  \sum _{ \substack{  (\mu , \nu )\in M_{min}\\
 }}     \nu_j    \   \ |\zeta  ^{\mu  }
 \overline{ {\zeta }}^ { {\nu}   }|^2
\langle  P.V. \frac{1}{H-(\nu-\mu)\cdot \mathbf{e}}  G^*_{\mu\nu}  ,
  {G}_{\mu \nu }  \rangle  \\&   +    \sum _{ \substack{  (\mu  , \nu  )\in M_{min}}}         \mu_j     \   |\zeta  ^{\mu   }
 \overline{ {\zeta }}^ { {\nu }   }|^2\langle P.V. \frac{1}{H-(\nu-\mu)\cdot \mathbf{e}}  {G}_{\mu   \nu  }  ,
  G^*_{\mu  \nu  }  \rangle     ] =0.
\end{aligned}\end{equation}

\end{lemma}
 \proof By formula \eqref{eq:pv} below
 the terms $\langle P.V. \frac{1}{H-(\nu-\mu)\cdot \mathbf{e}} f   ,
  f  ^* \rangle  $, for $f ={G}_{\mu   \nu  },G^*_{\mu  \nu  }$, are real valued.

 \qed

We have also the lemma below.
 \begin{lemma}
  \label{lem:pos1}   For any $(\mu , \nu )\in M_{min}$
 we have
  \begin{equation} \label{eq:pos3} \begin{aligned} &
\pi\sum_j \Im [ \im  \sum _{ \substack{  (\mu , \nu )\in  M_{min} }}     \nu_j  \,  |\zeta  ^{\mu }
 \overline{ {\zeta }}^ { {\nu}}|^2
\langle \delta ({H- (\nu-\mu)\cdot \mathbf{e}})  G^*_{\mu \nu}  ,
  {G}_{\mu \nu }  \rangle  \\&  -\im     \sum _{ \substack{  (\mu , \nu  )\in M_{min}}}         \mu_j     \,  |\zeta  ^{\mu   }
 \overline{ {\zeta }}^ { {\nu }}|^2
\langle \delta ({H- (\nu-\mu)\cdot \mathbf{e}})  {G}_{\mu   \nu  }  ,
  G^*_{\mu  \nu  }  \rangle     ]  \\& =
\pi   \sum _{ \substack{  (\mu , \nu )\in  M_{min} }} \,  |\zeta  ^{\mu }
 \overline{ {\zeta }}^ { {\nu}}|^2\langle \delta ({H-(\nu-\mu)\cdot \mathbf{e}}) G^*_{\mu \nu}  ,
    G_{\mu \nu}  \rangle  \ge 0.
\end{aligned}\end{equation}

\end{lemma}
\proof
By \eqref{eq:delta} we have $ \langle \delta ({H- (\nu-\mu)\cdot \mathbf{e}})  {G}_{\mu   \nu  }  ,
  G^*_{\mu  \nu  }  \rangle =\langle \delta ({H- (\nu-\mu)\cdot \mathbf{e}})  {G}_{\mu   \nu  }  ^*,
  G _{\mu  \nu  }  \rangle \ge 0$.
Then, commuting the summations and by     $|\nu| -  |\mu | =1$  we get
 \begin{align} \label{eq:pos4} \text{l.h.s.\eqref{eq:pos3}}=
 \pi\sum_j\Re [    \sum _{ \substack{  (\mu , \nu )\in M_{min}}}    ( \nu_j  -\mu_j )     \,  |\zeta  ^{\mu   }
 \overline{ {\zeta }}^ { {\nu }}|^2
\langle \delta ({H-  (\nu-\mu)\cdot \mathbf{e}})  G^*_{\mu \nu}  ,
  {G}_{\mu \nu }  \rangle       ]& \nonumber \\
  =\pi \sum _{ \substack{  (\mu , \nu )\in M_{min}}}  \,\langle \delta ({H-(\nu-\mu)\cdot \mathbf{e}})  G^*_{\mu \nu}  ,
    G_{\mu \nu}   \rangle \ge 0  &. \nonumber
\end{align}

 \qed

By an application of Lemmas \ref{lem:can1} and \ref{lem:pos1} to the identity \eqref{eq:FGR252} we arrive  as in \cite{CM1} at the following.
\begin{corollary}\label{lem:eqzeta}
 We have
\begin{equation}\label{eq:FGR25} \begin{aligned} &
   2 ^{-1} \frac{d}{dt}\sum _j  \  |  \zeta _j| ^2
 =  \sum _j  \Im [
  \mathcal{G} _j  \overline{\zeta} _j ]
     -\pi \sum _{ \substack{  (\mu , \nu )\in  M_{min} }} \,  |\zeta  ^{\mu }
 \overline{ {\zeta }}^ { {\nu}}|^2 \langle \delta ({H-(\nu-\mu)\cdot \mathbf{e}})   G^*_{\mu \nu}  ,
    G_{\mu \nu}   \rangle       .
\end{aligned}\end{equation}
\end{corollary}
\qed

By Lemma \ref{plemeljDir} in the Appendix \ref{app:Pformula} we
have \begin{equation}\label{eq:restrDir}     \begin{aligned} &
\langle  \delta (H-(\nu -\mu )\cdot \mathbf{e})  G^*_{\mu  \nu}  ,
  {G}_{\mu \nu }  \rangle  = \frac {  |(\nu -\mu )\cdot \mathbf{e}| }{\sqrt{((\nu -\mu )\cdot \mathbf{e})^2-\mathscr{M}^2}} \int _{S_{\mu \nu}}|\widehat{{G}}_{\mu \nu }(\xi)|^2 dS(\xi),
\end{aligned}\end{equation}
for $S_{\mu \nu}=\{ \xi \in \R^3: |\xi |^2 +\mathscr M^2=|(\nu - \mu )\cdot \mathbf{e} |^2 \}$ and for $\widehat{{G}}_{\mu \nu }(\xi) = \mathcal{F}_{V,+}({{G}}_{\mu \nu })(\xi )$ the distorted Fourier
transform in \eqref{distortedfourier}. The  $ \widehat{{G}}_{\mu \nu }$  are continuous functions
by the fact that the ${{G}}_{\mu \nu }$ are rather regular
and quite rapidly decreasing. Since   by
(H5) we have $ \widehat{G }_{\mu \nu}| _{S_{\mu \nu}}\neq 0$
for all $(\mu , \nu )\in M_{min}$, it follows   that there exists a $\Gamma >0$ such that
\begin{equation} \label{eq:poss5}   \begin{aligned} &
\langle  \delta (H-(\nu -\mu )\cdot \mathbf{e})  G^*_{\mu  \nu}  ,
  {G}_{\mu \nu }  \rangle > \Gamma >0, \,  \text{ for all \, $(\mu , \nu )\in M _{min}$}.
\end{aligned}\end{equation}
Now we complete the proof of Proposition \ref{prop:mainbounds}.
 By the inequalities
\eqref{Strichartzradiation}--\eqref{L^inftydiscrete} and \eqref{4.33} in Lemma \ref{lem:chcoo},
integrating \eqref{eq:FGR25} and using \eqref{eq:poss5}
we get for, any $t\in [0,T]$ and for a fixed $c_2$,
\begin{equation} \label{eq:crunch}\begin{aligned}&  \sum _j   | z
_j(t)|^2 +\sum _{ ( \mu , \nu )
   \in M}  \| z^{\mu +\nu } \| _{L^2(0,t)}^2\le c_2(
\epsilon ^2+ C_0\epsilon ^2).\end{aligned}
\end{equation}
Here we have used the inequalities \eqref{equation:FGR3} in Lemma \ref{lem:chcoo}, which allow
  to switch from inequalities on  $\zeta$ to inequalities on $z$.
  The function $|\dot z_j (t)|$  can be estimated by  using   \eqref{eq:FGR02}.

\noindent In particular \eqref{eq:crunch}  yields \eqref{eq:crunch0}.   This completes the proof of Proposition
\ref{prop:mainbounds}.

\subsection{Proof   of the asymptotics \eqref{eq:small en31}} \label{subsec:prop1}

We write \eqref{eq:eq f} in the form $\im \dot \eta =    D_\mathscr{M} \eta  +V \eta + \mathbb{B}$ with, see \eqref{eq:eq f},
\begin{equation*}  \begin{aligned} &  \mathbb{B} =  \sum _{(\mu , \nu )\in {M}  }    \overline{z} ^{\mu} {z}  ^{\nu}G^* _{\mu \nu} +   \mathbb{A}  . \end{aligned}
\end{equation*}
Then $\partial _t(e^{\im D_\mathscr{M} t}\eta )=-\im e^{\im D_\mathscr{M} t}(V \eta + \mathbb{B})$ and so
\begin{equation*}  \begin{aligned} &   e^{\im  D_\mathscr{M}  t_2}\eta (t_2) -e^{\im  D_\mathscr{M}  t_1}\eta (t_1) =-\im  \int _{t_1} ^{t_2}  e^{\im  D_\mathscr{M}  t} (V \eta  (t) +\mathbb{B} (t))  dt  \text{  for $t_1<t_2$}.\end{aligned}
\end{equation*}
Then for a fixed $c_2$ by  Lemma \ref{lem:conditional4.21}, and specifically \eqref{eq:421}
  for $k=4$, $p=\infty$, $q=2$  and using $B^{  k  } _{2,2}=H^4$ we have
\begin{equation}\label{eq:compl1}  \begin{aligned} &  \|  e^{\im H t_2}\eta (t_2) -e^{\im H t_1}\eta (t_1) \| _{H^4 }\le  c_2  \|  V \eta (t) + \mathbb{B} (t)   \|  _{L^1([{t_1} ,{t_2}], H^4) + L^2([{t_1} ,{t_2}], H ^{4, 10}))}.\end{aligned}
\end{equation}
By \eqref{Strichartzradiation}, valid now in $[0,\infty)$ and     for a fixed $C$, we have
\begin{equation*}  \begin{aligned} &    \|  V \eta (t)     \|  _{L^2([{t_1} ,{t_2}], H ^{4, S})} \le   c_1  \|   \eta      \|  _{L^2(\R _+, H ^{4, -10})}\le C \epsilon  .\end{aligned}
\end{equation*}
By Proposition \ref{th:main} and \eqref{L^2disbis} we similarly   have
\begin{equation*}  \begin{aligned} &    \|   \sum _{(\mu , \nu )\in {M}  }    \overline{z} ^{\mu} {z}  ^{\nu}  {G}^* _{\mu \nu}  \|  _{  L^2(\R _+,H ^{4, 10} )} \le C '   \sum _{(\mu , \nu )\in {M}  }  \|   \overline{z} ^{\mu} {z}  ^{\nu}     \|  _{  L^2(\R _+,\C )} \le C \epsilon .
\end{aligned}
\end{equation*}
We also have \eqref{eq:eq A} for $I_T=\R _+$.
 We then conclude that   there exists an $\eta _+\in   H^4\cap  \mathcal{H}_{c}[0]$ with
\begin{equation*}  \begin{aligned} & \lim _{t\nearrow \infty}    e^{ \im D_\mathscr{M} t }\eta (t )
  = \eta _+   \text{   in $ H^4$  and with $\| \eta _+ \| _{H^4(\R ^3)}\le C \epsilon $.}   \end{aligned}
\end{equation*}

\noindent Now we prove the existence of $\rho _+$  and the facts about it in  Theorem \ref{thm:mainbounds}. First of all  we have \begin{equation*}  \begin{aligned} &
 \frac{1}{2}  \sum _j \frac{d}{dt} | z _j|^2 =
 \sum _j \Im \big [
\partial _{   \overline{j}}   \mathcal{R} \overline{z} _j   +     \sum _{ (\mu , \nu )\in M }    \nu _j   z  ^{\mu }
 \overline{  z }^ {  \nu   }
\langle \eta  ,
  {G}_{\mu \nu }  \rangle      +  \sum _{(\mu ', \nu' )\in M}    \mu _j '   z  ^{\nu '}
 \overline{  z  }^ {  \mu  '    }
\langle  {\eta} ^* ,
   {{G}}_{\mu' \nu' }^*  \rangle  \big ]      .
\end{aligned}  \end{equation*}
Since the r.h.s. has $L^1(0,\infty )$ norm  bounded by $C \epsilon ^2$ for a fixed $C$,
we conclude that   the limit
\begin{equation*} \begin{aligned} &
\lim _{t\nearrow \infty} ( | z _1(t)|  , ...   , | z _n(t)|)=  ( \rho _{1+}  ,  ...   , \rho _{n+})
\end{aligned}  \end{equation*}
exists, with  $| \rho  _+  | \le C  \| u (0)\| _{H^4}   $.
By $\lim _{t\nearrow \infty} z  ^{\mu }
 \overline{  z }^ {  \nu   } (t)=0$ for all $(    {\mu },
  {  \nu   })\in M$,  we can conclude that all but at most one
of the $\rho _{j+}$  are equal to 0.

\appendix

\section{Appendix:  proof of the formula \eqref{eq:plemelj} }
\label{app:Pformula}

This section is devoted to prove the Plemelj formula \eqref{eq:plemelj} associated to the resolvent of
the operator \eqref{Eq:op}. With this aim we need to rely now on the following facts. Borrowed by \cite{boussaid} we first introduce the matrix functions $\psi_{0}(x,\xi)\in M_4(\C)$ with vector column given by
\begin{align}\label{eq:plwav0}
\psi^{j}_{0}(x,\xi)=e^{\rmi x\cdot \xi}v(\xi)e_j, \, \, \, \,  j=1,...4,
\end{align}
with $L(\xi)=\sqrt{|\xi|^2+\mathscr{M}^2},$
\begin{align}\label{unit}
v(\xi)=\frac{(L(\xi)+\mathscr{M})I_4-\beta \alpha \cdot \xi}{\sqrt{2L(\xi)((L(\xi)+\mathscr{M}))}},
\end{align}
is a unitary matrix and $e_j $ are vectors of the canonical basis of $\C^4$ (for more details see \cite{Thaller}, Sect. 1). We recall that the transformation
$$v(\xi)\mathcal{F}(u)(\xi)= \frac 1{(2\pi )^{\frac{3}{2}}}\int _{\R ^3}\psi_{0}(x,\xi)^*u(x) dx,$$
 diagonalizes the free Dirac operator $D_\mathscr{M}$ as follows
$v^*(\xi)\mathcal{F} D_\mathscr{M} \mathcal{F}^{*}v(\xi)=L(\xi)\beta$ (we recall that  $\mathcal{F}$ denotes the classical Fourier transform with inverse  $\mathcal{F}^{*}).$
Consequently we can define the distorted plane wave functions
$\psi^{\pm}_{V}(x,\xi)\in M_4(\C)$ associated to the
continuous spectrum of $H$ (see for instance \cite{Agmon} and \cite{bambusicuccagna} for the cases of Schr\"ondinger and Klein-Gordon) as follows,
\begin{align}\label{eq:plwav}
\psi^{j,\pm}_{V}(x,\xi)=\psi^{j}_{0}(x,\xi)-\Lambda ^{j,\pm}_{V}(x,\xi),  \, \, \, \,  j=1,...4,
\end{align}
where
\begin{equation}\label{eq.al}
\Lambda ^{j,\pm}_{V}(x,\xi)=
\begin{cases}
\lim_{\varepsilon \searrow 0} R_{H}(L(\xi)\pm\rmi \varepsilon)V\psi^{j}_{0}(x,\xi) \ \ \ \  \  \  \text{for} \ \ \ j\in \{1,2\},\\
\lim_{\varepsilon \searrow 0} R_{H}(-L(\xi)\pm\rmi \varepsilon)V\psi^{j}_{0}(x,\xi)\ \ \ \ \ \ \text{for} \ \ \ j\in \{3,4\}.
\end{cases}
\end{equation}
We recall also that the distorted plane wave associated to the perturbed Dirac operator \eqref{Eq:op}, here denoted by $\psi^{+}_{V}(x,\xi),$
satisfies the equation
\begin{equation}
H\psi^+_{V}(x,\xi)=\pm L(\xi) \psi^+_{V}(x,\xi),
\end{equation}
(the same equation holds for $\psi^-_{V}(x,\xi)$). By this, for any $g\in \Sc(\R^3, \C^4),$ we have the distorted Fourier transform associated to $H$ (in the sense of Sect. 3.3 in \cite{boussaid}).
\begin{equation}\label{distortedfourier}
\mathcal{F}_{V,\pm}(g)(\xi )= \frac 1{(2\pi )^{\frac{3}{2}}}\int _{\R ^3}\psi^{\pm}_{V}(x,\xi)
^*g(x) dx,
\end{equation}
is a bounded linear operator from $\mathcal H_c [0]$ in $L^2(\R^3, \C^4)$ and with inverse
$\mathcal{F}^{*}_{V,\pm}$ defined as in Theorem 3.2 of \cite{boussaid}.
Notice that we have also the relation $\mathcal{F}(g)(\xi)=v^{*}(\xi)\mathcal{F}_{0,\pm}(g)(\xi).$
Motivated by this we state the following lemma

\begin{lemma}
\label{plemeljDir}
Let be $\lambda\in \R\backslash[-\mathscr{M}, \mathscr{M}],$ then we have the following
representation of the resolvent $R_H^{\pm }( \lambda)$ of the perturbed Dirac operator $H$ defined as in Lemma \ref{lem:smooth1},
\begin{equation}\label{eq:plem}   \begin{aligned} &
  R_H^{\pm }( \lambda) = P.V. \frac{1}{H-\lambda}\pm \im \pi \delta ( H-\lambda),
\end{aligned}
\end{equation}
characterized by
\begin{align}\label{eq:delta}
  \pi \im \langle\delta(H-\lambda)f, f^*\rangle=\frac 12\langle R_H^{+ }( \lambda)-R_H^{-}( \lambda)f,f^*\rangle &
\nonumber \\
 =\frac {\pi \im |\lambda| }{\sqrt{\lambda^2-\mathscr{M}^2}} \int _{|\xi|=\sqrt{\lambda^2-\mathscr{M}^2}}|\mathcal{F}_{V,+}(f)|^2 d\xi&,
\end{align}
and
\begin{align}\label{eq:pv}
 \langle P.V. \frac{1}{H-\lambda}f, f^*\rangle=\frac 12\langle R_H^{+ }( \lambda)+R_H^{-}( \lambda)f,f^*\rangle &
\nonumber \\
 =\lim_{\epsilon \searrow 0} \int _{\left | |\xi|-\sqrt{\lambda^2-\mathscr{M}^2}\right|\geq \epsilon}\frac{L(\xi)\beta+\lambda I_{\C^4}}{|\xi|^2+\mathscr{M}^2-\lambda^2}
 |\mathcal{F}_{V,+}(f)|^2 d\xi&,
\end{align}
for any function $f \in \Sc(\R^3, \C^4)\cap \mathcal H [0].$
\end{lemma}

\proof
We deal with the proof of the formulas \eqref{eq:delta} and \eqref{eq:pv} for $R_H^{+ }( \lambda)$ because the one for
$R_H^{- }( \lambda)$ is similar. Select a $f \in \Sc(\R^3, \C^4),$ then by transposing the arguments of \cite{Agmon} and following \cite{Yamada73} one can see that, for any $z\in \C\backslash\sigma(H),$
the following identity is fulfilled
\begin{equation}\label{eq:diag0}
\mathcal{F}(R_{H}(z)f)(\xi)=\mathcal{F}\left(\frac 1{H-z}f\right)(\xi)=v(\xi)\frac{L(\xi)\beta+zI_{\C^4}}{|\xi|^2+\mathscr{M}^2-z^2} f (\xi,z),
\end{equation}
where we set
\begin{equation*}
\frac{(L(\xi)\beta+zI_{\C^4})}{|\xi|^2+\mathscr{M}^2-z^2}=
 \begin{pmatrix}
 \frac {1}{L(\xi)-z }I _{\C^2} &
0 \\
0 & -  \frac {1}{L(\xi)+z } I _{\C^2}
 \end{pmatrix},
\end{equation*}
(see once again \cite{Thaller}) and with the vector valued function $f(\xi,z)$ having the form
\begin{align}\label{eq:general}
 f (\xi,z)=\frac 1{(2\pi )^{\frac{3}{2}}}\int _{\R ^3}(\psi_{0}(x,\xi)-R_{H}(z^*)V\psi_{0}(x,\xi))^*f(x)dx.
\end{align}

Moreover by a use of  $R_{ H} (z ) - R_{ H} ( {z}^*  ) =2\im R_{ H} (z )   R_{ H} ( {z}^* )\Im z$ and $R_{ H} ( {z}^*  )=(R_{ H} (z ))^*$ for any $z\in \C \backslash \R,$ in connection with the Parseval identity and with the fact that $v(\xi)v^{*}(\xi)=v^{*}(\xi)v(\xi)=I_{\C^4},$  one gets
\begin{align}\label{eq:pars}
\frac12 \langle [R_{H}(z)-R_{H}(z^*)]f, f^*\rangle=\rmi \int _{\R ^3}\Im\frac{(\lambda(\xi)\beta+zI_{\C^4})}{|\xi|^2+\mathscr{M}^2-z^2}  f (\xi,z) f (\xi,z)^* d\xi,
\end{align}

with the matrix

\begin{equation*}
\Im\frac{(L(\xi)\beta+zI_{\C^4})}{|\xi|^2+\mathscr{M}^2-z^2} =
 \begin{pmatrix}
 \frac {\Im z}{(L(\xi)-z)(L(\xi)-z^*) }I _{\C^2} &
0 \\
0 &  \frac {\Im z}{(L(\xi)+z)(L(\xi)+z^*)  } I _{\C^2}
 \end{pmatrix}.
\end{equation*}
Pick   now $z=\lambda+\rmi \varepsilon, $ then we are allowed by the trace Lemma \ref{lem:smooth1} to take the limit $\varepsilon \searrow 0$ (see also \cite{Agmon}).  Combining this step with an application of the Plemelj formula $\frac{1}{x\mp
i0}=PV\frac{1}{x}\pm\im \pi \delta (x),$ we obtain from \eqref{eq:pars} the identity
\begin{align}\label{eq:pars1}
\frac12 \langle [R^+_{H}(\lambda)-R^-_{H}(\lambda)]f, f^*\rangle=\pi \rmi  \int _{\R ^3}\varXi(L(\xi), \lambda) \mathcal{F}_{V,+}(f)\mathcal{F}_{V,+}(f)^* d\xi,
\end{align}
with
\begin{equation*}
\varXi(L(\xi), \lambda) =
 \begin{pmatrix}
\delta(L(\xi)-\lambda)I _{\C^2} &
0 \\
0 &  \delta(-L(\xi)-\lambda)I _{\C^2}
 \end{pmatrix}.
\end{equation*}
At this point, an application of the identities (in $\Sc'(\R^3, \C^4)$, see for example \cite{GelShi}, Chap. II, Sec. 2.5 or Chap III for a more general theory)
\begin{align}
\label{eq:deltarel1}
\delta(L(\xi)-\lambda)=\frac{\lambda}{\sqrt{\lambda^2-\mathscr{M}^2}}\delta(|\xi|-\sqrt{\lambda^2-\mathscr{M}^2}), \ \ \ \text{for} \  \ \ \lambda>\mathscr{M},
\\
\label{eq:deltarel2}
\delta(-L(\xi)-\lambda)=\frac{-\lambda}{\sqrt{\lambda^2-\mathscr{M}^2}}\delta(|\xi|-\sqrt{\lambda^2-\mathscr{M}^2}), \ \ \ \text{for} \  \ \ \lambda<-\mathscr{M},
\end{align}
shows that the r.h.s. of   identity  \eqref{eq:pars1} is equal to
\begin{numcases} {}
 \frac{\pi \rmi \lambda}{\sqrt{\lambda^2-\mathscr{M}^2}} \int _{|\xi|=\sqrt{\lambda^2-\mathscr{M}^2}}\sum_{i=1}^2|\mathcal{F}^i_{V,+}(f)|^2 d\xi, \ \ \ \ \ \ \text{for} \ \ \ \lambda>\mathscr{M},\label{eq:delta3p}\\
 \frac{-\pi \rmi \lambda}{\sqrt{\lambda^2-\mathscr{M}^2}} \int _{|\xi|=\sqrt{\lambda^2-\mathscr{M}^2}}\sum_{i=3}^4|\mathcal{F}^i_{V,+}(f)|^2 d\xi,\ \ \ \ \ \ \text{for} \ \ \ \lambda<-\mathscr{M},\label{eq:delta3n}
\end{numcases}
which in turn implies the identity \eqref{eq:delta}. A similar discussion (actually easier) yields

\begin{align}
\frac 12\langle R_H^{+ }( \lambda)+R_H^{-}( \lambda)f,f^*\rangle
 =P.V. \int \frac{L(\xi)\beta+\lambda I_{\C^4}}{|\xi|^2+\mathscr{M}^2-\lambda^2}
 |\mathcal{F}_{V,+}(f)|^2 d\xi &,
 \nonumber
\end{align}
that is the identity \eqref{eq:pv}. This completes the proof of the lemma.

\qed

\begin{remark}
All the convergence arguments are well defined because the Lemma \ref{lem:smooth1}. Moreover
in the proof  we used functions in  $f\in \Sc(\R^3, \C^4).$ One can easily extend   to $ f\in H^{k, s}(\R ^3, \C^4)\cap \mathcal{H}[0]  $    for $k\ge 0$ and $s>1/2$   (which implies $\mathcal{F}_{V,\pm}f\in H^{  s} _{loc}(\R ^3, \C^4) $  so that
   the restriction    to    spheres  makes sense)     by a density argument.
\end{remark}

\section*{Acknowledgments}  The authors  were   funded    by the grant FIRB 2012 (Dinamiche Dispersive)  from MIUR,  the Italian Ministry of Education,
University and Research. S.C. was supported also by the grant FIRB 2013 from the University of Trieste.
The authors  wish to thank Professor Masaya Maeda for
help with the proof of  Proposition \ref{prop:EnExp}.

\bibliographystyle{amsplain}

\begin{thebibliography}{CP03}

\bibitem {Agmon}
S.Agmon, {\em  Spectral properties of Schr\"odinger operators and scattering theory},
  An. Sc. N. Pisa   {  2} (1975),151--218.



\bibitem {bambusicuccagna}
  D.Bambusi, S.Cuccagna, {\em On dispersion of
small energy solutions of the nonlinear Klein Gordon equation with a
potential},   Amer. Math. Jour.    {    133} (2011), 1421--1468.

\bibitem {BejHerr} I.Bejenaru, S.Herr, {\em The cubic Dirac equation: Small initial data in $H^1(\R ^3)$}, Com. Math. Phys. {  335} (2014), 1--40.

\bibitem {com}  A.Berkolaiko, A.Comech, {\em On spectral stability of solitary waves of nonlinear Dirac equation on a
line}, Mathematical Modelling of Natural Phenomena {  7} (2017), 13--31.


\bibitem {BerthierGeorgescu}
A.M.Berthier, V.Georgescu, {\em On the point spectrum of {D}irac operators},
  J. Fun. Anal.   {    71} (2011), 309--338.



\bibitem{boussaid}
  N.Boussaid,   {\em  Stable directions for small nonlinear {D}irac standing waves},    Com. Math. Phys.  {  268 } (2006),   757--817.

 \bibitem{Boussaid2}
N.Boussaid,
 {\em  On the asymptotic stability of small nonlinear {D}irac standing waves
  in a resonant case},
  SIAM J. Math. Anal.,     {  40 } (2008),  1621--1670.

\bibitem{boussaidcuccagna}
  N.Boussaid, S.Cuccagna,  {\em  On stability of standing waves of nonlinear Dirac equations},    Comm. in Partial Diff. Eq.  {  37 } (2012),   1001--1056.


\bibitem{Fanelli2}
  N.Boussaid, P.D'Ancona, L.Fanelli,  {\em  Virial identity and weak dispersion for the magnetic Dirac equation},    J. Math. Pures Appl.  { 95 } (2011),   137--150.



\bibitem{BP2}
V.Buslaev, G.Perelman, {\em On the stability of solitary waves for
nonlinear Schr\"odinger equations},   Nonlinear evolution
equations, editor N.N. Uraltseva, Transl. Ser. 2, 164, Amer. Math.
Soc.,
   pp.  75--98, Amer. Math. Soc., Providence (1995).

\bibitem{Candy}
T.Candy,  {\em   Global existence for an $L^2$ critical nonlinear Dirac equation in one dimension},
  Adv. Diff. Eq.  {  16 }  (2011), 643--666.


  \bibitem{Cacciafesta}
F.Cacciafesta, P.D'Ancona {\em   Endpoint estimates and global existence for the nonlinear Dirac equation with potential},
  Jour. Diff. Eq.  {  254 }  (2013), 2233--2260.

\bibitem {CL}   T.Cazenave, P.L.Lions, {\em Orbital stability of
standing waves for  nonlinear Schr\"odinger equations},  Comm.
Math. Phys.  85  (1982),  549--561.

\bibitem{ChristKiselev} M.Chirst, A.Kiselev, {\em Maximal functions associated to filtrations}, J.
Funct. Anal. 179(2) (2001), 409--425.


\bibitem{Cohen} C.Cohen-Tannoudji,  J.Dupont-Roc, G.Grynberg, {\em Atom-Photon Interactions:
Basic Processes and Applications}, Wiley, New York, 1992.



\bibitem{Comech}
A.Comech, T.V.Phan, A.Stefanov,  {\em   Asymptotic stability of solitary waves in generalized Gross--Neveu model},
  arXiv:1407.0606v2,  to appear in Annales de l'Institute H. Poincar\'e (Analyse non lin.)  .

  


\bibitem{Jensen}
  H.D.Cornean, A.Jensen, G.Nenciu,  {\em Metastable States When the Fermi Golden Rule Constant
Vanishes}, Comm. Math. Physics
  {  334}  (2015), 1189--1218.

	
 			
\bibitem{Cu2}
  S.Cuccagna,  {\em The Hamiltonian structure of the nonlinear
Schr\"odinger equation and   the  asymptotic stability of its
ground states}, Comm. Math. Physics
  {   305}  (2011),  279--331.
	
\bibitem{Cu3} S.Cuccagna, {\em On scattering of small energy solutions of non autonomous hamiltonian nonlinear Schr\"odinger equations}, J. Differential Equations {  250} (2011), no. 5, 2347--2371.	
		

\bibitem{Cu0}
  S.Cuccagna,  {\em On the Darboux and Birkhoff   steps in the asymptotic
  stability of    solitons},  Rend. Istit. Mat. Univ. Trieste   {   44} (2012), 197--257.


\bibitem{CM1} S.Cuccagna, M.Maeda, {\em On   small energy stabilization    in the NLS with a trapping potential}, 	Analysis  \& PDE  Vol. 8 (2015), No. 6, 1289--1349.
    
    \bibitem{CMT} S.Cuccagna, M.Maeda, T.V.Phan {\em On  small energy stabilization in the NLKG with a trapping potential}, 	 arXiv:1511.04672.
    
    

\bibitem{Fanelli1}
  P.D'Ancona, L.Fanelli,  {\em Decay estimates for the wave and Dirac equations with a magnetic potential},  Comm. Pure Appl. Math.   {   60} (2007), 357--392.





\bibitem{Dias}
J.P.Dias, M.Figueira,
{\em
Time decay for the solutions of a nonlinear Dirac equation in one space dimension
},
Ric. Mat.   (1986) {  35},   309--316.

\bibitem{Du} J.Duoandikoetxea, {\em Fourier Analysis},
American Mathematical Society Books Series, Providence, 2000.

\bibitem{EGS}
 M.B.Erdo\u{g}an, M.Goldberg,   W.Schlag,  {\em Strichartz and smoothing estimates for {S}chr{\"o}dinger operators
  with almost critical magnetic potentials in three and higher dimensions}, Forum Math (2009) {  21}, 687--722.


\bibitem{Escobedo}
 M.Escobedo, L.Vega,  {\em A semilinear Dirac equation in $H^s(\R ^3)$ for $s>1$}, SIAM J. Math. Anal (1997) {  28}, 338--362.





\bibitem{GelShi}
A.Gelfand, G.Shilov {\em Generalized Functions}, Vol.1, Academic Press, 1964.

\bibitem{GSS1}   M.Grillakis, J.Shatah, W.Strauss, {\em Stability
of solitary waves in the presence of symmetries, I },  Jour. Funct.
An.   {74} (1987),   160--197.


\bibitem{GP}
S.Gustafson,    T.V.Phan,  {\em Stable directions for degenerate excited states of nonlinear  Schr\"odinger equations}, SIAM J. Math. Anal.    43  (2011) , 1716--1758.

\bibitem{GNT}
S.Gustafson,   K.Nakanishi, T.P.Tsai,  {\em Asymptotic stability and completeness in the energy space for nonlinear Schr\"odinger equations with small solitary waves}, Int. Math. Res. Not.  2004 (2004) no.  {\bf  66}, 3559--3584

\bibitem{Kevrekidis} P.G.Kevrekidis, D.E.Pelinovsky, A.Saxena, {\em When Does Linear Stability Not Exclude Nonlinear Instability ?}, Physical Review Letters 114, 214101 (6 pages) (2015).







\bibitem{Hofer}
H.Hofer, E.Zehnder,
 {\em Symplectic invariants and Hamiltonian dynamics},
  Birkh\"auser Verlag, Basel, 1994.

\bibitem{MacNakOz} S.Machihara, K.Nakanishi, T.Ozawa, {\em Small global solutions and the nonrelativistic limit for the nonlinear Dirac equation}, Rev. Mat. Iberoamericana   19  (2003), no. 1, 179--194.




\bibitem{Maeda}  M.Maeda, {\em Existence and asymptotic stability of quasi-periodic solution of discrete NLS with potential in $\Z$},   	 arXiv:1412.3213.


\bibitem{M1}  T.Mizumachi, {\em Asymptotic stability of small
solitons to 1D NLS with potential}, Jour. of Math.   Kyoto
University,    48  (2008),   471--497.






\bibitem {NPT} K.Nakanishi, T.V.Phan, T.P.Tsai, {\em Small solutions of nonlinear Schr\"odinger equations near first excited states}, Jour. Funct. Analysis  {\bf  263}  (2012), 703--781.

\bibitem {pego} R.L.Pego, M.I.Weinstein,
{\em Convective Linear Stability of Solitary Waves for Boussinesq Equations}, Studies in Appl. Math.     99   (1997), 311--375.


\bibitem {Pelinovsky}
D.Pelinovsky,  {\em
Survey on global existence in the nonlinear Dirac equations in one spatial dimension},  Harmonic analysis and nonlinear partial differential equations, 37--50,
  B26, Res. Inst. Math. Sci. (RIMS), Kyoto, 2011.



\bibitem {PelinovskyStefanov}
D.Pelinovsky,  A.Stefanov,  {\em Asymptotic stability of small gap solitons in the nonlinear Dirac
  equations},   J. Math. Phys.    53   (2012),   073705, 27 pp.

\bibitem {pel1}
D.Pelinovsky,   Y.Shimabukuro,  {\em $L^2$ orbital stability of Dirac solitons},     	  Letters in Mathematical Physics 104 (2014), 21--41.



\bibitem{PiW}
C.A.Pillet, C.E.Wayne, {\em Invariant manifolds for a class of
dispersive, Hamiltonian partial differential equations} J. Diff. Eq.
  141 (1997), pp. 310--326.

\bibitem {RW} H.A.Rose, M.I.Weinstein, {\em  On the bound states of the nonlinear Schr\"odinger equation with a linear potential\/}, Physica D, 30
(1988), pp. 207--218

\bibitem {SmithSogge}
H.Smith, C.Sogge, {\em
  {S}trichartz estimates for nontrapping perturbations of the
  {L}aplacian},  Com. Part. Diff. Eq.  {  25}   (2000),2171--2183.

\bibitem {SW1} A.Soffer, M.I.Weinstein, {\em  Multichannel nonlinear
scattering for nonintegrable equations \/}, Comm. Math. Phys., 133
(1990), pp. 116--146






\bibitem {SW2}
A.Soffer, M.I.Weinstein, {\em  Multichannel nonlinear scattering II.
The case of anisotropic potentials and data \/},  J. Diff. Eq., 98
 (1992), pp.
  376--390.





\bibitem{SW4}    A.Soffer, M.I.Weinstein,
{\em Selection of the ground state for nonlinear Schr\"odinger
equations}, Rev. Math. Phys.      {  16}  (2004),   977--1071.



\bibitem{SW3}
A.Soffer, M.I.Weinstein, {\em Resonances, radiation damping and
instability in Hamiltonian nonlinear wave equations},  Invent.
Math.   {   136}
 (1999),
  9--74.


\bibitem {So}
C.D.Sogge,
  {\em {Lectures on nonlinear wave equations}},
  International Press Boston, 1995.

\bibitem {StraussVazquez}
W.Strauss, L.V{\'a}zquez,
{\em Stability under dilations of nonlinear spinor fields},
  Phys. Rev. D.,   {   34} (1986),641--643.




\bibitem {taylor}
M.E.Taylor, {\em Partial Differential Equations}, volumes 115--117 of App. Math. Sci., Springer, New York (1996).

\bibitem{Thaller}
B.Thaller.
\newblock {\em The {D}irac equation}.
\newblock Texts and Monographs in Physics. Springer-Verlag, Berlin, 1992.

\bibitem {TY1}
  T.P.Tsai, H.T.Yau, {\em Asymptotic dynamics of nonlinear
Schr\"odinger equations: resonance dominated and radiation dominated
solutions}, Comm. Pure Appl. Math.  {  55}  (2002),   153--216.

\bibitem {TY2}
  T.P.Tsai, H.T.Yau, {\em Relaxation of excited states in
nonlinear Schr\"odinger equations}, Int. Math. Res. Not.  {  31}
(2002),   1629--1673.

\bibitem {TY3}
{  T.P.Tsai, H.T.Yau}, {\em Classification of asymptotic profiles
for nonlinear Schr\"odinger equations with small initial data}, Adv.
Theor. Math. Phys.  {  6} (2002),   107--139.


\bibitem {TY4}
{  T.P.Tsai, H.T.Yau}, {\em Stable directions for
excited states of nonlinear Schr\"odinger equations},   Comm.
P.D.E.     {  27} (2002),  2363--2402.




\bibitem{W1}  M.I.Weinstein, {\em Lyapunov stability of ground
states of nonlinear dispersive equations},  Comm. Pure Appl. Math.
 39  (1986),    51--68.

\bibitem{Yamada73} O.Yamada, {\em On the principle of limiting absorption for the Dirac operators}, Publ. Res. Inst. Math. Sci. 8 (1972/73), 557-577. 		

	\bibitem {Zhang}
{  Y.Zhang}, {\em Global strong solution to a nonlinear Dirac type equation in one dimension},   Nonlin. Analysis: Theory, Methods \& Applications     {  80} (2013),  150--155.	
		









	\end{thebibliography}

Department of Mathematics and Geosciences,  University
of Trieste, via Valerio  12/1,  Trieste 34127,  Italy

{\it E-mail Address}: {\tt scuccagna@units.it}
\\

Department of Mathematics, University of Pisa, Largo Bruno Pontecorvo 5,
Pisa 56127,  Italy

{\it E-mail Address}: {\tt tarulli@mail.dm.unipi.it}

\end{document}